\documentclass[leqno, 11pt, a4paper]{amsart}

\usepackage{amsmath}
\usepackage{amssymb, latexsym, slashed, amscd}
 \usepackage[all]{xy}
 \usepackage{amsthm}
\usepackage{amsfonts}
\usepackage{mathrsfs}
\usepackage{enumerate}
\usepackage{color}
  \usepackage{comment}

 \newtheorem{definition}{Definition}[section]
 \newtheorem{theorem}[definition]{Theorem}
 \newtheorem{lemma}[definition]{Lemma}
 \newtheorem{proposition}[definition]{Proposition}
 \newtheorem{corollary}[definition]{Corollary}

\newtheorem{conjecture}[definition]{Conjecture}

 \newtheorem*{theorem*}{Theorem}
\newtheorem*{proposition*}{Proposition}
\newtheorem*{lemma*}{Lemma}

 \theoremstyle{remark}
 \newtheorem{example}[definition]{Example}
 \newtheorem{remark}[definition]{Remark}

\newcommand{\op}[1]{\operatorname{#1}}

\newcommand{\Tr}{\ensuremath{\op{Tr}}}

\newcommand{\Tor}{\ensuremath{\op{Tor}}}

\def\XXint#1#2#3{{\setbox0=\hbox{$#1{#2#3}{\int}$}
\vcenter{\hbox{$#2#3$}}\kern-.5\wd0}}

\newcommand{\ind}{\op{ind}}
\newcommand{\Ch}{\op{Ch}}

\newcommand{\Ind}{\op{Ind}}

\newcommand{\C}{\ensuremath{\mathbb{C}}}

\newcommand{\Q}{\ensuremath{\mathbb{Q}}} 
\newcommand{\R}{\ensuremath{\mathbb{R}}} 
 
\newcommand{\Z}{\ensuremath{\mathbb{Z}}}

\newcommand{\Ca}[1]{\ensuremath{\mathcal{#1}}}
\newcommand{\cA}{\Ca{A}}

\newcommand{\cK}{\ensuremath{\mathcal{K}}}

\newcommand{\cR}{\Ca{R}}

\newcommand{\coker}{\op{coker}} 
\newcommand{\ad}{\op{ad}}
\newcommand{\Hom}{\op{Hom}}
\newcommand{\Pic}{\op{Pic}}

\newcommand{\tw}{\op{tw}}

\newcommand{\ch}{\op{ch}}

\newcommand{\ba}{\begin{eqnarray}}
   \newcommand{\na}{\end{eqnarray}}

\numberwithin{equation}{section}

\begin{document}

\title{Twisted Donaldson Invariants}

\author{Tsuyoshi Kato}
\thanks{T.Kato is supported by JSPS KAKENHI Grant Number JP17H02841} 
\address{Department of Mathematics, Kyoto University, Japan}
 \email{tkato@math.kyoto-u.ac.jp}

\author{Hirofumi Sasahira}
\thanks{H. Sasahira is supported by JSPS KAKENHI Grant Number JP19K03493} 
\address{Faculty of Mathematics, Kyushu University, Japan}
\email{hsasahira@math.kyushu-u.ac.jp}

\author{Hang Wang}
\thanks{H. Wang is supported by grants NSFC-11801178 and Shanghai Rising-Star
Program 19QA1403200.} 
 \address{School of Mathematical Sciences, East China Normal University, China}
 \email{wanghang@math.ecnu.edu.cn}

%
%
%

%

 \maketitle

  \begin{abstract}
 Fundamental group of a manifold gives a
  deep effect on its underlying smooth structure.
 In this paper we introduce a new variant of the Donaldson invariant in Yang-Mills gauge theory
 from twisting by the Picard group
of a $4$-manifold   in the case when  the fundamental group is free abelian.
We then generalize it to  the general case of fundamental groups by use of
the framework of non commutative geometry.
We also verify that our invariant distinguishes smooth structures between some homeomorphic $4$-manifolds.
  \end{abstract}

\section{Introduction}

Gauge theory and non commutative geometry are both  the central branches of
the Atiyah-Singer index theory from the view point of study on
smooth structure of manifolds.
 Both theories gave serious developments in differential topology.
In Yang-Mills gauge theory, the ASD moduli space is the core object which possesses deep
topological information on the underlying 4-manifolds.
The Novikov conjecture on homotopy invariance of higher signatures has been developed extensively
by applying non commutative geometry. Lusztig's approach to the Novikov conjecture
was quite influential, and is a basis of  successive developments of the subject \cite{Lusztig}.
This arose from study of smooth structure in high dimension, where
 fundamental group plays a central role  in surgery theory.

  It would be quite natural to try to combine both theories by
introducing a unified framework to analyse smooth structure on 4-manifolds.
 In this paper we start constructing a twisted Donaldson invariant when the fundamental group of the underlying 4-manifold is free abelian, by combinig  Lusztig's method of use of families of flat line bundles.
 Following Connes-Moscovici, we introduce non commutative geometry framework to generalize
 our construction to produce the twisted Donaldson invariant for non commutative groups.
See \cite{kato} for a related work on Seiberg-Witten and Bauer-Furuta invariant.

Let $E \to X$ be an $SU(2)$ vector bundle and ${\mathfrak B}^*(E)$ be the set of the gauge equivalent classes of irreducible connections on $E$. We have the universal $SU(2)$-vector bundle
$${\mathbb E} \to  X \times {\frak B}^*(E)$$
which gives a family of bundles with connections $E_{[A]} = {\mathbb E}|_{X \times \{ [A ] \}}$ on $X$
and  which consists of a  structural framework in Yang-Mills gauge theory. (The precise construction
requires more technical cares. See Section \ref{sec:family.det.line}).

Let $\Gamma$ be a finitely generated group, and fix a homomorphism ${\bf f} : \pi_1 (X) \rightarrow \Gamma = \pi_1(B\Gamma)$ which is represented by a continuous map
$f: X \to B\Gamma$. Typically
   $\Gamma = \pi_1(X)$ and $f:X \rightarrow B\Gamma$ is the classifying map of the universal cover $\widetilde{X} \rightarrow X$.   Let  $\rho:\pi_1	(B\Gamma) \rightarrow U(1)$ be a representation.  Suppose that $\Gamma$ is isomorphic to a free abelian group of rank $m$. Then $\rho$ presents a trivial line bundle $L_{\rho}$ with a $U(1)$ flat connection on $B \Gamma (\cong T^m)$.
The set of $U (1)$ representations, denoted by $\Pic (B\Gamma)$,
 is the moduli space of flat $U(1)$ connections.
 All the line bundles $L_{\rho}$ together form a complex line bundle
$${\mathbb L} \to B\Gamma \times \Pic (B\Gamma).$$

Combining these two objects,
in this paper we  introduce deformation of the universal bundle from the view point of the fundamental group
of an underlying four manifold.
Suppose again that $\Gamma$ is isomorphic to a free abelian group of rank $m$.
A connection $A$ on $E$ extends to a connection $A_{\rho}$ on $E \otimes f^*L_{\rho}$,
passing through the homomorphism
${\bf f } : \pi_1(X) \to \Gamma = \pi_1(B \Gamma)$ induced by $f$.
If one considers all the irreducible connections in the same way,
there is a transform
$$ T_{\rho} : {\frak B}^*(E) \to {\frak B}^*(E \otimes f^*L_{\rho})$$
from
the set of irreducible connections on $E$ to the one
 on $E \otimes  f^* L_{\rho}$.
We obtain the full map if we consider all the representations
$$ T : {\frak B}^*(E)  \times \Pic(B\Gamma) \to \bigsqcup_{\rho} \ {\frak B}^*(E \otimes f^*L_{\rho}).$$
Thus ${\frak B}^*(E) \times \Pic(B\Gamma)$ parametrizes connections on the  bundles $\{ E \otimes  f^* L_{\rho} \}_{\rho}$.
Consider the following bundle
\[
          {\mathbb E} \boxtimes_{X} (f \times id)^* {\mathbb L} \rightarrow X \times {\frak B}^*(E) \times \Pic(B\Gamma)
\]
which consists of  the family of bundles with connections $({\mathbb E}|_{X \times \{ [A] \}}) \otimes f^* L_{\rho}$ on $X$ parametrized by ${\frak B}^*(E) \times \Pic(B\Gamma)$.
Here $\boxtimes_{X}$ denotes
 the fibered product over $X$.  For our construction of twisted Donaldson invariants,
we need a family of bundles with connections on $X \times B\Gamma$.

 \begin{definition}
 Let $X$ be a compact and spin $4$-manifold and suppose that $c_2(E)$ is odd.
The twisted universal bundle is defined by
    $$ {\mathbb E}^{\tw}_{\bf f} :=  {\mathbb E} \boxtimes_{X} (f \times id)^* {\mathbb L} \boxtimes_{\Pic} {\mathbb L} \rightarrow
                                        X  \times B\Gamma \times {\frak B}^* (E) \times \Pic(B\Gamma)   $$
where $\boxtimes_{\Pic} $ is the fibered product
with respect to the Picard group.
 \end{definition}

The conditions that $X$ is spin and $c_2(E)$
 is odd are necessary for the existence of the universal bundle ${\mathbb E}$. See Lemma \ref{lem universal bundle}.

In Donaldson theory, a homomorphism
\begin{equation} \label{eq mu map}
\mu: H_2(X; {\mathbb Z}) \to H^2({\frak B}^*(E); {\mathbb Q})
\end{equation}
is given by the first Chern class of the determinant bundle of the index of the family
of the twisted Cauchy-Riemann operators on an embedded surface $\Sigma$  in $X$
parametrized by ${\frak B}^*(E)$.

Let $D: \Gamma(E) \to \Gamma(F)$ be a differential  operator over $X$, and
 $\rho: \pi_1(X) \to U(1)$ be a representation. Then one can twist $D$ with $\rho$ so that it gives another operator
$D_{\rho} : \Gamma( E \otimes L_{\rho}) \to \Gamma( F \otimes L_{\rho})$.
This gives a family of operators
$${\mathbb D} := \{ D_{\rho} \}_{\rho \in \Pic(B \Gamma)} $$
 parametrized
by the Picard torus.
Such family of elliptic operators play the important  role in Lusztig's construction.

Now take a compact submanifold $M \subset X \times B\Gamma$.
Suppose the dimension of $M$ is even and $M$ is  spin with the complex spinor bundle ${\bf S} = {\bf S}^+  \oplus {\bf S}^-$
and the Dirac operator $D: \Gamma({\bf S}^+ ) \to  \Gamma( {\bf S}^-)$.
A representation $\rho : \pi_1(B\Gamma) \rightarrow U(1)$
gives the flat line bundles $L_{\rho}$ and $f^* L_{\rho}$ on $B\Gamma$ and $X$ respectively.
For each element $[A] \in {\frak B}^*(E)$, one can twist  $D_A$ with the flat line bundles $L_{\rho}, f^* L_{\rho}$,
 and obtain the Dirac operator with coefficients
$$D_{A, \rho}:
 \Gamma(({\bf S}^+  \otimes E \otimes f^* L_{\rho}) \boxtimes L_{\rho}) \to  \Gamma( ({\bf S}^- \otimes E \otimes f^* L_{\rho})  \boxtimes L_{\rho}).$$
This gives the family of  the full twist of the Dirac operators
$${\mathbb D}^{\tw}_M = \{ D_{A, \rho}
\}_{ ([A] , \rho)  \in  {\frak B}^*(E) \times \Pic( B\Gamma ) }$$
which is a basic object in our construction of the twisted Donaldson invariants.

Roughly we define
the {\em twisted} $\mu$ {\em map}  by
$$\mu^{\tw}_{\bf f}([M]) = \ch (\Ind {\mathbb D}^{\tw}_M) \in H^*({\frak B}^*(E) \times \Pic(B\Gamma) ; {\mathbb Q})$$
where $\Ind {\mathbb D}^{\tw}_M$ is the index bundle over ${\frak B}^*(E) \times \Pic (B\Gamma)$.
More precisely we have to take care of non compactness of  ${\frak B}^*(E)$.
 See Section \ref{sec:family.det.line} for the precise construction.

In order to see that
 our twisted Donaldson invariants are non trivial,
we induce
 two  properties on topology of four manifolds.
 These were suggested by Professor K.Fukaya,
 to whom we are thankful.

Let $Y, Y_1,Y_2$ all be spin $4$-manifolds with $b^+ >1$.
We will give concrete examples which satisfy the above conditions.

\begin{theorem}\label{appl}
Let $m \geq 0$ be a non negative integer.

$(1)$ If moreover $Y$ is simply connected,  algebraic, then
 $X =Y\#^m (S^1 \times S^3)$ cannot admit  connected sum decompositions
$X = X_1 \# X_2$ with $b^+(X_i) >0$.

$(2)$
Suppose that  the Donaldson invariant over $Y_1$ vanishes but the one over
 $Y_2$ does not. Moreover assume that they are mutually homeomorphic
 so that the pair is exotic.
Then the corresponding pairs $Y_1\#^m (S^1 \times S^3)
$ and $Y_2\#^m (S^1 \times S^3)$ are  also  exotic.
\end{theorem}

To verify these results, we
use  our computations of the twisted Donaldson invariant defined as below.
Let $X$ be a compact and spin  $4$-manifold with $b^+(X) > 1$.  Put
\begin{equation}  \label{eq A(X)}
    A(X)  =
     \op{Sym}(H_0(X;\Z) \oplus H_2(X;\Z) ) \ \otimes \  \Lambda (H_1(X;\Z) \oplus H_3(X;\Z) ).
\end{equation}

In the rank $1$ case $\Gamma \cong \Z$ ($B\Gamma \cong S^1$),
for each $r \in \{1,2,3,4\}$,
the twisted
Donaldson polynomial is formally given by the following
\[
  \begin{split}
& \Psi^{\tw, (r)}_{X, {\bf f}} : A(X) \otimes H_{odd}(X ; {\mathbb Z})
\to  H^*(\Pic( B\Gamma ) ; {\mathbb Q})  \cong H_*( \Gamma  ; {\mathbb Q} ) , \\
 & \Psi_{X, {\bf f}}^{\tw, (r)}([M_1], \dots, [M_d];  [M]) = \int_{{\frak M}(E)}
 \mu([M_1] ) \wedge \dots \wedge \mu([M_d]) \wedge \mu^{\tw}_{\bf f}([M])
  \end{split}
 \]
if there exists an  $SU(2)$-vector bundle $E$ on $X$ with $c_2(E)$ odd and with $\dim {\frak M}(E) = \sum_{j=1}^d(4-\dim M_j)+r$.
We put  $\Psi_{X, {\bf f}}^{\tw, (r)}([M_1], \dots, [M_d]; [M]) =0$, if
  there is no such $SU(2)$ vector bundle.
  Probably it is possible to define $\Psi_{X, {\bf f}}^{\tw, (r)}$ for $r \geq 5$, but we do not discuss in this paper.

Of course the integral above makes sense only if the cohomology class $\mu([M_1]) \cup \cdots \cup \mu([M_d]) \cup \mu^{\tw}_{\bf f}([M])$ has compact support since ${\frak M}(E)$ is not compact. We will overcome this issue by choosing a submanifold in ${\frak B}^*(E)$ dual to the cohomology class which behaves nicely near  the end of ${\frak M}(E)$
as in the usual Donaldson theory.

We will compute  $\Psi_{X}^{\tw, (3)}$ in the case of
the connected sum $X = Y \# S^1 \times S^3$.
Let $\ell$ be a loop in $X$ which represents a generator of $\pi_1(S^1 \times S^3) \cong \Z$.
We will verify the equality
\[
     \Psi_{X, {\bf f}}^{\tw, (3)}([M_1], \dots, [M_d], [\ell])
     =
     -\Psi_{Y}([M_1], \dots, [M_d]) \eta.
\]
Here $\Psi_Y$ is the Donaldson invariant over $Y$, and
$\eta \in H^1(\Pic(B\Z);\Z) \cong H_1(B\Gamma;\Z)  \cong \Z$
 is a generator.
See Example \ref{ex Y S^1 times S^3}.

The above construction naturally leads to an extension to higher rank case of free abelian groups.  We verify non triviality of such extension. In particular we
 reproduce Theorem \ref{appl} under the additional assumption of simple type,
by use of the extension of the twisted Donaldson invariant.

Based on the above constructions, we
 generalize the twisted Donaldson theory to
 the case of non commutative fundamental groups.
 Immediately two difficulties arise, which   prevent from a straightforward extension.
 One is that  $B\Gamma$ is not a manifold in general.
 The other is that Picard group becomes much harder to treat  in general.
 We shall use the spin bordism group $\Omega_*^{spin}(B \Gamma)$
 instead of $H_*(B \Gamma; \Z)$, which solves the former  difficulty.

 The latter difficulty  touches the central idea of non commutative geometry.
 Let us briefly explain it.
 It is well known as the Swan's theorem that topological $K$ group of a compact CW complex
 is naturally identified with the operator $K$ group of continuous functions over the space.
Non commutative algebras  are applied to the latter group, and
  this  plays a role of a basic  bridge between   non commutative algebras and algebraic topology.
   Associated with the de-Rham cohomology over a space
 is Connes cyclic cohomology $HC^*(\cA)$ for
 a non commutative algebra $\cA$.
The parallel notions have been established such as the Chern character
 $K_*(\cA) \to HP_*(\cA)$.

 Let us describe more details on the passage from commutative to non commutative analysis.
 An important aspect is Fourier transform on tori, which gives the correspondence of  functions between
 a torus and its dual torus. The dual torus parametrizes flat $U(1)$ connections over the torus
  called the Picard torus.
 Such space duality comes from a strong constraint that the group is free abelian.
 For a non commutative group $\Gamma$,
 there are no such space duality, but corresponding to  the space of functions over the dual torus
 is the (reduced) group $C^*$ algebra $C_r^*(\Gamma)$
  of the  group, given  as a completion
 of the complex group ring (see subsection $5.3$).
 An elliptic operator twisted by every flat $U(1)$ connection over the torus produces a family of the elliptic operators parametrized by
 the Picard torus, as we described above.
 The family index gives an element of $K^0(\Pic(B\Gamma))$.
 A non commutative interpretation is the elliptic operator with coefficient by the Mishchenko-Fomenko flat $C^*_r(\Gamma)$ bundle,
 and its index takes a value in $K_0(C^*_r(\Gamma))$.
 In case of $\Gamma \cong \Z^n$, where $B\Gamma$ is the $n$-torus $T^n$, the isomorphisms hold:
 $$K_0(C^*_r(\Gamma)) \cong K_0(C(\Pic(T^n)) \cong K^0(\Pic(T^n)).$$

 Chern-Weil theory presents a construction of the Chern classes
 by smooth differential forms which
take the values in the de-Rham cohomology.
A key to describe its
non commutative counterpart is the existence of
 a dense subalgebra
$C^{\infty}(\Gamma) \subset C^*_r(\Gamma)$
which is closed under holomorphic calculus.
The algebra plays a central role in non commutative de-Rham theory by use of
   cyclic cohomology. Conceptually a canonical choice of this algebra corresponds to the algebra of smooth functions
   on manifolds.

The Chern character
is given by a map
$$\Ch_*: K_*(C_r^*(\Gamma)) \to HP_*(C^{\infty}(\Gamma)).$$

There is a canonical map
$HP^*(C^{\infty}(\Gamma)) \to H^*(\Gamma; \C)$,
and it induces an isomorphism
$$HP^*(C^{\infty}(\Gamma))  \otimes \C \cong  H^*(\Gamma; \C)$$
 under some cohomological conditions on  $ \Gamma$, 
 to be more precise, when $\Gamma$ satisfies the Baum-Connes conjecture and the Chern character $\Ch_*$ is rationally isomorphic.
 Typical examples are free abelian groups.
A weaker condition for $\Gamma$, called \emph{admissiblity}, guarantees  surjectivity of the canonical map
(Proposition  \ref{prop:admsuj}).

 Let $B \subset {\frak M}(E)$ be a compact submanifold.
As a canonical extension of  the twisted $\mu$ map to non commutative case,
we introduce  the following map. Let us fix  a smooth algebra $C^{\infty}(\Gamma)$.
Then
we construct a twisted $\mu$-map
 \[
\mu^{\tw}_{\bf f}:
 \Omega^{Spin}_{*}(X) \otimes\Omega_{*}^{Spin}(B\Gamma)\rightarrow HP_0(
 C^{\infty}(\Gamma)\otimes C^{\infty}( \Gamma) \otimes C^{\infty}(B)).
\]
If moreover the group is admissible,
then we derive its cohomological formula which follows essentially from Connes-Moscovici's index theorem.
Put
\begin{equation}  \label{eq A pi}
   \begin{split}
   & A^{\pi}(X)  = \\
  &   \operatorname{Sym}(H_0(X;\Z) \oplus H_2(X;\Z) ) \ \otimes \  \Lambda \ H_3(X;\Z) \  \times \
     \Pi_{m \in \Z_{\geq 0}} \ \pi_1(X)^{\times m}
  \end{split}
\end{equation}
where we regard  $ \pi_1(X)$ as just a set, rather than group.
The twisted Donaldson map is given by
 \[
 \Psi^{\tw, (r)}_{X, {\bf f}}:
A^{\pi}(X)
\times   ( \Omega^{Spin}_{*}(X) \otimes  \Omega^{Spin}_{*}(B\Gamma))  \rightarrow
HP_*(C^{\infty}(\Gamma) \otimes C^{\infty}(\Gamma))
\]
 for $r =1,2$. In particular we obtain the following map
 \[
 \Psi^{\tw, (r)}_{X, {\bf f}}:
A^{\pi}(X)
\times   ( \Omega^{Spin}_{*}(X) \otimes  \Omega^{Spin}_{*}(B\Gamma) )    \rightarrow
H_*(\Gamma; \Q) \otimes  H_*(\Gamma; \Q)
\]
when $\Gamma$ is admissible and
$HP_*(C^{\infty}(\Gamma))$ is
 finite dimensional.
Note that when $\Gamma$ is abelian or more generally, $\Gamma$ admits a cohomological formula in the sense of Theorem~\ref{thm.coho.formula}, $\pi_1(X)$ can be reduced to $H_1(X;\Z)$ and $A^{\pi}(X)$ can be simplified to $A(X)$ defined in (\ref{eq A(X)}).

At present we do not know how to treat non compactness of the moduli spaces
in the non commutative case,
which is the reason why we restrict $r \leq 2$.

The condition of  $\Gamma$ being admissible touches with smooth structure on the underlying manifold.
Actually it is a  crucial condition to verify the Novikov conjecture which is still open in general and is a fundamental problem
in the study of high dimensional smooth structure:
\begin{proposition}\label{CM}
\cite{Connes-Moscovici}
Let $M$ be a compact oriented smooth manifold
of arbitrary dimension.
If $\pi_1(M) \equiv \Gamma$ is admissible,
then $\Gamma$ satisfies the Novikov conjecture on homotopy invariance of higher signatures.
\end{proposition}
 In particular,
 the conclusion holds,
 if it admits a smooth algebra $C^{\infty}(\Gamma)$ with polynomial cohomology and rapid decay property.
It has not yet been known whether each  finitely presented group could satisfy
admissibility,
nor any counter example has  been known.

From the view point of study on smooth structure on four manifolds,
 it would be of particular interest for us to
try to construct the twisted Donaldson map without the conditions on fundamental groups.
There are two aspects for this.
One is that  in order to compute the twisted Donaldson invariant,
we will have to use the
Connes -Moscovici index formula which cannot be applied to a general group,
and we have to assume that the group is admissible.
The other is  that the rational cohomology group of the configuration space is
generated by the image of the standard $\mu$ map, in the case
of simply connected 4-manifolds.
So  it makes sense to speculate that the fundamental group $\Gamma$
 will appear in the generalization of the Donaldson invariant
  to the non-simply connected case.

\vspace{2mm}

It turns out that the pull backs of these maps recover
 the twisted Donaldson's invariant for commutative case
described above.
Suppose $\Gamma \cong \Z^n$ is free abelian,
and let $\Delta: \Pic(B\Gamma) \to \Pic(B\Gamma) \times \Pic(B\Gamma)$
be the diagonal map. Then it induces the homomorphisms
$$\Delta_* : H_*(\Gamma; \Q) \otimes  H_*(\Gamma; \Q) \to H_*(\Gamma; \Q)$$
passing through the isomorphisms
$H_*(\Gamma \times \Gamma; \Q)  \cong
H^*(\Pic(B\Gamma) \times \Pic(B\Gamma)  ; \Q)$
 and $H_*(\Gamma; \Q)  \cong  H^*(\Pic(B\Gamma))   ; \Q)$.
Then
$$\Delta_*(   \Psi^{\tw, (r)}_X) :
A(X)
\otimes  H_*(X;\Z) \otimes  \Omega^{Spin}_{*}(B\Gamma)\rightarrow
H_*(\Gamma; \Q)
$$
coincides with our twisted Donaldson invariant for commutative case.

\subsection{A conjecture}

We would like propose:
\begin{conjecture}
If $X$ splits into a connected sum
$X= X_1 \sharp X_2$ with $b^+(X_i) \geq 1$ for $i =1,2$,
then the twisted Donaldson invariant vanishes:
$$\Psi^{\tw, (r)}_{X, {\bf f}} \equiv 0.$$
\end{conjecture}
Let us  explain how to verify it for more restrictive case of groups,
and how its argument breaks for more general case of groups.
Actually such a restrictive case admits a cohomological formula of the twisted Donaldson invariant, which is obtained by applying  the Connes-Moscovici's index theory.
They  verified a fact that the Novikov conjecture is satisfied for a group
if it admits such a cohomology formula.
The Novikov conjecture states homotopy invariance of higher signatures and
arose from study on smooth structure of high dimensional manifolds.
 So our vanishing conjecture can measure complexity of  smooth structure
from the both view points of gauge theory and surgery theory.

\subsubsection{Splitting ASD moduli space}
Let $X  = X_1 \sharp X_2$ be a connected sum and
put $\pi_1(X) = \Gamma$ and $\pi_1(X_i) = \Gamma_i$ for $i=1,2$.
Then $\Gamma = \Gamma_1 * \Gamma_2$.
Let us recall Donaldson's fundamental result (Theorem 9.3.4 of \cite{DK})
\begin{theorem}
Let $E \to X$ be an $SU(2)$ bundle over an oriented closed smooth 4-manifold.
Suppose $X = X_1\sharp X_2$ splits as a connected sum with $b^+(X_i) \geq 1$.
Then the Donaldson invariant vanishes:
$$\Psi_X \equiv 0.$$
\end{theorem}

\begin{proof}
Let us describe the idea of the proof.
Let ${\frak M}(E,g)$ be the ASD moduli space over $(X,g)$.
Choose a family of Riemannian metrics $g_a$ so that the neck part
length grow as $[-T_a, T_a] \times S^3$ with $T_a \to \infty$.
Take  elements
$[A_a]  \in {\frak M}(E, g_a)$. Then a subsequence converges to
a pair of ASD solutions $(A_1,A_2)$ over
$E_i \to (X_i, g^i)$ for $i =1,2$.
This implies that when one stretches the neck of the connected sum
so that the connected sum part is isometric to
$[-T,T] \times S^3$ for large $T \gg 1$,
then $E$ splits as $E \cong E_1 \# E_2$
with $SU(2)$ bundles  $E_i \to X_i$.

Suppose $\dim {\frak M}(E, g_a) =0$ is zero. Then the additive formula
\[
     \dim {\frak M}(E) = \dim {\frak M}(E_1) + \dim {\frak M}(E_2) + 3
\]
shows that
$\dim {\frak M}(E_i, g^i) <0$ is negative for $i=1$ or $i=2$, and hence it is generically empty, which is a contradiction to the fact that $[A_i] \in \frak{M}(E_i)$.  Here $\dim {\frak M}(E_j)$ is the formal dimension. See Section 7.2.5 of \cite{DK} for the additive formula.
When $\dim {\frak M}(E, g_a) >0$ is positive, then the dual submanifold of the image of $\mu$ map
satisfy the same phenomena.

This leads to vanishing of the invariant.
\end{proof}

Let us try to follow a parallel argument for the twisted Donaldson invariants.
However it turns out that the free product of groups behaves rather differently from direct product. For example
it is well known that
$$K_0(C^*_r(F_n))  =K_0(C^*_r(F_{n-1} * \Z) )
\cong \dots \cong  \Z.$$
So   the Mischenko bundle
$\gamma = \tilde{X} \times_{\Gamma} C^*_r(\Gamma)$
will not split as the sum of
$\gamma_i = \tilde{X} \times_{\Gamma} C^*_r(\Gamma_i)$, and
 the twisted $\mu$ map cannot split as
the standard $\mu$ map, and hence the parallel argument breaks.

\subsubsection{Cohomological formula}
However as we will see below, if we pass through more
cohomological argument in non commutative geometry,
then we can verify vanishing of the twisted Donaldson invariants
for groups which posses some characteristics.
It turn out that such characteristics touches with validity of the Novikov conjecture
on homotopy invariance of higher signatures.

\begin{lemma}
Let $\Gamma$ be admissible and $HP^*(C^{\infty}(\Gamma))$ is finite dimensional.
Then the twisted Donaldson invariants vanish
$\Psi^{\tw, (r)}_{X, {\bf{f}}} \equiv 0$, if $X= X_1\sharp X_2$
splits as a connected sum with $b^+(X_i) \geq 1$.
\end{lemma}

This is verified by using the cohomological formula on the
 twisted $\mu$ map, when the group is admissible.

In theorem \ref{thm.coho.formula}, we have  verified  a cohomology formula
\begin{align*}
& \langle \
[\eta] \otimes [\xi] \otimes [C],
 \mu^{\tw}_{\bf f}(\alpha, \beta)
 \ \rangle \\
 &  = \ \langle
 \ \ch(\hat{\mathbb E})\wedge f^*\eta  \ , \
 [M \times C]  \ \rangle
  \ \langle  \
 \hat A( N)\wedge j^*\xi \ ,  \ [N] \ \rangle \
\in \C
\end{align*}
Hence our twisted Donaldson invariant is described by the classical Donaldson invariant,
when $\Gamma$ is admissible
and $HP^*(C^{\infty}(\Gamma))$ is finite dimensional.
 In particular it  verifies our conjecture in this case.
On the other hand Connes-Moscovici has verified the Novikov conjecture for admissible groups
(see proposition \ref{CM}).
The Novikov conjecture is still open in general and is a fundamental problem
in the study of high dimensional smooth structure. Hence the admissibility condition consists of an important class
of discrete groups, but assumes a strong constrain on structure of groups.

\vspace{3mm}

Our main theme in this paper is to construct a bridge
between Yang-Mills theory and non commutative geometry.
What we made non commutative was the cohomology classes over
 the standard ASD moduli space.
We did
not use non commutative ASD moduli space
which consists of gauge equivalent classes of ASD connections
over $E \otimes ( \tilde{X} \times_{\Gamma} C^*_r(\Gamma))$.
Its analysis seems quite hard, since it is infinite dimensional.
On the other hand, a non commutative action functional in non
commutative geometry was already introduced by Connes \cite{ConnesYM}.
As a next step in future work, we shall study non commutative ASD moduli space
based on our approach combined with Connes action functional.
We hope very much it will lead to a construction of `non commutative Donaldson invariant'.

We
also include basics materials of Yang-Mills gauge theory in section $2$
and non commutative geometry in section $4$,
for the convenience of readers in different fields.

\vspace{1cm}


\section{Review of Yang-Mills Gauge Theory}
\label{sec:Review Yang-Mills}

\subsection{Donaldson invariant}
Let $(X,g)$ be an oriented, closed, smooth Riemannian  four-manifold. Take 
an $SU(2)$ vector bundle
$E \to X$, and
$\frak{g} = P \times_{SU(2)} {\frak{su}}(2)$ be
the adjoint bundle, where $P$ is the principal $SU(2)$-bundle associated with $E$.

We denote  by $\frak{A}(E)$
the space of connections on $E$,
 which is an affine space of  sections on the  cotangent bundle twisted by the adjoint bundle.
So if we choose a base connection $A_0$ on $E$, then we can write it as:
$$\frak{A}(E) = A_0 + \Gamma(X; \Lambda^1 \otimes \frak{g}).$$
The gauge group ${\frak G}(E)$  acts on $\frak{A}(E)$ by conjugation
$g^*(A_0+ a) \equiv g^{-1} \circ (A_0+a) \circ g$.

The action of the subgroup $\{ \pm 1 \}$ of $\frak{G}(E)$  is trivial. Hence we have the action of $\frak{G}(E) / \{ \pm 1 \}$ on $\frak{A}(E)$. 
It is a basic matter that $\frak{G}(E) / \{\pm 1 \}$ acts on the set of irreducible connections
$\frak{A}^*  (E) \subset \frak{A}(E)$  freely,
and hence the quotient space:
$$\frak{B}^*(E) = \frak{A}^*  (E) / {\frak G}(E)  = \frak{A}^*(E) / (\frak{G} (E) / \{ \pm \})$$
 becomes a regular Hilbert manifold under taking  completion
by a Sobolev norm.

The Hodge $*$ operator acts on  two forms into themselves as:
 $* : \wedge^2 V \to \wedge^2 V$ 
with $*^2=1$, if $V$ is a  4-dimensional Euclidean space. 
So $\wedge^2 V$ splits into  self-dual and anti self-dual vectors
$\wedge^2 V = (\wedge^2 V)^+ \oplus (\wedge^2 V)^-$.
Let:
$$F^+ : \frak{A}(E) \to \Gamma(X; \frak{g} \otimes \wedge^+)$$
be the self-dual part of the curvature, 
which is given by  the  orthogonal projection of the curvature $F$.
Let us
 consider
  the space of anti self-dual connections:
$$\tilde{\frak M}(E,g) = \{  \ A \in {\frak A}(E); \ F^+(A)=0 \ \} \ \subset \  {\frak A}(E) .$$
 The gauge group ${\frak G}(E)$ acts on $\tilde{\frak M}(E,g)$.
 
 \begin{definition}
 The moduli space of anti self-dual (ASD) connections is given by   the quotient:
 $${\frak M}(E,g)  =  \tilde{\frak M}(E, g)  \ / \ {\frak G}(E).$$
 \end{definition}
It follows from the Atiyah-Singer index theorem that 
the formal dimension of the moduli space is given by:
$$8 c_2(E) +\frac{3}{2}( \chi(X) - \tau(X)) =8 c_2(E) -3(1-b_1(X)+b^+(X))
$$
where $\chi(X)$ is the Euler characteristic  and $\tau(X)$ is the signature of $X$ respectively.

It follows from transversality argument that for
  a generic metric $g$, the moduli space $\frak{M}(E, g)$ is smooth
 if $b^+(X) > 1$ and $c_2(E) > 0$. 
 Hence we have the inclusion:
 $$\frak{M}(E, g) \ \subset  \ \frak{B}^*(E) .$$
 In general this does not give a homology class in $H_*(\frak{B}^*(E) ; \Z)$,
 since the moduli space is non compact.

Let us describe roughly how to formulate the Donaldson invariant.  Let $\frak{G}^0(E)$ be the framed gauge group:
\[
      \frak{G}^0(E) := \{ \ u \in \frak{G}(E) ; \ u(x_0) = 1 \ \}
\]
where $x_0$ is a fixed point in $X$.  Put: 
\[
   \widetilde{\frak{B}}(E) := \frak{A}(E)  \ / \ \frak{G}^0(E)
\]
and consider a vector bundle 
$\widetilde{\mathbb E}$ on $X \times \widetilde{\frak B}(E)$ defined by the following:
\[
        \widetilde{\mathbb E} := E \times_{ \frak{G}^0(E) } \frak{A}(E). 
\]
It is well known  that the adjoint bundle $\frak{g}_{\widetilde{\mathbb{E}}}$ of $\widetilde{\mathbb{E}}$ descends to a vector bundle $\frak{g}_{\mathbb E}$ over $X \times \frak{B}^*(E)$, while 
the vector bundle $\tilde{\mathbb E}$ itself does not descend to a vector bundle over $X \times \frak{B}^*(E)$ in general. 

We define:
\begin{align*}
  &    \mu : H_*(X;\Z) \rightarrow H^{4-*}(\frak{B}^*(E) ; \Q), \\
& \qquad \ 
     \mu(\alpha) := -\frac{1}{4}p_1( \frak{g}_{\mathbb E}) / \alpha.
\end{align*}

   A choice of  an orientation ${\mathcal O}$ of the vector space ${\mathcal H}_1(X) \oplus {\mathcal H}^+(X)$ gives an orientation on the moduli space, 
where ${\mathcal H}^1(X)$ and ${\mathcal H}^+(X)$ are the spaces of harmonic 1-forms and self-dual harmonic two forms respectively.   See \cite{Donaldson orientation}.  Let $A(X)$ be as in (\ref{eq A(X)}). Then the  {\em Donaldson invariant} is a linear map:
$$
      \Psi_{X} :   A(X) \rightarrow \Q
$$
which  depends on the orientation 
${\mathcal O}$ 
up to sign,
but  usually we drop it  from our notation. 
Roughly speaking, the Donaldson invariant
 is given  by the integral formula:
\[
      \Psi_{X}(\alpha_1, \dots, \alpha_d) = \int_{ {\frak M}(E, g)} 
      \mu(\alpha_1) \cup \cdots \cup \mu(\alpha_d)
\]
if there is an $SU(2)$ vector bundle $E$ with:
$$\dim {\mathfrak M}(E, g) = \deg \{ \ \mu(\alpha_1) \cup \cdots \cup \mu(\alpha_d)  \ \}.$$
 Otherwise we put $\Psi_X(\alpha_1, \dots, \alpha_d) = 0$.  
 
In order for 
the above formula to make sense,
 the cohomology classes $\mu(\alpha_1) \cup \cdots \cup \mu(\alpha_d)$ should have   compact support,
 since ${\frak M}(E)$ is not compact in general. 
 To overcome this difficulty,
   we pass through dual submanifolds of $\mu(\alpha_j)$,
    which behave nicely near  the end of the moduli space. 
    Later we will describe such construction.
 
     Technically speaking, reducible $SU(2)$-flat connections affect 
     to find such dual submanifolds. To avoid such  connections, we take the blow-up:
\begin{equation}\label{eq:X.blowup}
      \hat{X} = X \# \overline{\mathbb{CP}}^2
\end{equation}
and choose the $U(2)$-vector bundle $\hat{E}$ on $\hat{X}$ with: 
\begin{equation}\label{eq:cond.blowup}
c_1(\hat{E}) = e \quad \text{ and } \quad c_2(\hat{E}) = c_2(E).
\end{equation}
Here $e$ is the standard generator of $H^2( \overline{\mathbb{CP}}^2 ; \Z ) (\subset H^2(\hat{X};\Z))$.  
 See remark \ref{blow-up} below.
This technique was introduced in \cite{MM}.  See also \cite{KM} and section 6.3 of \cite{Donaldson Floer}. 
 
Fix a connection $a_0$ on $\det \hat{E}$. 
We will always consider connections on $\hat{E}$ compatible with $a_0$
so that their determinants concide with $a_0$.
The moduli spaces should be denoted as
$\frak{M}(\hat{E}, \hat{g},a_0)$ apriori.
But the invariants defined here are actually all independent of $a_0$.
So we omit to denote $a_0$ and write it as $\frak{M}(\hat{E}, \hat{g})$.

\begin{remark}
 Take two connections $a_0, a_1$ on $\det \hat{E}$. 
 For $t \in [0, 1]$, put $a_t= (1-t) a_0 + t a_1$.
  Then it  gives a family of the moduli spaces parametrized by $t$,
  and the family forms a cobordism between 
  $\frak{M}(\hat{E}, \hat{g}, a_0) $ and $\frak{M}(\hat{E}, \hat{g}, a_1)$. 
This cobordism is diffeomorphic to the trivial cobordism, and  the invariants are independent of choice of $a_0$. 
\end{remark}

If $c_2(\hat{E}) > 0$, for a generic metric $\hat{g}$, the moduli space ${\frak M}(\hat{E}, \hat{g})$ of ASD connections on $\hat{E}$  is smooth and  included in ${\frak B}^*(\hat{E})$. Moreover we have the dimension  formula:
\begin{equation}\label{eq:dim.moduli.blowup}
       \dim {\frak M}(\hat{E}, \hat{g}) = \dim {\frak M}(E, g) + 2. 
\end{equation}

Let us  explain how to construct a submanifold which is dual to $\mu(\alpha)$ and which  behaves nicely near 
 the end of the moduli space $\frak{M}(\hat{E}, \hat{g})$. 
  Take an integral homology class $\alpha = [M] \in H_{*}(X;\Z)$ for $* \leq 3$, where $M$ is a submanifold in $X$. 
  Consider a suitable neighborhood $U$ of $M$ in $X$ in the sense of \cite[p.589]{KM}.  
  We may  think of $U$ as a subset in $\hat{X}$.  Let ${\frak B}^*_U$ be the configuration space of irreducible 
   connections on $\hat{E}|_{U} (= E|_U)$. 
   Since the restricton of  an irreducible ASD connection over  $\hat{X}$ to the open subset $U$ is still irreducible, 
   we have the restriction map:
\[
         r_U : \frak{M}(\hat{E},  \hat{g}) \rightarrow \frak{B}^*_U. 
\] 
For any  subset $B \subset \frak{B}^*_{U}$, we put:
\[
        \frak{M}(\hat{E} , \hat{g}) \cap B := \{ \  [A] \in \frak{M}(\hat{E} , \hat{g}) ; \  r_U([A]) \in B \  \}. 
\]
We will construct  a generic submanifold $V_M$ in $\frak{B}^*_U$ of codimension $4-*$ such that:
\[
         \frak{M}(\hat{E}, \hat{g}) \cap V_{M}
\]
is transverse and is a submanifold in $\frak{B}^*(\hat{E})$ dual to $\mu([M])$. This
behaves  nicely near the end of the moduli space. 

Let us  present how to construct $V_{M}$  for  $\alpha = [M] \in H_*(X;\Z)$ for $* \leq 3$. 

\subsection{Dual submanifolds  of $\mu(\alpha)$}

\subsubsection{The case $\operatorname{deg} \alpha = 0$}  \label{mu map deg 0}
We consider the degree 0 case. Let $x$ be a point in $X$.   
Take a suitable neighborhood $\nu(x)$ of $x$.  Let $\frak{g}_{x, \mathbb{C}}$ be the complexification  of the vector bundle 
\[
      \frak{g}_{x}  :=  \widetilde{\frak{B}}_{\nu(x)}^* \times_{ad} \frak{su}(2)
\]
over $\frak{B}^*_{\nu(x)}$ with fiber $\frak{su}(2)$.  Take generic section $s_1$ and $s_2$ of $\frak{g}_{x, \mathbb{C}}$.  Let $V_x$ in $\frak{B}^*_{\nu(x)}$ be the locus where $s_1$ and $s_2$ are not linear independent.   Then $V_x$ is a codimension-4 stratified subset in $\frak{B}^{*}_{\nu(x)}$, which is  dual to $4\mu([x])$.   

\begin{remark} \label{rem V_x}
The subspace $V_{x}$ is dual to $4\mu([x])$ rather than $\mu([x])$. Because of this fact, we will have the factor $\frac{1}{4}$ in the definition of the Donaldson invariant in definition \ref{def D inv}. 
\end{remark}


\subsubsection{The case $\operatorname{deg} \alpha = 1$} \label{section mu map deg 1}
Take a homology class $[\ell] \in H_1(X;\Z)$, where $\ell$ is a loop in $X$. We may regard $\ell$ as a loop in $\hat{X}$.    Fix a neighborhood $\nu(\ell)$ of $\ell$ in $\hat{X}$ which is suitable in the sense of \cite[p.589]{KM}.  For each connection $A$ on $E|_{\nu(\ell)} (= \hat{E}|_{\nu(\ell)})$ consider its   holonomy $h_{\ell}(A) \in SU(2)$ along  $\ell$.  The holonomy assignment 
gives   a section $h_{\ell} : {\frak B}^*_{\nu(\ell)} \rightarrow \ad \widetilde{\frak B}^*_{\nu(\ell)}
= \widetilde{\frak{B}}^*_{\nu(\ell)} \times_{\ad} SU(2)$.   
If we perturb $h_{\ell}$ slightly, then  we get a section $h_{\ell}'$ of $\ad \widetilde{\frak{B}}^*_{\nu(\ell)}$ such that $h_{\ell}'$ is transverse to the trivial section $s_0$ of $\ad \widetilde{\frak{B}}^*_{\nu(\ell)}$ defined by $s_0([A]) = [A, 1]$. Put:
 \[
         V_{\ell} := (h_{\ell}')^{-1}( 1 ). 
 \]
 Here $1$ stands for the image of $s_0$.  
 Then $V_{\ell}$ is a submanifold in $\frak{B}^*_{\nu(\ell)}$ of codimension $3$.


\subsubsection{The case $\operatorname{deg} \alpha = 2$}
Take $\alpha \in H_2(X;\Z)$, and 
let
 $\Sigma  \ \subset \ X$
 be an embedded, closed and  oriented surface representing $\alpha$.
 Consider the restriction
 $E|_{\Sigma} \to \Sigma$, and let
 $\widetilde{{\frak B}}_{\Sigma} = {\frak A}_{\Sigma}/ \frak{G}_{\Sigma}^0$ 
be the configuration space of framed connections on $E|_{\Sigma} (= \hat{E}|_{\Sigma})$, 
with the framed gauge group
 ${\mathfrak G}^0_{\Sigma} = \{ g \in {\mathfrak G}_{\Sigma} | g(x_0) = 1 \}$ 
 over $\Sigma$ for a fixed point $x_0 \in \Sigma$.  
We have the universal $SU(2)$-vector bundle:
\[
    \widetilde{\mathbb E}_{\Sigma} = (E|_{\Sigma}) \times_{ {\frak G}^{0}_{\Sigma} } {\frak A}_{\Sigma} 
      \rightarrow 
     \Sigma \times \widetilde{\frak B}_{\Sigma}
\]
which presents  the universal family of bundles with connections on $\Sigma$.

Fix a spin structure on $\Sigma$, and let  $D$ be the Dirac operator over $\Sigma$ as  the twisted
$\bar{\partial}$ operator. 
Each $SU(2)$ connection $A$ gives the associated Dirac operator $D_A$, and hence
such assignment  induces the universal  family of the Dirac operators over the configuration space
 $\widetilde{{\frak B}}_{\Sigma}$. 
 
 Let $\widetilde{\frak{L}}_{\Sigma} \rightarrow \widetilde{\frak B}_{\Sigma}$ be the determinant line bundle of the family of the  Dirac operators. We write $\widetilde{\frak{B}}(\hat{E})$ and $\frak{B}^*(\hat{E})$ for the configuration spaces of framed connections and irreducible connections over
  $\hat{E}$ compatible with the fixed connection $a_0$.
 It  is known  that the 
pull-back $\tilde{r}_{\Sigma}^*   \widetilde{\frak L}_{\Sigma}$ descends to a line bundle:
 $${\frak L}_{\Sigma} \rightarrow {\frak B}^*( \hat{E} )$$ where
$\tilde{r}_{\Sigma} : \widetilde{\frak B}( \hat{E}) \rightarrow \widetilde{\frak B}_{\Sigma}$ is the restriction map. 
The Atiyah-Singer  index theorem for family computes  the following:

\begin{proposition}
$\mu([\Sigma]) = c_1(\frak{L}_{\Sigma}) \  \in  \ H^2({\frak B}^*( \hat{E} );  \Q)
$.
\end{proposition}

Take  a neighborhood $\nu(\Sigma) \subset X $ of $\Sigma$  which is suitable in the sense of \cite[p.589]{KM}.  As above $\tilde{\frak{L}}_{\Sigma}$ naturally induces  a line bundle $\frak{L}_{\nu(\Sigma)} \rightarrow \frak{B}^*_{\nu(\Sigma)}$.
Then the  submanifold $V_{\Sigma}$ is given by the zero set of $s_{\Sigma}$:
\[
         V_{\Sigma} := s_{\Sigma}^{-1}(0)
\]
for any  generic section $s_{\Sigma} : \frak{B}^*_{\nu(\Sigma)} \rightarrow \frak{L}_{\nu(\Sigma)}$. 
\begin{remark}
Notice that the first Chern class of the determinant line bundle is the same
as the one of the family of the index bundle. A merit to use the determinant line bundle is
that its Poincar\'e dual class is expressed by the  pull back of a generic section of $\mathfrak{L}_{\Sigma}$.
This is technically important since the  moduli space is non compact  in general. 
\end{remark}

\subsubsection{The case $\operatorname{deg} \alpha = 3$}
Take $\alpha = [M] \in H_3(X;\Z)$ represented by an closed, oriented 3-manifold $M$ embedded in $X$. Take a suitable neighborhood $\nu(M)$.  We have the Chern-Simons functional 
\[
        CS_{M} : \frak{B}^*_{\nu(M)} \rightarrow S^1 = \R / \Z. 
\]
We may suppose that $0 \in S^1$ is a regular value of $CS_{M}$ after perturbing $CS_{M}$ slightly.   The submanifold is given by
\[
        V_{M} := CS_{M}^{-1}(0). 
\]

\subsection{Homology and bordism classes of the intersection of submanifolds $V_{M}$}

\subsubsection{Main statement}

 Take homology classes  $[M_1], \dots, [M_d] \in H_{\leq 3}(X;\Z)$ and their   neighborhoods $\nu(M_i)
 \subset X$ of $M_i$, which is suitable in the sense of \cite[p.589]{KM}.  (If $[M_j] \in H_0(X;\Z)$, then we assume that $[M_j]$ is the generator $[pt]$ of $H_0(X;\Z)$ which is represented by a single point $pt$ of $X$.)
 We may suppose that
  \begin{equation} \label{eq M i j k}
         \nu(M_{i_1}) \cap \cdots \cap \nu( M_{i_p}) = \emptyset
\end{equation}
 for $\{ i_1, \dots, i_p \} \subset \{ 1, \dots, d \}$ with
 \[
       ( 4-\dim M_{i_1} ) + \cdots + ( 4 - \dim M_{i_p}) > 4.
 \]

\begin{proposition} \label{prop V compact}
Suppose $b^+(X) > 1$, and  the inequarity
\[
   \dim  \ {\frak M}(\hat{E}, \hat{g}) = \sum_{j=1}^{d}  \ (4 - \dim M_{j}) +r
\]
holds for some $r \in \{ 0, 1 , 2, 3 \}$. 
 Then  the intersection:
\begin{equation} \label{eq V}
        V(\hat{E}, \hat{g}; M_1, \dots, M_d) :=  \  {\frak M}(\hat{E}, \hat{g}) \cap V_{M_1} \cap \cdots \cap V_{M_d}
\end{equation}
is a compact, smooth submanifold in ${\frak M}(\hat{E}, \hat{g})$ of dimension $r$.  

The orientation ${\mathcal O}$ of $\mathcal{H}^1(X) \oplus \mathcal{H}^+(X)$ induces an orientation of the intersection.   If $r \leq 2$, the homology class: 
\[
[V(\hat{E}, \hat{g}; M_1, \dots, M_d)] \in H_r({\frak B}^*(\hat{E});\Z)
\]
 and the oriented bordism class: 
\[
      [V(\hat{E}, \hat{g}; M_1, \dots, M_d)] \in \Omega_{r}^{SO}({\frak B}^*(\hat{E}))
\] 
are both  independent of choice of $\hat{g}$ and $V_{M_j}$.
It depends only on $\hat{E}$, the homotopy classes $[M_j : S^1 \rightarrow X] \in \pi_1(X)$ with $\dim M_j = 1$   and the homology classes $[M_j] \in H_*(X;\Z)$  with $\dim M_j = 0, 2$ or $3$.
\end{proposition}

\subsubsection{Compactness of $V(\hat{E}, \hat{g};M_1, \dots, M_d)$}
We will prove the compactness of  the space $V(\hat{E}, \hat{g}; M_1, \dots, M_d)$ in Proposition \ref{prop V compact}. 
It follows from a standard argument that 
  $V(\hat{E}', \hat{g}; M_{i_1}, \dots, M_{i_{p}})$ 
 are  smooth manifolds of the expected dimension
 for a generic metric $\hat{g}$ on $\hat{X}$, generic representatives $V_{M_i}$
 and  $\{ i_1, \dots, i_{p} \}  \subset \{ 1, \dots, d \}$.
  Here
 $\hat{E}'$ is
 any $U(2)$ vector bundle  with $0 \leq c_2(\hat{E}') \leq c_2(\hat{E})$ and  $c_1(\hat{E}') = e \in H^2(\hat{X}; \Z)$. In particular, the negativity of the  dimension:
\[
      \dim V(\hat{E}'; M_{i_1}, \dots, M_{i_p}) = \dim \frak{M}(\hat{E}') - \sum_{k=1}^{p} (4 - \dim M_{i_k})  < 0
\]
implies that
\[
           V(\hat{E}';M_{i_1}, \cdots, M_{i_p}) = \emptyset. 
\]

To verify   that $V(\hat{E}, \hat{g}; M_1, \dots, M_{d})$ is compact, take a sequence $[A_n]$ in $V(\hat{E}, \hat{g}; M_1, \dots, M_d)$. After passing to a subsequence, 
it converges as
\[
      [A_n] \rightarrow ([A], x_1, \dots, x_m) \in \frak{M}(\hat{E}') \times \operatorname{Sym}^m X
\]
where   $\hat{E}'$ is a $U(2)$-vector bundle with $c_2(\hat{E}') = c_2(\hat{E}) - m$ and  $c_1(\hat{E}') = e$ for some $0 \leq m \leq c_2(\hat{E})$, and $x_i$ are the blowing up points.

We have to show $m = 0$. For each $x_j$, let  $M_{i(j, 1)}, \dots, M_{i(j, p(j))}$ be the submanifolds  with $x_j \in \nu(M_{i(j, k)})$ for $k = 1, \dots, p(j)$.   
Then 
we have the bounds:
\begin{equation}   \label{eq 4 - dim M}
      (4 - \dim M_{ i(j,1) }) + \cdots + ( 4 - \dim M_{ i(j, p(j)) }) \leq 4
\end{equation}
since (\ref{eq M i j k}) and
\[
             x_j \in \nu(M_{i(j,1)}) \cap \cdots \cap  \nu(M_{ i(j, p(j)) }).
\]
Let $M_{i_1}, \cdots, M_{i_q}$ be the submanifolds which do not contain any of $x_j$. Then  we have:
\[
     [A] \in V(\hat{E}'; M_{i_1}, \dots, M_{i_q}).
\]
On the other hand,  the following estimates hold:
\begin{equation} \label{eq dim V E'}
   \begin{split}
       \dim V(\hat{E}'; M_{i_1}, \cdots, M_{i_q}) 
       \leq  &  \dim V(\hat{E}; M_{1}, \dots, M_{d}) - 8m    \\
                 + \sum_{j=1}^{m} &  \left\{  (4 - \dim M_{i(j,1)}) + \cdots + ( 4 - \dim M_{ i(j, p(j)) } ) \right\}  \\
       \leq  &  \dim V(\hat{E};  M_{1}, \dots, M_{d}) - 4m \\
        =    & r - 4m. 
   \end{split}
\end{equation}
So if $m > 0$ could be  positive,  then  $\dim V(\hat{E}'; M_{i_1}, \dots, M_{i_{p}}) < 0$ is negative, 
and hence
$V(\hat{E}'; M_{i_1}, \dots, M_{i_p})$
must be empty.
This is a contradiction since $[A] \in V(\hat{E}'; M_{i_1}, \dots, M_{i_p})$. 
Thus  $m=0$ and $V(\hat{E}; M_{1}, \dots, M_{d})$ is compact.

\subsubsection{Invariance under change of metric}
We will  check that the classes  of $V(\hat{E}; M_1, \dots, M_d)$ 
in  both $H_{r}(\frak{B}^*(\hat{E});\Z)$ and $\Omega^{SO}_r(\frak{B}^*(\hat{E}))$ 
are   independent of choice of  metrics $\hat{g}$  as follows.
Take another metric $\hat{g}'$ and a generic path $G=\{ \hat{g}_t \}_{0 \leq t \leq 1}$ from $\hat{g}$ to $\hat{g}'$.
This    gives  a parametrized  moduli spaces $\frak{M}(\hat{E}, G)$. 
Then we obtain a cobordism $V(\hat{E}, G; M_1, \dots, M_d)$
between the boundary components
 associated with the moduli spaces with respect to  $\hat{g}$ and $\hat{g}'$, if we verify that it is compact.

 We follow a parallel argument as above.
 Take  $[A_n] \in V(\hat{E}, G; M_1, \dots, M_d)$. 
 Then they converge 
$    [A_n] \rightarrow ([A], x_1, \dots, x_m)$
after passing to a subsequence, and obtain an element 
$     [A] \in V(\hat{E}', G; M_{i_1}, \dots, M_{i_p})$.
Then the dimension estimate:
\[
     \dim  V(\hat{E}', G; M_{i_1}, \dots, M_{i_p}) \leq r+1 -4m
\]
still implies that $m $ cannot be positive
for $r = 0 ,1, 2$.

\subsubsection{Invariance under change of representatives $V_M$: The case $\dim M = 0$}

The classes $[V(\hat{E}, \hat{g}; M_1, \dots, M_d)]$ are also
 independent of the choice of the representatives $V_{M_j}$.  
 We are left to prove that the classes 
 \[
 [V(\hat{E}; M_1, \dots, M_d)] \in H_r({\frak B}^*(\hat{E});\Z),  \quad
 [V(\hat{E}; M_1, \dots, M_d)] \in \Omega_{r}^{SO}({\frak B}^*(\hat{E}))
 \]
 depend only on the homotopy class $[M_j : S^1 \rightarrow X] \in \pi_1(X)$ with $\dim M_j = 1$ and homology class $[M_j] \in H_{*}(X;\Z)$ with $\dim M_j = 0, 1$ or $3$. 
 
First we will consider the case $\dim M_1 = 0$.  That is $M_1$ is a point $x$ of $X$.  As in Section \ref{mu map deg 0}, choose pairs of generic sections $(s_1, s_2)$, $(s_1', s_2')$ of the bundle $\frak{g}_{x} \otimes \mathbb{C}$, $\frak{g}_{x'} \otimes \mathbb{C}$ over $\mathcal{B}^*_{\nu(x)}$, $\mathcal{B}^*_{\nu(x')}$.   Then we get two representatives $V_{M_1}, V_{M_1}'$ which are the loci where $(s_1, s_2)$, $(s_1', s_2')$ are not linear independent.   
Fix a path $\gamma$ on $X$ from $x$ to $x'$ and choose a neighborhood $\nu(\gamma)$ of $\gamma$ with 
\[
      \nu(x), \ \nu(x')  \subset \nu(\gamma). 
\]
We have the adjoint bundle $\frak{g}_{\nu(\gamma)} \otimes \mathbb{C}$ over $\mathcal{B}^*_{\nu(\gamma)}$.  Put
\[
        \mathcal{B}^{**}_{\nu(\gamma)} :=
        \{  [A] \in \mathcal{B}_{\nu(\gamma)} | \text{ Both $r_{\nu(x)}(A)$ and   $r_{\nu(x')}(A)$ are irreducible }    \}. 
\]
Here $r_{\nu(x)}$, $r_{\nu(x')}$ are the restrictions. 
Take pair of generic sections $(t_1, t_2)$  of the bundle $(\frak{g}_{\nu(\gamma)} \otimes \mathbb{C}) \times [0, 1]$ over $\mathcal{B}_{\nu(x)}^{**} \times [0, 1]$ such that
\[
       t_i|_{\mathcal{B}^*_{\nu(x)} \times \{ 0 \}} = r_{\nu(x)}^* s_i, \quad 
       t_i|_{ \mathcal{B}^*_{\nu(x)} \times \{ 1 \}  } =   r_{\nu(x')}^* s_i'. 
\]
Let $V(t_1, t_2)$ be the locus where $t_1, t_2$ are not linear independent.  Then  we can show that
\[
          \{ V(\hat{E}, \hat{g}; M_2, \dots, M_d) \times [0, 1] \}  \cap V(t_1, t_2)
\]
is a bordism between 
\[
              V(\hat{E}, \hat{g}; M_2, \dots, M_d) \cap V_{M_1}
\]
and 
\[
          V(\hat{E}, \hat{g}; M_2, \dots, M_d) \cap V_{M_1}'. 
\]

\subsubsection{Invariance under change of representatives $V_M$: The case $\dim M = 1$}
 Suppose that $M_1, N_1 : S^1 \rightarrow X$ are embeddings in $X$ with $M_1(0) = N_1(0) = x_0$ and with 
 \[
       [M_1] = [N_1]  \quad \text{in $\pi_1(X)$}.
\] 
We can find a homotopy
\[
     H : S^1 \times [0, 1] \rightarrow X
\]
with $H(0, t) = x_0$ and with $H(\cdot, 0) = M_1, H(\cdot, 1) = N_1$.  Put
\[
    W := \{ H(s, t) | s \in S^1, t \in [0, 1] \} \subset X.
\]
Choose a neighborhood $\nu(W)$ of $W$ in $X$ with
\[
      \nu(M_1), \  \nu(N_1)  \subset \nu( W ). 
\]
After perturbing $H$, we may suppose that if
\[
      \nu(W) \cap \nu(M_{i_1})  \cap \cdots \cap \nu(M_{i_p})  \not= \emptyset
\] 
then
\begin{equation}   \label{eq 4 - dim M 2}
       2 + (4 - \dim M_{i_1}) + \cdots + (4 - \dim M_{i_p}) \leq  4
\end{equation}
for distinct   $i_1, \dots, i_p \geq 2$.
Put
\[
    {\frak B}^{**}_{\nu(W)} \\
     := \{  [A] \in {\frak B}^*( E|_{\nu(W)} ) \ | \ \text{$[A|_{\nu(M_1)  }]$ and $[A|_{\nu(N_1)}]$  are irreducible} \}.
\]
We define a homotopy of  sections
\[
      H' : {\frak B}^{**}_{\nu(W)} \times [0, 1] \rightarrow \ad \widetilde{\frak B}^{**}_{\nu(W)}
\]
of the adjoint bundle by
\[
        H'([A], t) := h_{t}([A]),
\]
where $h_{t}$ stands for the holonomy along the loop $H|_{S^1 \times \{ t \}}$. 
We denote  the preimage $(H')^{-1}(1)$ by $V_{W}$.  We will prove that
\begin{equation} \label{eq V tildeW}
      \{  V(\hat{E}; M_2, \dots, M_d) \times [0, 1] \} \cap V_{W}
\end{equation}
is compact and a bordism between $V(\hat{E}; M_1, M_2, \dots, M_d)$ and $V(\hat{E}; N_1, M_2, \dots, M_d)$.  Perturbing the homotopy $H'$, we may suppose that the intersection (\ref{eq V tildeW}) is transverse.  To prove it is compact, take a sequence $([A_n], t_n)$ in the intersection (\ref{eq V tildeW}).  After passing to a subsequence,
\[
     [A_n] \rightarrow ([A_{\infty}], x_1, \dots, x_m) \in {\frak M}(\hat{E}') \times \operatorname{Sym}^m X, \  
     t_n \rightarrow t. 
\]
Suppose that $x_1, \dots, x_m$ are all not in $\nu(W)$.   Let $M_{i(j, 1)}, \dots, M_{i(j, p(j))}$ be the submanifolds with $x_j \in \nu(M_{i(j,k)})$ for $k = 1, \dots, p(j)$.  Then we have the inequality (\ref{eq 4 - dim M}) for each $j$ and
\[
       ([A], t) \in \{ V(\hat{E}'; M_{i_1}, \dots, M_{i_q}) \times [0, 1] \}  \cap V_{W}. 
\]
Here $i_1, \dots, i_q \geq 2$ and $M_{i_1}, \dots, M_{i_q}$ are the submanfiolds with $x_j \not \in \nu(M_{i_k})$ for all $j$.  As (\ref{eq dim V E'}), we have
\[
         \dim ( \{ V(\hat{E}' ; M_{i_1}, \dots, M_{i_q}  ) \times [0, 1] \}  \cap V_{ W } ) \leq r + 1 -4m. 
\]
If $m \geq 1$,  $\dim ( \{ V(\hat{E}' ; M_{i_1}, \dots, M_{i_q} )  \times [0, 1] \}  \cap V_{W}) < 0$ and we obtain a contradiction.  Thus $[A_n]$ must converge to a point in the intersection (\ref{eq V tildeW}). 
Next suppose that $x_1 \in \nu(W)$.  Then we have
\[
      [A] \in V(\hat{E}'; M_{i_1}, \dots, M_{i_q}).
\]
Here $i_1, \dots, i_q \geq 2$ and $M_{i_1}, \dots, M_{i_q}$ are the submanifolds with $x_j \not\in \nu(M_{i_k})$ for all $j$.
Let $M_{i(j, 1)}, \dots, M_{i(j, p(j))}$ be the submanifolds with $i(j, k) \geq 2$ and with $x_j \in \nu(M_{i(j, k)})$ for $k = 1, \dots, p(j)$.  
By  (\ref{eq 4 - dim M}) and (\ref{eq 4 - dim M 2}),  we have
\[
  \begin{split}
    & \dim V(\hat{E}'; M_{i_1}, \dots, M_{i_q})    \\
   & \quad  \leq V(\hat{E}; M_{1}, \dots, M_{d}) - 8m    \\
   & \qquad   + 3 + (4 - \dim M_{i(1, 1)} ) + \cdots + (  4 - \dim M_{i(1, p(1))} )  \\
   & \qquad   +  \sum_{j=2}^{m} \{ ( 4- \dim M_{i(j, 1)}) + \dots + ( 4 - \dim M_{i(j, p(j))}) \}  \\
   & \quad \leq r - 8m + 5 + 4(m-1) \\
   &\quad = r + 1 - 4 m.
  \end{split}
\]
Again, if $m \geq 1$,  $ \dim V(\hat{E}'; M_{i_1}, \dots, M_{i_q}) < 0$ and this is a contradiction. Therefore the intersection (\ref{eq V tildeW}) is compact and a bordism between $V(M_1, M_2, \dots, M_d)$ and $V(N_1, M_2, \dots, M_d)$.     So we have
\[
        [V(\hat{E};M_1, M_2, \dots, M_d)] = [V(\hat{E}; N_1, M_2, \dots, M_d)]
\]
in $H_r({\frak B}^*(\hat{E});\Z)$ and $\Omega_r^{SO}({\frak B}^*(\hat{E}))$. 


\subsubsection{Invariance under change of representatives $V_M$: The case $\dim M = 2$}
Next we will consider the case $\dim M_1 = 2$. Suppose that
\[
      [M_1] = [N_1] \in H_2(X;\Z). 
\]
Recall that $V_{M_1}$, $V_{N_1}$ are the zero sets of sections
\[
      s_{M_1} : {\frak B}^*_{\nu(M_1)} \rightarrow {\frak L}_{\nu(M_1)}, \
      s_{N_1} : {\frak B}^{*}_{\nu(N_1) } \rightarrow {\frak L}_{\nu(N_1)}. 
\]
There is the complex line bundle $L \rightarrow X$ with $c_1(L) = PD([M_1])$ and generic sections $s_0, s_1$ of $L$ such that
\[
       s_0^{-1}(0) = M_1, \ s_1^{-1}(0) = N_1. 
\]
Take a generic  section
\[
       H : X \times [0, 1] \rightarrow L
\]
with
\[
      H(\cdot, 0) = s_0, \quad H(\cdot, 1) = s_1. 
\]
Put
\[
      W := H^{-1}(0) \subset X \times [0, 1]. 
\]
Then $W$ is a submanifold in $X \times [0, 1]$ with $\partial W = M_1 \times \{ 0 \} \coprod N_1 \times \{ 1 \}$. 
Choose a neighborhood $\nu(W)$ of $W$ in $X \times [0, 1]$ such that
\[
    \nu(M_1) \times \{ 0 \} \cup \nu(N_1) \times \{ 1 \} \subset \nu(W). 
\]
First we will show that
\[
     \pi_1(\frak{B}^{**}_{\nu(W)}) = 1,
\]
where $\frak{B}^{**}_{\nu(W)}$ is the configuration space of connections $A$ on $\nu(W)$ such that $A|_{\nu(M_1) \times \{ 0 \}}$ and $A|_{\nu(N_1) \times \{ 1 \}}$ are irreducible. 
To see this, we consider the fibration
\[
         \frak{G}_{\nu(W)} \rightarrow \frak{A}^{**}_{\nu(W)} \rightarrow  \frak{B}^{**}_{\nu(W)}. 
\]
Recall that $\nu(M_1)$ is a neighborhood of the sum of $M_1$ and loops in $X$. Similarly for $\nu(N_1)$. (See p. 589 of \cite{KM}.) So we may suppose that $\nu(W)$ is a neighborhood of $W \cup l_1 \cup \cdots \cup l_{q}$, where $l_j$ is a loop. 
The homotopy exact sequence associated with the fibration shows that
\[
  \begin{split}
     \pi_1(\frak{B}^{**}_{\nu(W)})
     & \cong \pi_0(\frak{G}_{\nu(W)})   \\
       & = [\nu(W), SU(2)]   \\
    &  \cong [ W \cup l_1 \cup \cdots \cup l_q, SU(2)]   \\
   &  = H^3(W \cup l_1 \cup \cdots \cup l_k; \Z)   \\
    & = 0. 
  \end{split}
\]
We have used Corollary 16 of  \cite[Chapter 8, Section 1]{Spanier}. 

Let 
\[
     \mathbb{P}^{\text{ad}} \rightarrow \nu(W) \times \frak{B}^{**}_{\nu(W)}
\]
be the $SO(3)$-bundle defined as in p.176 of \cite{DK}. 
A calculation similar to that in Section 5.2 of \cite{DK} shows that
\[
     c_1( r_0^* {\frak L}_{\nu(M_1)}) 
     = -\frac{1}{4} p_1(\mathbb{P}^{\text{ad}}) / [M_1]
     = - \frac{1}{4} p_1( \mathbb{P}^{\text{ad}} ) / [N_1]
     = c_1(  r_1^* {\frak L}_{\nu(N_1)} ) 
\]
in $H^2(\frak{B}^{**}_{\nu(W)}; \Q)$,  where $r_0 : \frak{B}^{**}_{\nu(W)} \rightarrow {\frak B}^{*}_{\nu(M_1)}$, $r_1 : \frak{B}^{**}_{\nu(W)} \rightarrow  {\frak B}^{*}_{\nu(N_1)}$ are the restriction maps $r_0([A]) = [A|_{\nu(M_1) \times \{  0 \}}]$,  $r_1([A]) = [A|_{\nu(N_1) \times \{ 1 \}}]$. As we have seen, $\pi_1( {\frak B}_{\nu(W)}^{**} ) = 1$ and hence $H^2({\frak B}^{**}_{\nu(W)};\Z)$ is torsion free.  Therefore
\[
       c_1( r_0^* {\frak L}_{\nu(M_1)}) = c_1(  r_1^* {\frak L}_{\nu(N_1)} ) \ \text{in $H^2({\frak B}^{**}_{\nu(W)}; \Z)$},
\]
and $r_0^* {\frak L}_{\nu(M_1)}$ and $r_1^* {\frak L}_{\nu(N_1)}$ are isomorphic to each other.  Fix an isomorphism $r_0^* {\frak L}_{\nu(M_1)} \cong r_1^* {\frak L}_{\nu(N_1)}$. Then we can take a generic section
\[
       H' : {\frak B}_{\nu(W)}^{**} \times [0, 1] \rightarrow r_0^* {\frak L}_{\nu(M_1)}
\]
with
\[
     H'(\cdot, 0) = r_0^* s_{M_1}, \
     H'(\cdot, 1 ) = r_1^* s_{N_1}. 
\]
Put $V_{W} := (H')^{-1}(0)$. 
The usual dimension counting argument shows that 
\[
             ( V(\hat{E}; M_2, \dots, M_d) \times [0, 1] ) \cap V_{W}
\]
is compact and a bordsim between $V(\hat{E}; M_1, M_2, \dots, M_d)$ and $V(\hat{E}; N_1, M_2, \dots, M_d)$.

\subsubsection{Invariance under change of representatives $V_M$: The case $\dim M = 3$}

We will consider the case $\dim M_1 = 3$.   Suppose $[M_1] = [N_1]$ in $H_3(X;\Z)$. We have smooth maps  $f_0, f_1 : X \rightarrow S^1$ with $PD([M_1]) = [f_0] = [f_1]$ in  $H^1(X;\Z) \cong [X, S^1]$ such that $0 \in S^1$ is a regular value of $f_0, f_1$ and $M_1 = f^{-1}(0), N_1 = f_1^{-1}(0)$.  Since $[f_0] = [f_1]$ in $[X, S^1]$ we have a smooth map
\[
       H : X \times [0, 1] \rightarrow S^1
\]
with $H(\cdot, 0) = f_0, H(\cdot, 1) = f_1$ such that $0 \in S^1$ is a regular value of $H$.  Put
\[
      W = H^{-1}(0) \subset X \times [0, 1].
\]
Then $W$ is a submanifold in $X \times [0, 1]$ with $\partial W = M_1 \times \{ 0 \} \coprod N_1 \times \{ 1 \}$.  We have the Chern-Simons functionals 
\[
      \begin{split}
         & CS_{M_1} : {\frak B}^*_{\nu(M_1)} \rightarrow S^1,  \ 
             \widetilde{CS}_{M_1} : \widetilde{\frak B}_{\nu(M_1)}^* \rightarrow S^1, \\
         & CS_{N_1} : {\frak B}^*_{\nu(N_1)} \rightarrow S^1, \ 
             \widetilde{CS}_{N_1} : \widetilde{\frak B}_{\nu(N_1)}^* \rightarrow S^1.
      \end{split}
\] 
Let $\tilde{r}_0 : \widetilde{\frak B}^{**}_{\nu(W)} \rightarrow  \widetilde{\frak B}^*_{\nu(W)}$, $\tilde{r}_1 : \widetilde{\frak B}_{\nu(W)}^{**} \rightarrow \widetilde{\frak B}_{\nu(N_1)}^*$ be the restriction maps and  $\widetilde{\mathbb{P}} \rightarrow \nu(W) \times \widetilde{\frak B}^{**}_{\nu(W)}$ be the universal $SU(2)$-bundle.  We have
\[
          [\tilde{r}_{0}^* \widetilde{CS}_{M_1}] 
          = c_2(\widetilde{\mathbb{P}} )/ [M_1] 
          = c_2(\widetilde{\mathbb{P}} )/ [N_1]
          = [\tilde{r}_{0}^* \widetilde{CS}_{N_1}] 
\]
in $H^1(\widetilde{\frak B}_{\nu(W)}^{**};\Z) = [\widetilde{\frak B}_{\nu(W)}^{**}, S^1]$. See p.178 of \cite{DK}.  The homomorphism $H^1({\frak B}_{\nu(W)}^{**}; \Q) \rightarrow H^1(\widetilde{\frak B}^{**}_{\nu(W)};\Q)$ induced by the base point fibration is injective (see the proof of Proposition 5.1.15) and $H^1({\frak B}^{**}_{\nu(W)}; \Z)$ is torsion free by the universal coefficient theorem. These facts imply that
\[
       [r^*_0 CS_{M_1}] = [r_1^* CS_{N_1}]  \ 
       \text{in $H^1({\frak B}_{\nu(W)}^{**}; \Z) = [{\frak B}^{**}_{\nu(W)}, S^1]$}. 
\]
Therefore we have a homotopy
\[
       H' : {\frak B}^{**}_{\nu(W)} \times [0, 1] \rightarrow S^1
\]
with
\[
     H'(\cdot, 0) = r^*_0 CS_{M_1}, \ 
     H'(\cdot, 1) = r^*_1 CS_{N_1}. 
\]
We may suppose that $0 \in S^1$ is a regular value of $H'$. Put
\[
       V_{W} := (H')^{-1}(0). 
\]
Then we can show that
\[
         \{ V(\hat{E}; M_2, \dots, M_d) \times [0, 1] \} \cap V_{W}
\]
is compact and a bordism between $V(\hat{E};M_1, M_2, \dots, M_d)$ and $V(\hat{E}; N_1, M_2, \dots, M_d)$.  


Take integral homology classes $[M_1], \dots, [M_d] \in H_*(X;\Z)$ for $* \leq 3$.     If $[M_j] \in H_0(X;\Z)$, assume that $[M_j]$ is the generator of $H_0(X;\Z)$. 
If  we have an $SU(2)$-bundle $E$ with: 
\[
         r: =  \dim V(\hat{E}; M_1, \dots, M_d) \leq 2, 
\]
then $V(\hat{E}; M_1, \dots, M_d)$  defines the classes $[V(\hat{E}; M_1, \dots, M_d)]$ in $H_{r}(\frak{B}^*(\hat{E});\Z)$ and $ \Omega_{r}^{SO}(\frak{B}^*(\hat{E}))$ by proposition \ref{prop V compact}, 
 which depends only on their homology classes $[M_j]$.

 \subsection{Definition of $\Psi_X$}
 
We extend the definition of the class $[V(\hat{E};M_1, \dots, M_{d})]$ to $H_0(X;\Z)$ linearly. For example,    suppose that $[M_1] = n_1 [pt], \dots, [M_{d_0}] = n_{d_0} [pt] \in H_0(X;\Z)$ where $[pt]$ is the generator of $H_0(X;\Z)$ and that $[M_{d_0+1}], \dots, [M_{d}]$ are degree $1$ or larger.  Then we put
\begin{equation}   \label{eq V H_0}
      [V(\hat{E}; M_1, \dots, M_d) ] := n_1 \dots n_{d_0} [V(\hat{E}; pt, \dots, pt, M_{d_0+1}, \dots, M_{d})].
\end{equation}

Note that if $E$ is an $SU(2)$ vector bundle with
\[
              \dim \frak{M}(E) = \sum_{j=1}^{d} (4 - \dim M_j), 
\]
we have
\[
       [ V(\hat{E}; \Sigma_0, M_1, \dots, M_d) ] \in H_0(\frak{B}^*(\hat{E});\Z) \cong \Z
\]
because of the formula (\ref{eq:dim.moduli.blowup}). 
Here $[\Sigma_0]$ is the Poincare dual of the standard generator $e$ of $H^2(\overline{ \mathbb{CP} }^2;\Z)$.

\begin{definition} \label{def D inv}
We define the Donaldson invariant by:
\begin{align*}
      \Psi_{X}([M_1], \dots, [M_d]) 
      &   = \frac{1}{2 \cdot 4^{d_0}} [V(\hat{E}; \Sigma_0, M_1, \dots, M_d)]  \in \Q 
\end{align*}
Here $d_0$ is the number of homology classes $[M_j]$ of degree $0$.  
If there is no $SU(2)$ vector bundle $E$ with $\dim {\frak M}(E) = \Sigma_{j=1}^{d} (4- \dim M_{j})$, 
then we define $\Psi_{X} ([M_1], \dots, [M_d])$ to be $0$. 
\end{definition}

\begin{remark}\label{blow-up}
The factor $\frac{1}{4}$ in the definition of $\Psi_X$ comes from the  fact that $V_{M_j}$ is dual to $4\mu([M_j])$ rather than $\mu([M_j])$ if $[M_j] \in H_0(X;\Z)$.   See remark \ref{rem V_x}.

The factor $\frac{1}{2}$  comes from the blow-up formula. 
 In general the Donaldson invariant $\Psi_{X}$  can not be  defined by a straightforward way,
 since  reducible flat connections affect its construction.
 On the other hand when  it is surely defined,  we have the following formula:
$$\Psi_{X} ([M_1], \dots, [M_d]) = \frac{1}{2} \Psi_{\hat{X}}([\Sigma_0], [M_1], \dots, [M_d])  $$
where $\hat{X}$ is the blow up of $X$.
Even though  the left hand side is not defined in general,
 the right hand side is  always defined as we have described.
So, it is natural to put  $\Psi_{X}([M_1], \dots,[M_d]) $ to be the half of $\Psi_{\hat{X}}([\Sigma_0], [M_1], \dots, [M_d])$.  
\end{remark}

We will make one more remark on the Donaldson invariant. 

\begin{remark} \label{Psi hom}
The Donaldson invariant $\Psi_X$ is a homomorphism.  On $H_0(X;\Z)$, it is easy to see that $\Psi_X$ is a homomorphism from (\ref{eq V H_0}).   
Suppose that $1 \leq \dim M_1 \leq 3$ and that there is an $SU(2)$ vector bundle $E$ with $\dim {\frak M}(E) = \sum_{j=1}^{d} (4 - \dim M_j)$.  Note that
\[
        1 \leq \dim V(\hat{E}; \Sigma_0, M_2, \dots, M_d) \leq 3. 
\]
By proposition \ref{prop V compact},  $V(\hat{E}; \Sigma_0,  M_2, \dots, M_d)$ is compact. Therefore we can write
\[
     \Psi_{X}([M_1], \dots, [M_d]) =
     \frac{1}{2 \cdot 4^{d_0} }\left<  \mu([M_1]), [V (\hat{E}; \Sigma_0, M_2, \dots, M_d) ]   \right>.
\]
Since $\mu$ is a homomorphism, $\Psi_X$ is a homomorphism on $H_*(X;\Z)$ with $1 \leq * \leq 3$. 

Recall that if $\dim M_j = 1$, the class $[V(\hat{E}; M_1, \dots, M_d)]$ depends on the homotopy class $[M_j]$ in $\pi_1(X)$  rather than the homology class in $H_1(X; \Z)$. (See proposition \ref{prop V compact}). 
Since the values of $\Psi_X$ are in $\Q$ which is commutative and $\Psi_X$ is a homomorphism, $\Psi_X$ descends to a map on $H_1(X;\Z)$.  
\end{remark}

We have a natural sum decomposition:
\[
     \Psi_{X} = \sum_{k \in \Z_{>0}} \Psi_{X, k},
\]
where $\Psi_{X, k} : A(X) \rightarrow \Q$ is defined by using the $U(2)$-vector bundle $\hat{E}$ with $c_2 = k$ and with $c_1 = e$:
\[
  \begin{split}
    &  \Psi_{X, k}([M_1], \dots, [M_d]) =   \\
    &  \left\{ 
         \begin{array}{ll}
            \frac{1}{2 \cdot 4^{d_0}} [V(\hat{E}; \Sigma_0, M_1, \dots, M_d)] & \text{if $\displaystyle{8k -  3(1-b_1(X) + b^+(X)) = \sum_{j=1}^d (4 - \dim M_j)} $}  \\
            0 & \text{otherwise.}
         \end{array}
       \right.
   \end{split}
\]
Then we put:
\begin{equation}  \label{eq Psi odd}
     \Psi_{X, odd} := \sum_{k \in \Z_{ > 0 }, \  k \equiv 1  (2)} \Psi_{X, k}.
\end{equation}


\section{Twisted Donaldson invariant for commutative case}
\label{sec:family.det.line}

A unitary representation of the fundamental group can be used to twist
 connections over a $4$-manifold. In particular
the characters of the Picard torus $\Pic(B \pi_1(X))$
 give rise to a family of moduli spaces fibered over the Picard torus when $\pi_1(X)$ is free abelian group
 $\Gamma = \Z^m$.
 More generally, we can consider  a group homomorphism ${\bf f}:\pi_1(X) \rightarrow \Gamma$,
 when $\pi_1(X)$ is not necessarily free abelian.
By
 use of such structure,
we introduce twisted Donaldson invariants
 \begin{align*}
  &  \Psi_{X, {\bf f}}^{ \tw, (r)}  :  A(X) \otimes H_{odd}(X;\Z) \otimes H_1(B\Gamma;\Z) \rightarrow    H^*(\Pic(B\Gamma);\Q)
\end{align*}
for $r=1, 2, 3, 4$.
Here $A(X)$ is the domain of Donaldson invariants~(\ref{eq A(X)}).

We verify that this version of the invariants are non trivial in the sense
that they  can be used to  produce quite interesting properties on
 smooth structure  on four manifolds. So this is a model case and
 motivates us to generalize our construction to the
 case of
 non commutative fundamental groups, by introducing a framework of non commutative geometry
 in section~$5$.

Throughout Section $3$, we always assume that $X$ is a closed, spin and  smooth $4$-manifold with $b^+(X) > 1$.
The reason why we need spin condition will be explained in subsection $3.1$
below.

\subsection{Rank 1 case}
Put $ \Gamma := \Z$ and take a homomorphism ${\bf f} : \pi_1(X) \rightarrow \Gamma$.
Below we introduce a  twist of the $\mu$-map
\[
    \mu^{\tw}_{\bf f} : H_{odd}(X;\Z) \rightarrow \    \lim_{ \substack{\longleftarrow \\ B} } H^*(B \times \Pic ( B\Gamma ) ;\Q)
\]
using the family of  flat connections, where $B$ runs over compact subsets in $\frak{B}^{*}(\hat{E})$.

\vspace{3mm}

We denote the Picard torus of $B\Gamma$ by $\Pic(B\Gamma)$
\[
     \Pic(B\Gamma) = H^1(B\Gamma;\R) / H^1(B\Gamma; \Z)  \cong S^1.
\]
Let us consider the universal line bundle
 $${\mathbb L} \rightarrow B\Gamma \times \Pic (B\Gamma)$$
   which consists of  the family $\{ L_{\rho} \}_{ \rho \in \Pic(B\Gamma) }$
   of flat line bundles on $B\Gamma =S^1$
   where $L_{\rho} := {\mathbb L}|_{ B\Gamma \times \{ \rho \} }$ is the restriction.
Its  first Chern class is given by
\begin{align*}
        c_1({\mathbb L})     = \omega \cup \eta \  \in  & \ H^1(B\Gamma;  \Z)   \otimes H^1(\Pic(B \Gamma);   \Z)  \\
                                                                            = & H^2(B\Gamma \times \Pic(B\Gamma); \Z),
\end{align*}
where  $\omega \in H^1(B\Gamma;  \Z)$ and $\eta \in H^1(\Pic (B\Gamma);  \Z)$
   are the generators respectively.
\vspace{3mm}

Take an $SU(2)$-vector bundle $E$ on $X$ with $c_2(E)$ odd.   To avoid difficulties  caused by reducible flat connections, we use the $U(2)$-vector bundle $\hat{E}$
 on the blow-up $\hat{X} = X \# \overline{\mathbb{CP}}^2$ as in (\ref{eq:X.blowup}) with the conditions~(\ref{eq:cond.blowup}).    

 Let us describe the construction of
 the {\em universal bundle}
  $$\hat{\mathbb E} \rightarrow \hat{X} \times {\frak B}^*(\hat{E})$$
  of $\hat{E}$, following  the discussion of Section 8.3.5 of \cite{DK}.
 Consider  the bundle
\[
       \widetilde{ \hat{\mathbb E} } :=  \hat{E} \times_{ \frak{G}^0(\hat{E}) } {\frak A}(\hat{E})
       \rightarrow
       X \times \widetilde{\frak B}(\hat{E})
\]
where ${\frak A}(\hat{E})$ is the space of $U(2)$-connections $A$ on $\hat{E}$ with a fixed determinant
$\det A = a_0$.
${\frak G}^0(\hat{E})$ is the gauge group  on  $\hat{E}$ whose elements $u$ satisfy $\det(u) =1$
and  $u(x_0) = 1$ for a fixed point $x_0$.

Let ${\frak c}$ be the spin$^c$ structure on $\hat{X}$ with $c_1({\frak c}) = - e$. Coupling the spin$^c$
Dirac operator over $\hat{X}$ with $\widetilde{\hat{\mathbb E}}$,
we obtain  the family ${\mathbb D}_{\hat{X}}$ of the twisted spin$^c$ Dirac operators parametrized by
 $\widetilde{\frak B}(\hat{E})$.
We can take  the real part of the operators $({\mathbb D}_{\hat{X}})_{\R}$ and obtain  the real line bundle
$\widetilde{\hat{\Lambda}} = \det \Ind ({\mathbb D}_{\hat{X}})_{\R}$
on $\widetilde{\frak B}(\hat{E})$. See Section 5 of \cite{AMR}.

\begin{lemma} \label{lem universal bundle}
Suppose $X$ is spin and $c_2(\hat{E})$ is odd.
Then $\widetilde{\hat{\mathbb E}} \otimes_{\R} \widetilde{\hat{\Lambda}}$ descends to a vector bundle
$$\hat{\mathbb E} \to \hat{X} \times {\frak B}^*(\hat{E})$$
\end{lemma}

\begin{proof}
Note that the numerical index of the Dirac operators is odd. (See p.75 of \cite{AMR}.)
 This implies  that $\widetilde{\hat{\mathbb E}} \otimes_{\R} \widetilde{\hat{\Lambda}}$ descends to a vector bundle on $\hat{X} \times \frak{B}^*(\hat{E})$ (see proposition 8.3.15 of \cite{DK}).
 \end{proof}

 This is the point where we have to assume that $X$ is spin.   (See also remark \ref{Rem tw mu} below.)
 Later on we will always assume that all four manifolds are
  spin, otherwise stated.

 \vspace{3mm}

The homomorphism
${\bf f}: \pi_1(X) = \pi_1(\hat{X}) \rightarrow \Gamma = \Z$ induces a continuous map
\[
     \hat{f} : \hat{X} \rightarrow B\Gamma \cong S^1
\]
which is unique up to homotopy.
 Let us introduce a {\em twisted universal bundle}
\[
       \hat{\mathbb E}^{\tw}_{\bf f} :=  \hat{\mathbb E} \boxtimes_{\hat{X}} ( \hat{f} \times id)^*
       {\mathbb L} \boxtimes_{\Pic}
        {\mathbb L} \rightarrow
                                        \hat{X} \times B\Gamma \times {\frak B}^*(\hat{E}) \times \Pic(B\Gamma)
\]
as a family of vector bundles with connections on $\hat{X} \times B\Gamma$ parametrized by ${\frak B}^*(\hat{E}) \times \Pic(B\Gamma)$.

Firstly consider
 the fibered product
\begin{align*}
&  \hat{\mathbb E } \boxtimes_{\hat{X}}  (\hat{f} \times id)^* {\mathbb L}
\rightarrow     \hat{X} \times  {\frak B}^*(\hat{E}) \times \Pic(B\Gamma)
\end{align*}
which is given by the  family of twisted $U(2)$ bundles $(\hat{E} \otimes L_{\rho} , [A_{\rho}] )$
with twisted connections at $([A], {\rho}) \in  {\frak B}^*(\hat{E}) \times \Pic(B\Gamma)$.

Next let
$ ( \hat{f} \times id)^* {\mathbb L} \boxtimes_{\Pic} {\mathbb L}  \to \hat{X} \times B \Gamma \times \Pic(B\Gamma)$
be the fibered product,
which is given by the family of flat line bundles $\hat{f}^* L_{\rho}  \boxtimes L_{\rho}$ at $\rho \in \Pic(B\Gamma)$.

Then the  twisted universal bundle is given by combination of these two constructions.

Take a homology class $[M] \in H_{odd}(X;\Z)$ and realize it
 by a submanifold  $M \subset X$.  We may regard $M$ as a submanifold in $\hat{X}$.
The restriction
 $$\hat{\mathbb E}_{\bf f}^{\tw}|_{ M \times  B\Gamma \times \frak{B}^*(\hat{E}) \times \Pic(B\Gamma)}$$
  gives a family of vector bundles with connections on $M \times B\Gamma$ parametrized by ${\frak B}^*(\hat{E}) \times \Pic(B\Gamma)$.

 $M \times B\Gamma$ admits a spin structure
since $\dim M \leq 3$. So choose  and fix
it.  Coupling the Dirac operator on $M \times B\Gamma$ with $\hat{\mathbb E}_{\bf f}^{\tw}|_{M \times B\Gamma \times \frak{B}^*(\hat{E}) \times \Pic(B\Gamma) }$,  we get a family of the Dirac operators
${ \mathbb D }_{M \times B\Gamma, {\bf f}}$  parametrized by ${\frak B}^*(\hat{E}) \times \Pic(B\Gamma)$.

 Over a compact subset $B \subset {\frak B}^*(\hat{E})$, consider  the index bundle
 of the family of the Dirac operators
\[
         \Ind_{B} {\mathbb D}_{M \times B\Gamma, {\bf f}} \  \in \  K^*(B \times \Pic( B\Gamma)).
\]
It is compatible with the restriction in $K$ theory
\[
        \Ind_{B}{\mathbb D}_{M \times B\Gamma, {\bf f}} =  \Ind_{B'}{\mathbb D}_{M \times B\Gamma, {\bf f}}|_{ B \times \Pic(B\Gamma) }
\]
for $B \subset B'$.

\begin{definition}
The twisted $\mu$-map
$\mu^{\tw}_{\bf f}([M])$  is given  by the collection
$$\mu^{\tw}_{\bf f}([M]) \equiv \{ \ch( \Ind_{B}{\mathbb D}_{M \times B\Gamma, {\bf f}} ) \}_{B}  \ \in \
\displaystyle \lim_{\substack{\longleftarrow \\ B} } H^*(B \times \Pic(B\Gamma);\Q).$$
Here $B$ runs over compact subsets in $\frak{B}^{*}(\hat{E})$.
\end{definition}

We have the following.

\begin{lemma} \label{lem tw mu M}
Under the above situation,  the formula
$$\mu^{\tw}_{\bf f}([M]) = \ch( \hat{\mathbb E}_{\bf f}^{\tw} ) /[M \times B\Gamma]$$
holds, where
 the right hand side is the image of $\ch( \hat{\mathbb E}_{\bf f}^{\tw} ) /[M \times B\Gamma]$ by the natural map
 $$H^*({\frak B}^*(\hat{E}) \times \Pic(B\Gamma); \Q) \rightarrow \displaystyle \lim_{\substack{\longleftarrow \\ B} } H^*(B \times \Pic(B\Gamma);\Q).$$

 In particular $\mu^{\tw}_{\bf f}([M])$ is independent   of choices of representative $M$
   of  the homology class $[M] \in H_{odd}(X;\Z)$ and spin structure on $M \times B\Gamma$.
\end{lemma}

\begin{proof}
For each compact subset $B \subset {\frak B}^*(\hat{E})$,
\[
  \begin{split}
    \ch( \Ind_{B}{\mathbb D}_{M \times B\Gamma, {\bf f}} )
    & =
    \ch( \hat{\mathbb E}_{\bf f}^{\tw} |_{M \times B\Gamma \times B \times \Pic(B\Gamma)}) \hat{A}(M \times B\Gamma) / [M \times B\Gamma]  \\
    & = \ch( \hat{\mathbb E}^{\tw}_{\bf f} |_{M \times B\Gamma \times B \times \Pic(B\Gamma)}) / [M \times B\Gamma]
 \end{split}
\]
by the Atiyah-Singer index theorem for family.
Here we have used the fact that $\hat{A}(M \times B\Gamma) = \hat{A}(M) \hat{A}(B\Gamma) = 1$ since $\dim M \leq 3$.
\end{proof}

\begin{remark}
For $[M] \in H_{even}(X;\Z)$, we can define $\mu^{\tw}([M])$ as above. But $\mu^{\tw}_{\bf f}([M]) = 0$ since the dimension of $M \times B\Gamma$ is odd in this case.
\end{remark}

Let us introduce the twisted Donaldson invariant below.
     To do this, we compute the following.

     Recall that $\mu$ is the standard  $\mu$ map in section $2$.

\begin{lemma} \label{lem tw mu}
For $\alpha \in H_{odd}(X;\Z)$, we have
\[
      \mu^{\tw}_{\bf f}(\alpha)
        = (\ch( \hat{\mathbb E})/\alpha) \cup \eta
        = - \mu(\alpha) \cup \eta + ( \ch_{\geq 4}( \hat{\mathbb E}) / \alpha ) \cup \eta
\]
where $\eta$ is the generator of $H^1(\Pic(B\Gamma);\Z)$.
\end{lemma}

 In particular the following two cases happen:

 \vspace{2mm}

 $(1)$
 \[
       \mu^{\tw}_{\bf f}(\alpha) / c = 0
\]
if
 $\alpha \in H_{odd}(X;\Z)$ and  $c \in H_2( \frak{B}^*(\hat{E});\Z )$, or if  $\alpha \in H_3(X;\Z)$ and
 $c \in H_3(\frak{B}^*(\hat{E});\Z)$.

  \vspace{2mm}

$(2)$
 \[
       \mu^{\tw}_{\bf f}(\alpha) / c = - ( \mu(\alpha)/c ) \cup \eta
\]
if $\alpha \in H_1(X;\Z)$ and $c \in H_3(\frak{B}^*(\hat{E});\Z)$.

\begin{proof}
Take  a homology class
$\alpha = [M] \in H_{odd}(X;\Z)$, and realize it by    a submanifold $M \subset X$.
Regard it  as a submanifold in $\hat{X}$.
By Lemma \ref{lem tw mu M},
\begin{equation}  \label{eq mu tw}
   \begin{split}
      \mu^{\tw}_{\bf f}(\alpha)
      &= \ch(\hat{\mathbb E}_{\bf f}^{\tw})/[M \times B\Gamma]   \\
      &= \ch(\hat{\mathbb E}|_{ M \times {\frak B}^*(\hat{E}) }) \ch( (\hat{f} \times id)^* {\mathbb L}) \ch({\mathbb L}) / [M \times B\Gamma]
  \end{split}
\end{equation}
where $\hat{f} : \hat{X} \rightarrow B\Gamma$ is the continuous map.
Let us verify
\begin{equation} \label{eq c_1 hat E}
       c_1(\hat{\mathbb E}|_{ M \times {\frak B}^*(\hat{E}) }) = 0.
\end{equation}
To see this, it is sufficient to check  that $\hat{\mathbb E}|_{ M \times {\frak B}^*(\hat{E}) }$
has an $SU(2)$-vector bundle structure.
Recall that $\widetilde{\hat{\mathbb E}} \otimes_{\R} \widetilde{\hat{\Lambda}}$
descends to  $\hat{\mathbb E}$ and the structure group of $\widetilde{\hat{\Lambda}}$
 is $\{ \pm 1 \}$. Hence it is sufficient to verify
  that $\widetilde{\hat{\mathbb E}}|_{M \times \widetilde{\frak B}(\hat{E})}$ is an $SU(2)$-vector bundle.
  Let $P \rightarrow X$ be the $SU(2)$-principal bundle associated to $E$.   Put
  a principal $SU(2)$-bundle
\[
       \widetilde{\mathbb P}_{M} := (P|_{M}) \times_{ \frak{G}^0(\hat{E}) } \frak{A}(\hat{E})
     \   \to  \   M \times \widetilde{\frak B}(\hat{E}).
\]
Note that $\hat{E}|_{M} = E|_{M}$, and hence the vector bundle associated to $\widetilde{\mathbb P}_{M}$ is
$\widetilde{\hat{\mathbb E}}|_{ M \times \widetilde{\frak B}(\hat{E}) }$.
So  $\widetilde{\hat{\mathbb E}}|_{M \times {\frak B}^*(\hat{E})}$ has an $SU(2)$-vector bundle structure.

We also have
\begin{equation} \label{eq c_1 L}
   c_1({\mathbb L}) = \omega \cup \eta,  \   \  c_1( {\mathbb L})^2 = 0, \  \ c_3(\hat{\mathbb E}) = 0,
\end{equation}
where $\omega \in H^1(B\Gamma;\Z), \eta \in H^1(\Pic(B\Gamma);\Z)$ are the generators.
Putting (\ref{eq c_1 hat E})  and (\ref{eq c_1 L})  together
 into (\ref{eq mu tw}), we obtain
\[
   \begin{split}
      \mu^{\tw}_{\bf f}(\alpha) &= - (c_2( \hat{ \mathbb{E} } ) / \alpha ) \cup \eta + ( \ch_{\geq 4}(\hat{\mathbb{E}}) / \alpha ) \cup \eta \\
                                  &= - \mu(\alpha) \cup \eta + ( \ch_{\geq 4}(\hat{\mathbb{E}}) / \alpha ) \cup \eta.
   \end{split}
\]
Here we have used the equality  $c_2( \hat{\mathbb{E}}) = -\frac{1}{4} p_1 (\frak{g}_{\hat{\mathbb{E}}})$.
\end{proof}

\begin{remark} \label{rm tw mu rank 1}
The formula in  Lemma \ref{lem tw mu}
 verifies that $\mu^{\tw}_{\bf f}(\alpha)$ is actually independent  of ${\bf f}$.
On the other hand
$\mu^{\tw}_{\bf f}(\alpha)$ does depend on ${\bf f}$  in the higher rank case.
See section \ref{subsection higher rank} below.
\end{remark}

\subsubsection{Construction of $\Psi_{X, {\bf f}}^{\tw, (1)}$ and $\Psi^{\tw, (2)}_{X, {\bf f}}$}

Take $[M_1], \dots, [M_d] \in H_{\leq 3}(X;\Z)$, where $M_j$ are submanifolds in $X$.  Suppose that we have an $SU(2)$-vector bundle $E$ with $c_2(E)$ odd and with $\dim \frak{M}(E) = \sum_{j=1}^{d} ( 4 - \dim M_j) + 1$.
Let $\hat{E}$ be the $U(2)$-vector bundle on $\hat{X} = X \# \overline{\mathbb{CP}}^2$
with conditions ~(\ref{eq:cond.blowup}) and the relations on dimentions of moduli spaces~(\ref{eq:dim.moduli.blowup}).

Let
 $\Sigma_0$ be  an embedded surface in $\hat{X}$ with $[\Sigma_0] = PD(e)$ where $e \in H^2(\overline{ \mathbb{CP} }^2;\Z)$ is  the standard generator.
Consider the situation in  proposition~\ref{prop V compact} with $r=1$, and so
$V(\hat{E}; \Sigma_0, M_1, \dots, M_d)$
 defines  a homology class in $H_1({\frak B}^*(\hat{E});\Z)$.

 Recall that $A(X)$ is the domain of the Donaldson invariant.

\begin{definition}
We define
\begin{align*}
  & \Psi_{X, {\bf f}}^{\tw, (1)} : A(X) \otimes H_{odd}(X;\Z) \rightarrow H^*(\Pic(B\Gamma);\Q), \\
&
      \Psi_{X, {\bf f}}^{\tw, (1)}([M_1], \dots, [M_d]; \alpha)
      = \frac{1}{2 \cdot 4^{d_0}} \mu^{\tw}_{\bf f}(\alpha) / [V( \hat{E} ;\Sigma_0, M_1, \dots, M_d)]
\end{align*}
if there is an $SU(2)$-vector bundle $E$ over $X$ with $c_2(E)$ odd and with $\dim \frak{M}(E) = \sum_{j=1}^{d} (4 - \dim M_j)+1$. Here $d_0$ is the number of homology classes $[M_j]$ of degree $0$.

We put    $\Psi_{X, {\bf f}}^{\tw, (1)}([M_1], \dots, [M_d]; \alpha)  = 0$ otherwise.
\end{definition}

\begin{remark}
$(1)$ $\Psi^{\tw, (1)}_{X, {\bf f}}$ is independent of ${\bf f}$
when $\pi_1(X) \cong \Z$, as in Remark \ref{rm tw mu rank 1}.

\vspace{2mm}

\noindent
$(2)$
We can  define the  twisted Donaldson invariant
\[
       \Psi_{X, {\bf f}}^{\tw, (2)} :  A(X) \otimes H_{odd}(X;\Z) \rightarrow H^*(\Pic(B\Gamma);\Q)
\]
in the same way,
using an $SU(2)$ vector bundle $E$ on $X$ with $c_2(E)$ odd and with $\dim {\frak M}(E) = \sum_{j=1}^{d} (4- \dim M_j) +2$.  But this invariant is always trivial by the formula for $\mu^{\tw}_{\bf f}$ in lemma \ref{lem tw mu}.

\vspace{2mm}

\noindent
(3) As stated in  proposition \ref{prop V compact}, $[V(\hat{E}; \Sigma_0, M_1, \dots, M_d)] \in H_r({\frak B}^*(\hat{E});\Z)$ depends on the homotopy class $[M_j : S^1 \rightarrow X] \in \pi_1(X)$ with $\dim M_j = 1$.  So if $\pi_1(X)$ is not commutative, $H_1(X;\Z)$ in $A(X)$ should be replaced with $\pi_1(X)$.

\end{remark}

\vspace{2mm}
\subsubsection{Construction of $\Psi^{\tw, (3)}_{X, {\bf f}}$}

Let us try to define a twisted Donaldson invariant
\[
        \Psi_{X, {\bf f}}^{\tw, (3)} : A(X) \otimes H_{odd}(X;\Z) \rightarrow H^*( \Pic(B\Gamma); \Q)
\]
using  an $SU(2)$-vector bundle $E$ with $\dim \frak{M}(E) = \sum_{j=1}^{d} (4 - \dim M_j) + 3$.
We can not apply proposition \ref{prop V compact}  straightforwardly
to the $3$ dimensional
 manifold $V(\hat{E}; \Sigma_0, M_1, \dots, M_d)$.
 Actually   it  may not define a homology class in  $H_3({\frak B}^*(\hat{E});\Z)$ which is  independent of $g$
 (see the last argument in the proof of proposition \ref{prop V compact}).

    To avoid this issue, we use  a submanifold in ${\frak B}^*(\hat{E})$ dual to
     (the relevant part of)  $\mu^{\tw}_{\bf f}(\alpha)$.
It follows from  lemma \ref{lem tw mu} that
  $\Psi_{X, {\bf f}}^{\tw, (3)}([M_1], \dots, [M_d], \alpha)$ should be of the form for $\alpha \in H_1(X;\Z)$
\[
          -  \frac{1}{2 \cdot 4^{d_0}} \left<  \mu(\alpha), [V(\hat{E}; \Sigma_0, M_1, \dots, M_d)] \right>  \eta
          \ \in  \ H^1(\Pic(B\Gamma); \Q),
\]
where $d_0$ is the number of homology classes $[M_j]$ of degree $0$.
As explained in Section \ref{sec:Review Yang-Mills}, a dual submanifold $V_{\ell}$ of $\mu(\alpha)$ in $\frak{B}^*(\hat
{E})$ is obtained
 by the map defined by holonomies around $\ell$,
 where $\ell$ is a loop representing $\alpha$.  By proposition \ref{prop V compact},
  the intersection
\[
V(\hat{E}; \Sigma_0, M_1, \dots, M_d, \ell) :=
      V_{\ell} \cap V(\hat{E};\Sigma_0, M_1, \dots, M_d)
\]
defines a well-defined $0$-dimensional homology class $[V(\hat{E}; \Sigma_0, M_1, \dots, M_d, \ell)] \in H_0( \frak{B}^*(\hat{E});\Z ) \cong \Z$.   So it is natural to define
\[
          \Psi_{X, {\bf f}}^{\tw, (3)}([M_1], \dots, [M_d]; \alpha) =
          - \frac{1}{2 \cdot 4^{d_0}} [V(\hat{E}; \Sigma_0, M_1, \dots, M_d, \ell)]  \eta.
\]

On the other hand,  for $\alpha \in H_3(X;\Z)$, we put $\Psi_{X}^{\tw, (3)}([M_1], \dots, [M_{d}],  \alpha) = 0$ by Lemma \ref{lem tw mu}.

 \begin{definition}
We define
 \begin{align*}
  &  \Psi_{X, {\bf f}}^{ \tw, (3) }  :  A(X) \otimes H_{odd}(X;\Z) \rightarrow    H^*(\Pic(B\Gamma);\Q) , \\
&
  \Psi_{X,{\bf f}}^{\tw, (3)}([M_1], \dots, [M_d]; \alpha) =
            -\frac{1}{2 \cdot 4^{d_0} }  [V(\hat{E}; \Sigma_0, M_1, \dots, M_d, \ell)]  \eta
\end{align*}
if $\alpha = [\ell] \in H_1(X;\Z)$ and there is an $SU(2)$-vector bundle $E$ with $c_2(E)$ odd and with $\dim \frak{M}(E) = \sum_{j=1}^{d} ( 4 - \dim M_j) +3$.

We put $\Psi_{X, {\bf f}}^{\tw, (3)}([M_1], \dots, [M_d]; \alpha) = 0$ otherwise.
  \end{definition}

\begin{remark} \label{Rem tw mu}
\begin{enumerate}[(i)]
\item
For $r = 1, 3$ we can write
\[
     \Psi_{X, {\bf f}}^{\tw, (r)}([M_1], \dots, [M_d]; \alpha) = - \Psi_{X, odd}([M_1], \dots, [M_d], \alpha) \eta
\]
if $\alpha \in H_{4-r}(X;\Z)$ and $\Psi_{X, {\bf f}}^{\tw, (r)}([M_1], \dots, [M_d]; \alpha) = 0$ if $\alpha \in H_r(X;\Z)$.
See (\ref{eq Psi odd}) for the definition of $\Psi_{X, odd}$.
Therefore $\Psi_{X, {\bf f}}^{\tw, (r)}$ coincides with  the known invariants of $X$.
But we  emphasize that we have obtained $\mu(\alpha) \cup \eta$ as (part of) the Chern character of the index of a family of Dirac operators. To define the non-commutative version of the twisted Donaldson invariants in section \ref{sec:twD-inv.nc},
it is necessary that the twisted $\mu$ map is defined as the Chern character of the index of a familiy of Dirac operators.

\item
Since $\Psi_{X, {\bf f}}^{\tw, (r)}$ are written in terms of the classical Donaldson invariant $\Psi_X$, $\Psi_{X, {\bf f}}^{\tw, (r)}$ are also  homomorphisms and descend to  maps on $H_1(X;\Z)$. See remark \ref{Psi hom}.

\item
It seems possible to generalize  our construction of the twisted Donaldson invariants over non-spin 4-manifolds $X$, using $U(2)$-vector bundles $E$ with $c_1(E) \equiv w_2(X) $ mod $2$ and with $c_2(E)$ odd.  We can
still construct the universal bundle on $X \times \frak{B}^*(E)$ of $E$ and proceed in a similar way.

 The construction of the twisted Donaldson invariants for a non-spin 4-manifold will be a bit more complicated. We have to choose a $U(2)$-bundle E with  $c_1(E) = w_2(X)$ mod $2$, and hence we need to fix an integral lift of $w_2(X)$.
In the case we have to check that the twisted Donaldson invariants are independent of the choice of the integral lift.
Also the formula of lemma \ref{lem tw mu} in the non-spin case is more complicated than the spin case.
\end{enumerate}

\end{remark}

\subsection{Non triviality of the invariants}

\label{ex Y S^1 times S^3}
Let us verify that our twisted invariants are non trivial by presenting
the following non trivial results by use of them:

\begin{theorem} \label{thm no connected sum}
Let $Y$ be a simply connected, spin, algebraic surface with $b^+ > 1$ and $m$ be a non-negative integer.
Then $X= Y \# (\#^{m} S^1 \times S^3 )$ cannot admit any connected sum decomposition as
$X = X_1 \# X_2$ with $b^+(X_i) >0$.
\end{theorem}

 Theorem \ref{thm no connected sum}  follows  from proposition \ref{Prop tw Psi} below
 with the Donaldson's fundamental
 result on non vanishing of the invariants for algebraic surfaces.

Let us say that $Y_1$ and $Y_2$ be an {\em exotic pair}, if they are homeomorphic but have
different smooth strurctures.

\begin{theorem} \label{thm exotic X_i}
Let
$Y_1$ and $Y_2$ be a pair of compact spin four manifolds with
$b^+(Y_i) >1$, which
are mutually homeomorphic.
Moreover assume that $\Psi_{Y_1, odd}$ is trivial,
but $\Psi_{Y_2, odd}$ is not trvial,
so that they consist of  an exotic pair.
Then the pair $( Y_1\# ( \#^m S^1 \times S^3), Y_2\# (\#^m S^1 \times S^3) )$ is also exotic.
\end{theorem}

 Theorem \ref{thm exotic X_i}  is a consequence of  proposition \ref{Prop tw Psi} and lemma \ref{lem vanishing 2}  below.
 Notice that the statement is quite unlikely to hold,
   if we replace $S^1 \times S^3$ by $S^2 \times S^2$.

In the case when $Y_1$ and $Y_2$ are simply connected and $m = 1$, theorem \ref{thm exotic X_i} follows from proposition in \cite[p.494]{Kawauchi}.
The authors would like to thank Kouichi Yasui for informing the authors about this.
Our result can be applied to the case when  $Y_{1}$, $Y_{2}$ are not simply connected and $m > 1$.


 \subsubsection{Computation of the invariants}

 Later on the rest of section $3$, we always assume that  all 4-manifolds $Y$ or $X$ are compact and spin with $b^+ >1$, without mention.

 \begin{proposition} \label{Prop tw Psi}
  Let us put $X := Y \# (\#^m S^1 \times S^3)$,
  where $m$ is a positive integer.
Then the  formula
\[
        \Psi_{X, {\bf f}}^{\tw, (3)}([M_1], \dots, [M_d], [\ell_1], \dots, [\ell_{m-1}]; [\ell_m]) = -\Psi_{Y, odd}([M_1], \dots, [M_d]) \eta
\]
holds, where  $[M_1], \dots, [M_d] \in A(Y)$,   $\ell_1, \dots, \ell_m$ are loops in $\#^m S^1 \times S^3$ which generate $H_1(\#^m S^1 \times S^3;\Z)$, and $\eta$ is the generator of $H^1(\Pic(B\Gamma);\Z) \cong \Z$.
\end{proposition}

\begin{proof}
We will give a proof for the case $m = 1$. We can prove the statement for the general case similarly.

Put $X = Y \# S^1 \times S^3$. Take an $SU(2)$-vector bundle $E$ on $X$ with $c_2(E)$ odd and homology classes $[M_1], \dots, [M_d]$ such that
\[
     \dim {\frak M}(E) = \sum_{j=1}^{d} (4 - \dim M_j) + 3.
\]
Let $\ell$ be a loop which represents a generator of $H_1(S^1 \times S^3;\Z)$.
Let us verify  the formula
\begin{equation}  \label{eq V Psi}
       \frac{1}{2}  \# V(\hat{E}; \Sigma_0, M_1, \dots, M_d, \ell )  = \Psi_{Y}([M_1], \dots, [M_d]).
\end{equation}
Here $[\Sigma_0] = PD(e)$ and $\hat{E}$ is the $U(2)$-vector bundle over $\hat{X} = X \# \overline{\mathbb{CP}}^2$ with $c_1(\hat{E}) = e $ and with $c_2(\hat{E}) = c_2(E)$ as before.  Recall that
\[
       V(\hat{E}; \Sigma_0, M_1, \dots, M_d, \ell) = V(\hat{E}; \Sigma_0, M_1, \dots, M_d) \cap V_{\ell}
\]
with $
            V_{\ell} = (h_{\ell}')^{-1}(1)$,
where $h_{\ell}'$ is a perturbation of the section $h_{\ell} : \frak{B}^*_{\nu(\ell)} \rightarrow \ad \widetilde{\frak{B}}^*_{\nu(\ell)}$ defined by holonomies around $\ell$.  (See section \ref{sec:Review Yang-Mills}.)

To verify  (\ref{eq V Psi}), take a Riemann metric $\hat{g}_{T}$ on $\hat{X}$
such that $\hat{X}$ has a long neck $S^3 \times [0, T]$ for $T > 0$
\[
         \hat{X} = ( \hat{Y} \setminus D_1) \cup (S^3 \times [0, T]) \cup ((S^1 \times S^3) \setminus D_2).
\]
Here $D_1, D_2$ are small 4-balls in $\hat{Y}$, $S^1 \times S^3$ with $D_2 \cap \ell = \emptyset$.
Take a sequence
$[A_{n}] \in \widetilde{V}(\hat{E}, \hat{g}_{T_n};\Sigma_0, M_1, \dots, M_d) \cap \tilde{h}_{\ell}^{-1}(1)$
with $T_{n} \rightarrow \infty$.
 Here $\tilde{h}_{\ell} : \widetilde{\frak B}_{\nu(\ell)} \rightarrow SU(2)$
is the unperturbed map defined by holonomies,
and $\widetilde{V}(\hat{E}, \hat{g}_{T_n}; \Sigma_0, M_1, \dots, M_d)$
is the pull-back of the intersection to the configuration space
$\widetilde{\frak B}(\hat{E})$ of framed connections.

A standard process verifies that it converges  after passing to a subsequence
\[
            [A_n] \rightarrow ([A], [B]) \in \widetilde{V}(\hat{F}; \Sigma_0, M_1, \cdots, M_d) \times ( \widetilde{\frak M}(F_0) \cap \tilde{h}_{\ell}^{-1}(1) ).
\]
Here $\hat{F}$ is the $U(2)$-bundle on $\hat{Y}$ with $c_2(\hat{F}) = c_2(\hat{E})$ and with $c_1(\hat{F}) = e$ and  $\widetilde{\frak M}(F_0)$ is the moduli space of framed flat connections on the trivial bundle $F_0$ on $S^1 \times S^3$.
 The map gives an identification
\[
     \tilde{h}_{\ell} :  \widetilde{{\frak M}}(F_0) \stackrel{\cong}{\rightarrow} SU(2).
\]
In particular, $\tilde{h}_{\ell}$ is transverse to $1$ without perturbation,
and hence
the intersection $\widetilde{V}(\hat{E}, \hat{g}_T; \Sigma_0, M_1, \dots, M_d) \cap \tilde{h}_{\ell}^{-1}(1)$ is transverse for large $T > 0$.
Moreover $[B]$ must be the gauge equivalent class of the trivial flat connection.
It follows from a well known argument
 that gluing ASD on $Y$ and the trivial flat connection on $S^1 \times S^3$ gives an $SO(3)$-equivariant identification
\[
      \widetilde{V}(\hat{E}, \hat{g}_T; \Sigma_0, M_1, \dots, M_d)  \cap \tilde{h}_{\ell}^{-1}(1) \cong
      \widetilde{V}(\hat{F}; \Sigma_0, M_1, \dots, M_d).
\]
Here we have used the fact that there is no obstruction to glueing instantons on $\hat{Y}$ and the trivial connection on $S^1 \times S^3$  since $b^+(S^1 \times S^3) = 0$.
   The above identification implies (\ref{eq V Psi}).

\end{proof}

\begin{proposition}\label{vanish}
Let us  fix a homomorphism   ${\bf f} : \pi_1(X) \to \Gamma := \Z$.
Suppose that  a connected sum decomposition
$X = X_1 \# X_2$ exists with $b^+(X_i) > 0$. Then the twisted Donaldson invariant $\Psi_{X, {\bf f}}^{\tw, (3)}$ is identically zero.
\end{proposition}

\begin{proof}
For $\alpha \in H_3(X;\Z)$, $\Psi_{X, {\bf f}}^{\tw, (3)}( [M_1], \dots, [M_d]; \alpha) = 0$ by definition.
We have to show that for $\alpha \in H_1(X;\Z)$, $\Psi_{X, {\bf f}}^{\tw,(3)}([M_1], \dots M_d;  \alpha) = 0$. We have
\[
  \begin{split}
      \Psi_{X, {\bf f} }^{\tw, (3)}([M_1], \dots, [M_d]; \alpha)
       &= -\Psi_{X, odd}([M_1], \dots, [M_d], \alpha) \eta \\
   \end{split}
\]
as explained in remark \ref{Rem tw mu}.
The Donaldson invariant $\Psi_{X}([M_1], \dots, [M_d], \alpha)$ is zero if $X = X_1 \# X_2$ with $b^+(X_j) > 0$ for $j = 1, 2$, hence the right hand side in the above formula is zero.  See \cite{MM} and theorem 6.4 of \cite{Donaldson Floer}.

\end{proof}

We also have the following vanishing result which is needed for the proof of theorem \ref{thm exotic X_i}.

\begin{lemma}  \label{lem vanishing 2}
Assume    $\Psi_{Y, odd}$ is trivial.
Then $\Psi_{X, {\bf f}}^{\tw, (3)}$ is trivial,
where $X = Y \# (\#^m S^1 \times S^3)$ and  $m$ is a positive integer.
\end{lemma}

\begin{proof}
The proof is an easy consequence of a standard dimension counting argument and proposition \ref{Prop tw Psi}.  We omit the detail.
\end{proof}

\subsubsection{Examples}  \label{subsection example rank 1}

Recall that we have assumed that
 all four manifolds $Y$ or $X$
 are compact and spin with $b^+ >1$, otherwise stated.

$(1)$
Let $X = Y \# (\#^m S^1 \times S^3)$ be as above. Moreover
assume that $Y$ is an algebraic surface. For example, we can take the hyperplane surface $S_d$ in $\mathbb{CP}^{3}$ of degree $d \geq 4$ with $d \equiv 0 $ mod $ 2$.  Then $\Psi_{X, {\bf f}}^{\tw, (3)}$ is non-trivial.

\vspace{1mm}

{\em
For $d \geq 6$ with $d$ even,  $S_d \# (\#^m S^1 \times S^3)$ is  homeomorphic to
  $(\#^k K 3) \# (\#^l S^2 \times S^2) \ \# \ (\#^m S^1 \times S^3)$  for some $k, l > 0$, but they are not diffeomorphic
  to each other.}

\vspace{1mm}

The argument goes as follows.
 It is known that $S_d$ is simply connected and spin for $d$ even.
 Let $h \in H_2(Y;\Z)$ be  a hyperplane class in $Y$.
 Then we have
 \[
     \Psi_{X, {\bf f}}^{\tw, (3)} \not= 0.
\]
 This follows from proposition  \ref{Prop tw Psi} with a fact
 $\Psi_{S_d}( \overbrace{h,\dots, h}^{n}) > 0$   for all  large $n$, where  $c_2(E)$ varies with respect to $n$
 (See Theorem C of \cite{Donaldson Polynomial}).
In particular $X$ does not admit a connected sum decomposition $X = X_1 \# X_2$ with $b^+(X_1), b^+(X_2) > 0$ by proposition \ref{vanish}.

If $d$ is even and if $d \geq 6$, the intersection form of $S_d$ is isomorphic to that of the connected sum $(\#^k K3) \# (\#^l S^2 \times S^2)$ of copies of $K3$ with copies of $S^2 \times S^2$ for some $k, l \in \Z_{> 0}$. (See p.13 of \cite{DK}.) Hence, by Freedman's theory,  $X = S_d \# (\#^m S^1 \times S^3)$ is homeomorphic to $(\#^k K 3) \# (\#^l S^2 \times S^2) \ \# \ (\#^m S^1 \times S^3)$.
On the other hand, it follows from  the above that
  $X$ can not be diffeomorphic to $(\#^k K 3) \#  (\#^l S^2 \times S^2) \# (\#^m S^1 \times S^3)$
  by proposition \ref{vanish}.

\vspace{3mm}

$(2)$
Let $E(n)_{p,q}$ be Elliptic surfaces over ${\mathbb C}P^1$ with the Euler number $\chi = 12n >0$
and multiple torus fibers  of multiplicities $p$ and $q$ respectively.

\vspace{1mm}

{\em
There are infinitely many $(p,q)$ such that $E(n)_{p,q} \# (\#^m S^1 \times S^3)$
 are all homeomorphic but have  different smooth structures from each other.}

\vspace{1mm}

If $n$ is even, if $p, q$ are odd and $\gcd(p, q) = 1$, then $E(n)_{p,q}$ is simply connected, spin with $b^+ =2n-1 >1$, and hence satisfy the conditions above. Moreover,  $E(n)_{p, q}$ and $E(n)_{p', q'}$ are homeomorphic to each other.
On the other hand, the Donaldson invariants of $E(n)_{p, q}$ are calculated and it shows that if $E(n)_{p, q}$ and $E(n)_{p', q'}$ are diffeomorphic to each other, then $p=p'$ and $q = q'$.  See \cite{FM, Lisca, MM2, MO}.

Using  the calculation of the Donaldson invariants and proposition \ref{Prop tw Psi},  if $(p, q) \not= (p', q')$,   we can see that  $E(n)_{p,q} \# (\#^m S^1 \times S^3)$ and $E(n)_{p', q'} \# ( \#^m S^1 \times S^3 )$ are homeomorphic but have  different smooth structures.

\vspace{3mm}

$(3)$
Using Seiberg-Witten theory, we can prove the same result  by a quite parallel argument as below.

Let $Y$ be a $4$-manifold with $b^+ (Y)> 0$ and take a spin${^c}$ structure of $Y$.
 For simplicity, suppose that  the dimension of the moduli space of solutions to the SW equations on $Y$ associated with the spin$^{c}$ structure is 0.
 Then the SW invariant  is the number of solutions with sign.
On $Y \# (S^1 \times S^3)$,
consider  the connected sum of the spin$^c$ structure  on $Y$
 with  the spin$^{c}$ structure of $S^1 \times S^3$.  (Note that $S^1 \times S^3$ has a unique spin-c structure up to isomorphisms.)
 Then the SW moduli space of $Y \#(S^1 \times S^3)$ has  dimension $1$.

For each connection of the determinant line bundle of the spinor bundle on $Y \# (S^1 \times S^3)$, we get the holonomy around the circle $S^1 \times \{ pt \}$. It induces a map from the SW moduli space on
$Y \# (S^1 \times S^3)$ to $U(1)$.   The inverse image of a generic point in $U(1)$ by this map is a finite set. Counting the elements of the inverse image with sign, we get an invariant of $Y \# (S^1 \times S^3)$.
As proposition \ref{Prop tw Psi}, we can prove that the invariant of  $Y \# (S^1 \times S^3)$ is equal to the SW invariant  of $Y$.

Using this, we can show that for $d \geq 6$ with $d$ even,  $S_d \# (S^1 \times S^3)$ is not diffeomorphic to a connected sum of copies of
$K3$,  $S^2 \times S^2$ and $S^1 \times S^3$ as above. This discussion can be extended to $Y \# (\#^m S^1 \times S^3)$.

\begin{remark}
 So far we have assumed that $Y$ is spin in the previous examples,
since the twisted Donaldson invariant is defined for spin 4-manifolds.
 But, using the Donaldson invariants, we can get the similar examples for non-spin 4-manifolds $Y$.
\end{remark}

\subsubsection{More examples of exotic smooth structures} \label{subsection more examples}

We can generalize the results in the previous subsection as follows. In this subsection 4-manifolds are not necessarily spin.

\vspace{3mm}

$(1)$
Let $Y$ be a closed $4$-manifold with $b^+(Y) > 1$,
 and $Z$ be a closed $4$-manifold with $\pi_1(Z) \cong \Z^m$ and  $b^+(Z) = 0$.

 Take loops $\ell_1, \dots, \ell_{m}$ in $Z$ which generate $\pi_1(Z)$.
 Then a paralell argument   to proposition \ref{Prop tw Psi}
gives  the equality
\[
    \Psi_{Y \# Z}([M_1], \dots, [M_d], [\ell_1], \dots, [\ell_m]) = \Psi_{Y}([M_1], \dots, [M_d])
\]
 for  $[M_1], \dots, [M_d] \in H_{\leq 3}(Y;\Z)$.
 This produces the following statement:

\vspace{1mm}

{\em Let $Y$ and $Y'$ be  closed and $4$-manifolds with $b^+ > 1$,
which are mutually homeomorphic. If
the Donaldson invariant over $Y$ vanishes but the one over $Y'$ does not so that
they  are exotic.
Then  $(Y \# Z, Y' \# Z)$ also consists of  an exotic pair.}

\vspace{3mm}

$(2)$
We can remove the assumption on $\pi_1(Z)$ using the Seiberg-Witten theory.
Let $Y,Y', Z$ be as above, but $\pi_1(Z)$ is not necessarily isomorphic to $\Z^m$.

If the SW invariant of $Y$ vanishes but that of $Y'$ does not, then we can show that $Y \# Z$ and $Y' \# Z$ also form
 an exotic pair.

 If $b_1(Z) = k$ we take $k$ loops in $Z$ which are generators of $H_1(Z;\Z) / \Tor$
  and consider holonomies around the loops.
  Then we get a map from the moduli space to $U(1) \times \cdots \times U(1)$.
   We cut down the moduli space by this map.
(If $k = 0$, we do not need to cut down the moduli space. )
Then the proof goes in a parallel way to (3) of section \ref{subsection example rank 1}.

\vspace{3mm}

$(3)$
Both the Donaldson  and SW invariants vanish
over  $Y_1 \# Y_2 \# Z$ with $b^+(Y_1), b^+(Y_2)>0$ and $b^+(Z)=0$.
The Bauer-Furuta invariant \cite{BF} is a refinement of the SW invariant,
and it  is not always  trivial over $Y_1 \# Y_2 \# Z$.

Let  $Y_1, Y_2$ are simply connected, symplectic 4-manifolds with $b^+ \equiv 3 \mod 4$,
and put $Z= (\#^{l} \overline{\mathbb{CP}}^2) \# (\#^{m} S^1 \times S^3)  $ with $l \geq 1, m \geq 0$.
Freedman's theory tells us that $Y_1 \# Y_2 \# Z$ is homeomorphic to
\[
  W :=(\#^{b^+(Y_1) + b^+(Y_2)}\mathbb{CP}^{2}) \#
          (\#^{b^-(Y_1) + b^-(Y_2)+l}\overline{\mathbb{CP}}^2) \#
          (\#^{m} S^1 \times S^3).
\]
Then we can verify that
\vspace{1mm}

\begin{center}
{\em  $Y_1 \# Y_2 \# Z$ is not diffeomorphic to $W$. }
\end{center}

\vspace{1mm}

Let us use the BF invariant to distinguish these exotic smooth structures.
It is known that the  invariant is not trivial for $Y_1 \# Y_2$.
Combining this with  corollary $8$ of \cite{IL}, the BF invariant of $Y_1 \#Y_2 \# Z$ turns out to be non-trivial.
 On the other hand, the BF invariant of $W$ is trivial because $W$ has a positive scalar curvature metric. Therefore they are not diffeomorphic to each other.

If we replace $S^1 \times S^3$ with $S^2 \times S^2$ or $T^2 \times S^2$, the above discussion does not work since  $b^+(S^2 \times S^2), b^+(T^2 \times S^2) > 0$. It seems that there is no example of an exotic smooth structure of a 4-manifold of the form $Y \# (S^2 \times S^2)$ or $Y \#(T^2 \times S^2).$

Instead of $S^2 \times S^2$ and $T^2 \times S^2$, let us consider $\Sigma_{g} \times \Sigma_{h}$, where $\Sigma_{g}, \Sigma_{h}$ are oriented, closed surface of genus $g, h$.  Let $Y_1, Y_2$ be simply connected symplectic $4$-manifolds with $b^+(Y) \equiv 3 \mod 4$. In \cite{Sasa}, it is proved that if $g, h$ are odd the Bauer-Furuta invariants of $Y_1 \# \Sigma_{g} \times \Sigma_{h}$, $Y_1 \# Y_2 \# \Sigma_{g} \times \Sigma_{h}$ are non-trivial.  Using this fact, we can show that there are exotic smooth structures of $Y_1 \# (\Sigma_{g} \times \Sigma_{h}) \# Z, Y_1 \# Y_2 \# (\Sigma_{g} \times \Sigma_{h}) \# Z$, where $Z =  (\#^l \overline{\mathbb{CP}}^2) \# (\#^m  S^1 \times S^3)$ with $l \geq 1, m \geq 0$.  See also \cite{IS} similar results.

\subsection{Higher rank case} \label{subsection higher rank}
Recall that we  assume that
 all four manifolds $Y$ and $X$
 are compact and spin with $b^+ >1$, without mention.
\subsubsection{Construction of the twisted Donaldson invariants}
Let us
fix a homomorphism
$${\bf f} : \pi_1(X) \to \Gamma  :=\Z^m.$$

 Let ${\mathbb L} \rightarrow B \Gamma \times \Pic(B \Gamma)$ be the universal line bundle. The first Chern class of ${\mathbb L}$ is  given by the following formula
\begin{equation}  \label{eq c1 L}
          c_1({\mathbb L}) = \omega_1 \cup \eta_1 + \cdots + \omega_m \cup \eta_m.
\end{equation}
Here $\{ \omega_j \}_j$ is a basis of $H^1(B\Gamma;\Z)$ and $\{ \eta_j \}_j$ is the basis of $H^1(\Pic (B\Gamma);\Z)$ which is induced by $\{ \omega_j \}_j$.

Take an $SU(2)$-vector bundle $E$ on $X$ with $c_2(E)$ odd.  As before we take the $U(2)$-vector bundle $\hat{E}$ on $\hat{X} = X \# \overline{\mathbb{CP}}^2$
with $c_1(\hat{E}) = e$ and with $c_2(\hat{E}) = c_2(E)$. Then  we have the universal bundle
\[
       \hat{\mathbb E} \rightarrow  \hat{X} \times {\frak B}^*(\hat{E}).
\]

Let us  introduce  a twisted $\mu$ map
\[
       \mu^{\tw}_{\bf f} :  H_{odd}(X;\Z) \otimes H_1( B\Gamma;\Z) \rightarrow
                          \lim_{ \substack{\longleftarrow  \\ B} } H^*(  B \times \Pic( B\Gamma); \Q),
\]
where $B$ runs over  compact subsets in $\frak{B}^*(\hat{E})$.
Take a homology class $\alpha = [M] \in H_{\leq 3}(X;\Z)$  realized by a submanifold  $M \subset X$.
  We can regard $M$ as a submanifold in $\hat{X}$.
Take $\beta = [ j: S^1 \rightarrow B\Gamma] \in H_1(B\Gamma;\Z)$, where $j$ is a continuous map.
Let us  consider a vector bundle
\[
    \hat{\mathbb E}^{\tw}_{j, {\bf f}} :=  \hat{\mathbb E} \boxtimes_{\hat{X}} (\hat{f} \times id)^* {\mathbb L}
     \boxtimes_{\Pic}  (j \times id)^* {\mathbb L}
  \     \rightarrow  \
       \hat{X} \times S^1 \times {\mathfrak B}^*(\hat{E}) \times \Pic(B\Gamma)
\]
where $\hat{f} : \hat{X} \rightarrow B\Gamma$ is the continuous map
corresponding to the group homomorphism ${\bf f} : \pi_1(X) = \pi_1(\hat{X})  \rightarrow  \Gamma$.

Fix a  spin structure on $M \times S^1$.
  Coupling the Dirac operator over  $M \times S^1$
   with $\hat{\mathbb E}^{\tw}_{j, {\bf f}}|_{M \times S^1 \times \frak{B}^*(\hat{E}) \times \Pic(B\Gamma)}$, we
   obtain the family  of twisted Dirac operators $\mathbb{D}_{M \times S^1, {\bf f} }$
   over $M \times S^1$ parametrized by $\frak{B}^*(\hat{E}) \times \Pic(B\Gamma)$.
   Then we  define the index
\[
       \Ind_B  \mathbb{D}_{M \times S^1, {\bf f}} \in K^*(B \times \Pic(B\Gamma))
\]
for each compact subset $B \subset \frak{B}^*(\hat{E})$.

\begin{definition}
$$\mu^{\tw}_{\bf f}(\alpha, \beta) \equiv
\{ \ch( \Ind_{B}( {\mathbb D}_{M \times N, {\bf f} }) ) \}_{B}  \ \in \
\displaystyle{ \lim_{ \substack {\longleftarrow \\ B} }  H^*(B \times \Pic(B\Gamma);\Q)},$$
 where $B$ runs over the  compact subsets in ${\frak B}^*(\hat{E})$.
\end{definition}

\begin{lemma}  \label{lem higher rank twisted mu}
Under the above situation,  the equality
\[
     \mu^{\tw}_{\bf f}(\alpha, \beta)
      = \{ \ch( \hat{\mathbb{E}} \boxtimes_{\hat{X}} ( \hat{f} \times id)^* \mathbb{L}) / \alpha \} \cup \{  c_1 (\mathbb{L}) / \beta \}
\]
holds,
where the right hand side is the image of the cohomology class by the natural map
$H^*({\frak B}^*(E) \times \Pic(B\Gamma);\Q) \rightarrow \displaystyle{ \lim_{ \substack {\longleftarrow \\ B} }
H^*(B \times \Pic(B\Gamma);\Q)}$.

Moreover $\mu^{\tw}_{\bf f}(\alpha, \beta)$ depends only on $\alpha, \beta$ so that it
is independent of choices of  representatives
$M$ and  $j:S^1 \rightarrow B\Gamma$ for
 $\alpha$ and  $\beta$ respectively,
  and of spin structure on $M \times S^1$.
\end{lemma}

\begin{proof}
The Atiyah-Singer index theorem for the  family  gives a formula
\[
   \mu^{\tw}_{\bf f}(\alpha, \beta) = \ch(\hat{\mathbb E}^{\tw}_{j, {\bf f}}) \hat{A}(M \times S^1) / [M \times S^1].
\]
Since $\dim M \leq 3$, the equalities
$\hat{A}(M \times S^1) = \hat{A}(M) \hat{A}(S^1) = 1$ hold. So
\[
   \begin{split}
     \mu^{\tw}_{\bf f}(\alpha, \beta)
     &= \ch( \hat{\mathbb E}^{\tw}_{j, {\bf f}})  / [M \times S^1]  \\
     &= \ch( \hat{\mathbb{E}} \boxtimes_{\hat{X}} ( \hat{f} \times id)^* \mathbb{L} \boxtimes_{\Pic} (j \times id)^* \mathbb{L}  ) / [M \times S^1] \\
     &= \{ \ch( \hat{\mathbb{E}} \boxtimes_{\hat{X}} ( \hat{f} \times id)^* \mathbb{L}) / [M] \} \cup \{(j \times id)^* \ch(\mathbb{L}) / [S^1] \} \\
     &= \{ \ch( \hat{\mathbb{E}} \boxtimes_{\hat{X}} ( \hat{f} \times id)^* \mathbb{L}) / [M] \} \cup \{  c_1(\mathbb{L}) / \beta \}.
   \end{split}
\]
Here we have used the equalities
\[
         (j \times id)^* c_1( \mathbb{L}) / [S^1] = c_1(\mathbb{L}) / \beta, \
         (j \times id)^* c_1(\mathbb{L})^n / [S^1] = 0 \  (n \geq 2).
\]
The second equality follows from the formula (\ref{eq c1 L}).
This formula implies that $\mu^{\tw}_{\bf f}(\alpha, \beta)$ is independent of choices of representatives $M$, $j$ of the homology classes and  spin structure on $M \times S^1$.
\end{proof}

Let $\Sigma_0$  be  a surface in $\hat{X}$ with $[\Sigma_0] = PD(e)$,
and  $V(\hat{E}; \Sigma_0, M_1, \dots, M_d)$ be the intersection defined as (\ref{eq V}).

\begin{definition}
Let us fix a homomorphism
${\bf f} : \pi_1(X) \to \Gamma := \Z^m$.

For $r =1,2$ we define
\begin{align*}
   &  \Psi_{X, {\bf f} }^{\tw, (r)} :
      A(X) \otimes H_{odd}(X;\Z) \otimes H_1(B\Gamma;\Z)  \rightarrow  H^*(\Pic(B\Gamma);\Q) ,  \\
  &
     \Psi_{X, {\bf f}}^{\tw, (r)}( [M_1], \dots [M_d]; \alpha, \beta) :=
       \frac{1}{2 \cdot 4^{d_0}} \mu^{\tw}_{\bf f}(\alpha, \beta) / [V(\hat{E}; \Sigma_0, M_1, \dots, M_d)]
       \end{align*}
if there is an $SU(2)$-vector bundle $E$ with $c_2(E)$ odd and with $\dim \frak{M}(E) = \sum_{j=1}^{d} ( 4-\dim M_j) + r$. Here $d_0$ is the number of homology classes $[M_j]$ of degree $0$.

 We put $\Psi_{X,{\bf f}}^{\tw, (r)}([M_1], \dots, [M_d]; \alpha, \beta)=0$ otherwise.
 \end{definition}
Both   $\Psi_{X, {\bf f}}^{\tw, (1)}$ and $\Psi_{X, {\bf f}}^{\tw, (2)}$ are well-defined
by proposition \ref{prop V compact}.

Next let us formulate $\Psi_{X, {\bf f}}^{\tw, (3)}$.
In this case
we  take a dual submanifold of $\mu^{\tw}_{\bf f}(\alpha, \beta)$ as in section $3.1.2$, since the homology class  of $V(\hat{E}; \Sigma_0, M_1, \dots, M_d)$ may not be well-defined, if $\dim V(\hat{E}; \Sigma_0, M_1, \dots, M_d) = 3$. See proposition \ref{prop V compact}.

\begin{lemma} \label{lem tw mu c}
Let $\beta \in H_1(B\Gamma;\Z)$ and $c \in H_3(\frak{B}^*(\hat{E}); \Z)$.
Then the equalities hold:
$$    \mu^{\tw}_{\bf f}(\alpha, \beta) / c =
\begin{cases}
  - \left< \mu(\alpha), c \right>  \{ c_1( {\mathbb L}) / \beta \}, &  \text{ for }  \alpha \in H_1(X;\Z), \\
  0 &  \text{ for }  \alpha \in H_3(X;\Z).
\end{cases}$$

\end{lemma}

\begin{proof}
Take $\alpha = [\ell] \in H_1(X;\Z)$, $\beta  \in H_1(B\Gamma;\Z)$, $c \in H_3(\mathfrak{B}^*(\hat{\mathbb E});\Z)$.
One can check $c_1( \hat{\mathbb{E}}|_{\ell \times \frak{B}^*(\hat{E}) }) = 0$ as  in the proof of lemma \ref{lem tw mu}.
 From lemma~\ref{lem higher rank twisted mu}, we have the equalities
 \[
   \begin{split}
     & \mu^{\tw}_{\bf f}(\alpha, \beta) / c   \\
     & = \{ \ch( \hat{\mathbb E} ) (\hat{f} \times id)^*\ch( {\mathbb L}) / \alpha \times c \} \cup  \{ c_1(\mathbb{L}) /  \beta \}  \\
     & = \left\{ ( 2 - c_2(\hat{ \mathbb E }) + \ch_{\geq 4}( \hat{\mathbb E} ) )
 \left.  \left( 1 + (\hat{f} \times id)^* c_1({\mathbb L}) + \frac{1}{2!} ( \hat{f} \times id)^* c_1(\mathbb{L})^2 + \cdots  \right) \right/ \alpha \times c \right\}   \\
     & \qquad  \cup \{ c_1(\mathbb{L}) / \beta \}  \\
    & = - \left< c_2(\hat{\mathbb E}) / \alpha, c  \right>   \{ c_1( \mathbb{L} ) / \beta \}
    = - \left< \mu(\alpha),  c \right> \{ c_1(\mathbb{L}) / \beta \}.
   \end{split}
 \]

One can check that
 $     \mu^{\tw}_{{\bf f}}(\alpha, \beta) / c = 0$ vanishes  by a similar computation,  if $\alpha \in H_3(X;\Z)$.
\end{proof}

This lemma leads us  to the  following definition.

\begin{definition}
We define
\[
  \Psi_{X, {\bf f}}^{\tw, (3)} : A(X) \otimes H_{odd}(X;\Z) \otimes H_1(B\Gamma, \Z)\rightarrow H^*(\Pic(B\Gamma);\Z)
\]
by the following formulas.
Let $[M_1], \dots, [M_d] \in A(X)$ and $\beta \in H_1(B\Gamma;\Z)$. Then:
\begin{align*}
  &  \Psi_{X, {\bf f}}^{\tw, (3)}([M_1], \dots, [M_d]; \alpha, \beta)  \\
  &  \qquad =
\begin{cases}
 - \Psi_{X, odd}([M_1], \dots, [M_d], \alpha)\{  c_1(\mathbb{L}) / \beta \}
&   \text{ for } \alpha \in H_1(X;\Z), \\
0 &   \text{ for } \alpha \in H_3(X;\Z).
\end{cases}
\end{align*}
\end{definition}

\begin{remark}
$\Psi_{X, \bf{f}}^{\tw, (3)}$ is again  independent of  ${\bf f} : \pi_1(X) \rightarrow \Gamma$.
\end{remark}

Next we define $\Psi^{\tw,(4)}_{X, {\bf f}}$.
It turns out that they do depend on $\bf f$.

\begin{lemma}
Let $\beta \in H_1(B\Gamma;\Z)$ and  $c \in H_4(\frak{B}^*( \hat{E} );\Z)$. Then  we have
the equalities
\begin{align*}
&   \mu^{\tw}_{\bf f}(\alpha,\beta)/c  \\
&    =
\begin{cases}
 - \left< \mu([pt]), c \right>
     \{ (\hat{f} \times id)^*(c_1(\mathbb{L}))/ \alpha \} \cup
     \{ c_1(\mathbb{L}) / \beta \} &   \text{ for } \alpha \in H_1(X;\Z), \\
      - \left< \mu([pt]), c \right>
       \{ (\hat{f} \times id)^* (c_1(\mathbb{L} )^3) / \alpha \} \cup
       \{ c_1(\mathbb{L}) / \beta \} &   \text{ for } \alpha \in H_3(X;\Z).
\end{cases}
\end{align*}
\end{lemma}

\begin{proof}
Let $\alpha = [M] \in H_1(X;\Z)$,  $\beta \in H_1(B\Gamma;\Z)$ and $c \in H_4(\frak{B}^*(\hat{E});\Z)$.

From  (\ref{eq c1 L}), we have
\[
      (\hat{f} \times id)^* c_1( \mathbb{L})^n / \alpha = 0
\]
for $n \not= 1$. Using this,   $c_1( \hat{\mathbb{E}}|_{M \times \frak{B}^*( \hat{E} )} ) = 0$ and Lemmas \ref{lem higher rank twisted mu},  we get
\[
   \begin{split}
  &  \mu^{\tw}_{\bf f}( \alpha, \beta ) / c\\
    = & \{  \ch( \hat{\mathbb{E}} ) (\hat{f} \times id)^* \ch(\mathbb{L}) / \alpha \times c \}  \cup \{ c_1(\mathbb{L})/\beta \} \\
  = & \left. \left\{ ( 2 - c_2( \hat{\mathbb{E}}) + \ch_{\geq 4}( \hat{\mathbb{E}}) ) \left( 1 + ( \hat{f} \times id)^* c_1(\mathbb{L})+ \frac{1}{2!} ( \hat{f} \times id)^* c_1(\mathbb{L})^2 + \cdots  \cdots \right) \right/ \alpha \times c  \right\}  \\
    & \qquad  \cup  \left\{ c_1(\mathbb{L})  / \beta \right\} \\
    =&  - \left< c_2( \hat{\mathbb{E}}) / \{ pt \},   c  \right>  \{ (\hat{f} \times id)^* c_1( \mathbb{L}) / \alpha \} \cup \{ c_1(\mathbb{L}) / \beta \} \\
    =&  - \left< \mu([pt]), c \right>  \{ (\hat{f} \times id)^* c_1( \mathbb{L}) / \alpha \} \cup \{ c_1(\mathbb{L}) / \beta \}.
  \end{split}
\]
We have obtained the first formula in the statement.  The proof of the second formula is similar.
\end{proof}

With this lemma in mind, let us introduce  the following definition.

\begin{definition}
We define
\[
    \Psi_{X, {\bf f}}^{\tw, (4)} : A(X) \otimes (H_{odd}(X;\Z) \otimes H_1(B\Gamma; \Z)) \rightarrow H^*( \Pic(B\Gamma); \Q )
\]
by  the following formulas. Let   $\beta \in H_1(B\Gamma;\Z)$. Then
  \begin{align*}
   &   \Psi_{X, {\bf f}}^{\tw, (4)}([M_1], \dots, [M_d]; \alpha, \beta)   \\
   & =
   \begin{cases}
     -\Psi_{X, odd}([pt], [M_1], \dots, [M_d]) \{ ( \hat{f} \times id)^*  (c_1(\mathbb{L})) / \alpha  \} \cup
     \{ c_1( \mathbb{L} )/ \beta \}    \\
     \qquad \qquad    \qquad \qquad    \qquad \qquad    \qquad \qquad    \qquad \qquad
      \text{ for } \alpha \in H_1(X;\Z), \\
      - \Psi_{X, odd}([pt], [M_1], \dots, [M_d])   \{ (\hat{f} \times id)^*   (c_1( {\mathbb L} )^3 ) / \alpha \}  \cup ( c_1( \mathbb{L} ) / \beta)    \\
     \qquad \qquad    \qquad \qquad    \qquad \qquad    \qquad \qquad    \qquad \qquad
        \text{ for } \alpha \in H_3(X;\Z).
\end{cases}
\end{align*}
\end{definition}

\begin{remark}
The invariants  $\Psi_{X, {\bf f}}^{\tw, (r)}$ defined in this subsections are homomorphisms and descends to maps on $H_1(X;\Z)$. See remark \ref{Psi hom} and remark \ref{Rem tw mu}.
\end{remark}

\subsection{Applications}

\subsubsection{Statements}

We will consider  connected sums
$\#^m S^1 \times S^3$, and let
$\pi_1(\#^m S^1 \times S^3) \to \Gamma :=\Z^m$ be the homomorphism obtained by commutating the fundamental group.
By forgetting $\pi_1(Y)$ factor, we obtain the group homomorphism
$${\bf f}: \pi_1(Y \# (\#^m S^1 \times S^3)) \cong
\pi_1(Y) * \pi_1(\#^m S^1 \times S^3)\to \Gamma :=\Z^m.$$

Let $\ell_i = S^1 \times \{ pt \} \subset
\#^m S^1 \times S^3$ be the $i$-th  canonical loop for $1 \leq i \leq m$,
which consists of a canonical generating set of $\pi_1(\#^m S^1 \times S^3)$.

Let us introduce
the simple type condition for $\Psi_X$.  We say that $X$ has {\em simple type}, if the following equality holds:
\[
       \Psi_{X}([pt], [pt], [M_1], \dots, [M_d]) = 4 \Psi_{X}([M_1], \dots, [M_d]).
\]
See definition $1.4$ of \cite{KM}.
Several classes of algebraic surfaces satisfy this property.

By use of the invariant $\Psi_{X, \bf{f}}^{\tw, (4)}$, we can reproduce both  the non decomposability
(theorem \ref{thm no connected sum}) and exotic pairs (theorem \ref{thm exotic X_i}) under the additional assumption of simple type:

\begin{corollary}
Let $m$ be an integer with $m \geq 2$.
Using the invariant $\Psi^{\tw, (4)}_{X, \bf{f}}$, we can show the following statements:

$(1)$
Let $Y$ be a simply connected, spin, algebraic surface of simple type, with $b^+ > 1$.
Then $X= Y \# (\#^{m} S^1 \times S^3)$ cannot admit any connected sum decomposition as
$X = X_1 \# X_2$ with $b^+(X_i) >0$.

$(2)$
Suppose
$Y_1$ and $Y_2$ is a pair of compact and spin $4$-manifolds with $b^+ >1$, which
are mutually homeomorphic. Assume moreover
$Y_2$ has simple type.

If
 $\Psi_{Y_1, odd}$ is trivial, but $\Psi_{Y_2, odd}$ is not trivial,  so that they consist of  an exotic pair, then
 the pair $( Y_1\# ( \#^m S^1 \times S^3), Y_2\# (\#^m S^1 \times S^3) )$ is also exotic.
\end{corollary}
This is obtained by
a parallel argument
as Section \ref{ex Y S^1 times S^3}, by
 use of
 proposition \ref{prop higher rank tw D inv}
 and  proposition \ref{prop higher rank vanishing}
 below.

\subsubsection{Computation of the invariants}
Assume that
 all four manifolds $Y$ or $X$
 are compact and spin with $b^+ >1$.
\begin{proposition} \label{prop higher rank tw D inv}
Let  $X = Y \# (\#^m S^1 \times S^3)$.
Then for $\bf f$  as above,  we have the equality
\[
  \begin{split}
    & \Psi_{X, {\bf f}}^{\tw,(4)} ([M_1], \dots, [M_d], [\ell_1], \dots, [\ell_{m-1}]; [\ell_{m}], \beta)  \\
    &  =  - \Psi_{Y, odd}([pt], [M_1], \dots, [M_d]) \eta_m \cup \{ c_1({\mathbb L}) / \beta \},
  \end{split}
\]
where  $[M_1], \dots, [M_d] \in H_{\leq 3}(Y;\Z)$, and
$\eta_m$ is the canonical generotor of the $m$-th component of $H^1(\Pic(B\Z^m); \Z) = \oplus^{m} H^1(\Pic (B\Z);\Z)$.
\end{proposition}

\begin{proof}
Use the equality
\[
      (\hat{f} \times id)^* c_1( \mathbb{L}) / [\ell_m] = \eta_m,
\]
and follow the discussion of the proof of Proposition \ref{Prop tw Psi}.
\end{proof}

\begin{proposition} \label{prop higher rank vanishing}
If $X = X_1 \# X_2$ with $b^+(X_1), b^+(X_2) > 0$, $\Psi_{X, {\bf f}}^{\tw, (4)}$ is trivial.
\end{proposition}

\begin{proof}
This follows from the vanishing result of the Donaldson invariant $\Psi_{X}$.
See Proposition \ref{vanish}.
\end{proof}



\subsubsection{Examples}
Recall  $S_d$ in  $3.2.2$, and put
 $X_{d, m} := S_d  \# (\#^m S^1 \times S^3)$ with $m \geq 2$. Let $d$ be even and greater than or equal to $4$.

 Lemma \ref{non vanishing}  below with
 proposition \ref{prop higher rank tw D inv} gives
 infinitely many exotic smooth structures
$X_{d, m}$  for $m \geq 2$ as in section \ref{ex Y S^1 times S^3}.

\begin{lemma}\label{non vanishing}
$\Psi_{X_{d, m}, {\bf f}}^{\tw, (4)}$ is non-trivial for $m \geq 2$.
\end{lemma}

\begin{proof}
It is known that $S_d$ have simple type if $d \geq 4$. See lemma 8.9 of \cite{KM}.

By proposition \ref{prop higher rank tw D inv}, we have
\[
  \begin{split}
            & \Psi_{X_{d, m}, {\bf f}}^{\tw, (4)}([pt], \overbrace{h, \dots, h}^{n}, [l_1], \dots, [l_{m-1}]; [l_m], \beta_1)   \\
           &\quad = - \Psi_{S_d, odd}([pt], [pt], \overbrace{h, \dots, h}^{n}) \eta_m \cup \eta_1 \\
   \end{split}
\]
Here $\beta_1$ is the generator of the first component of $H_1(B \Z^m;\Z) \cong \oplus^m H_1(B\Z;\Z)$,  $h \in H_2(S_d;\Z)$ is the hyperplane class and $n$ is a large integer of the form
\[
        n = 4n' - \frac{3}{2} ( 1 + b^+(S_d) )
\]
with some even integer $n'$. Note that for such $n$ we have
\[
        \Psi_{S_d, odd}( [pt], [pt], \overbrace{h, \dots, h}^{n})
          = \Psi_{S_d}([pt], [pt], \overbrace{h, \dots, h}^{n})
         =4 \Psi_{S_d}( \overbrace{h, \dots, h}^{n}) \not= 0. \\
\]
Therefore, $\Psi_{X_{d, m}, {\bf f}}^{\tw, (4)}$ is non-trivial for $m \geq 2$ (since $\eta_m \cup \eta_1 \not= 0$).
\end{proof}


\section{Review of cyclic cohomology}
Recall that in the commutative case in Section~\ref{sec:family.det.line},
the twisted Donaldson invariant was introduced by using the cohomology of the classifying space $H^*(\Pic \Gamma ; \R)$.
In the case $\Pic \Gamma $ is homotopy equivalent to the dual torus $\hat{T}^m$, and hence
one can use the de Rham theory over the classifying space.

In the general case of a non commutative group $\Gamma$,
it would be possible to define $\Pic \Gamma$, however its structure is quite messy.
In fact, its topology can be quite difficult to treat.

There are four steps of  non commutative geometry to overcome this difficulty.
Let $M$ be a finite dimensional smooth manifold.

\begin{itemize}

\item
Firstly the de Rham cohomology theory is reformulated by a new cohomology theory called
cyclic cohomology theory in terms of the algebra  $C^{\infty}(M)$
of  smooth functions on $M$.

\item
Secondly the cyclic cohomology theory can be applied also to non-commutative algebras.

\item
Thirdly the isomorphism
$C(\Pic \Z^m)  \cong C^*_r(\Z^m)$ holds between  two $C^*$-algebras,
where the right hand side is given by completion of the group ring $\C \Z^m$ under the operator norm. Here $\C\Z^m$ is identified as a subalgebra of $\mathcal{B}(l^2(\Z^m))$ by the left regular representation.

\item
Fourthly the $C^*$-algebra $C^*_r (\Gamma)$ exists reasonably
as the norm completion of the group ring
$\C \Gamma$ in $\mathcal{B}(l^2(\Gamma))$ for a general group $\Gamma$.

\end{itemize}
There is a $*$-subalgebra $C^{\infty}(\Pic \Z^m)  \subset C(\Pic \Z^m) $, and hence if one can find
a $*$-subalgebra $\mathcal A (\Gamma) \subset C^* \Gamma$ which plays the role of `cohomology of smooth algebras',
then one would be able to generalize the construction of the previous section.
It will be called a smooth algebra later in this section.

\vspace{2mm}

Let $\cA$ be a unital algebra over $\C$ and $K_0(\cA)$
be the Grothendieck group associated to the stable isomorphism classes of finitely generated projective modules over $\cA$.

Consider a particular case $\cA = C(X)$, where $X$ is a (reasonable) compact topological space.
Let us state a basic correspondence by Swan.
\begin{lemma}
There is a natural isomorphism:
$$K^0(X) \cong K_0(C(X))$$
where the left hand side is the topological $K$-theory.

The isomorphism assigns to a complex vector bundle over $X$ the set of continuous sections which admits a structure of a finitely generated projective $C(X)$ module.
\end{lemma}

The main goal of this section is to introduce theory of cyclic cohomology for a general non-commutative algebra
and explain how to pair it with $K$-theory of the algebra. We will also present some basic examples.
Throughout this section, we refer the comprehensive book by Connes \cite{Connes94}.

\subsection{Cyclic cohomology}
Cyclic cohomology was introduced by Connes. See Section III.1 in \cite{Connes94}.
\begin{definition}
Cyclic cohomology $HC^{\bullet}(\cA)$ is the cohomology of the complex $(C^n_{\lambda}(\cA), b)$, where: $$C^n_{\lambda}(\cA) =\{  \ \phi : \cA \otimes \dots \otimes \cA \to {\mathbb C }  \ \  |  \ \ ( *) \  \}$$
 is the space of $(n+1)$ multi-linear functionals satisfying the cyclic condition
\[
(*) \qquad \phi(a^1, \ldots, a^n, a^0)=(-1)^n\phi(a^0, \ldots, a^n) \qquad a^j\in\cA,
\]
and
\[
b: C^n_{\lambda}(\cA)\rightarrow C^{n+1}_{\lambda}(\cA)
\]
is the Hochschild coboundary map given by
\begin{align*}
(b\phi)(a^0,  &\dots, a^{n+1})= \\
& \sum_{j=0}^n\phi(a^0, \dots, a^ja^{j+1}, \ldots, a^{n+1})+
(-1)^{n+1}\phi(a^{n+1}a^0, \dots, a^n) .
\end{align*}
\end{definition}

To  recover the de Rham theory for commutative case,
one has to pass through stabilization of the cyclic theory.

For two unital algebras $\cA_1$ and $\cA_2$, there is a well defined cup product
\begin{equation}
\label{eq:cup.product}
\cup: HC^i(\cA_1)\times HC^j(\cA_2)\rightarrow HC^{i+j}(\cA_1\otimes\cA_2).
\end{equation}
Let  $\cA_1=\cA_2=\C$,
then $HC^{\bullet}(\C) \cong \C[\sigma]$ is a polynomial ring through
 the cup product  with one generator $\sigma$ of degree $2.$
For  $\cA_1=\cA,  \ \cA_2=\C$,   $HC^{\bullet}(\cA)$ is a module over $HC^{\bullet}(\C)$.

The {\em periodic cyclic cohomology} is defined  by
\begin{equation}
\label{eq:periodic.cyclic.cohomology}
HP^{\bullet}(\cA) : =\lim_{k\to\infty} HC^{\bullet+2k}(\cA),\quad  \bullet=0, 1
\end{equation}
with respect to the inductive limit of the $S$-maps given below
\begin{equation}
\label{eq:S-map}
S: HC^{n}(\cA)\rightarrow HC^{n+2}(\cA) .
\end{equation}
by multiplying the generator $\sigma$ of $HC^2(\C)$ to $HC^{n}(\cA)$ with respect to~(\ref{eq:cup.product}).


\begin{example}
\label{ex:cyclic.coho}
$(1)$
Let $M$ be a compact smooth manifold. Then
\[
HC^k(C^{\infty}(M))\simeq [\ker d^t \subset\Omega_k(M)]\oplus H_{k-2}(M, \C)\oplus H_{k-4}(M, \C)\oplus\cdots
\]
where $\Omega_k(M)$ is the  space of $k$-forms  and $d^t$ is the de Rham differential.
The periodic cyclic cohomology group is isomorphic to the de Rham homologies of $M$ (Section~III.2.$\alpha$ \cite{Connes94})
\[
HP^{0}(C^{\infty}(M))\simeq H_{even}(M), \qquad HP^{1}(C^{\infty}(M))\simeq H_{odd}(M).
\]

\vspace{3mm}

$(2)$
Let $\Gamma$ be a discrete group and $\C\Gamma$ be its group ring. 
There is a canonical map
\begin{equation}
\label{eq:group.cyclic.coho}
\tau : \ H^{\bullet}(\Gamma; \C)\simeq H^{\bullet}(B\Gamma; \C)\rightarrow HC^{\bullet}(\C\Gamma).
\end{equation}
See page 377 of \cite{Connes-Moscovici} or~\cite{Lott} for a concrete description.
\end{example}

Let $Z_{\Gamma}(g)=\{h\in\Gamma| gh=hg\}$
be the centrilizer of $g$ in $\Gamma$ and put  $N_g=Z_{\Gamma}(g)/\langle g\rangle$.

\begin{lemma}\label{ab,hyp}
Assume $\Gamma$ is torsion free.
Then
the isomorphism
\begin{equation}
\label{eq:HP.H.gp.alg}
\tau : \ H^{even/odd}(B\Gamma, \C) \simeq HP^{0/1}(\C\Gamma)
\end{equation}
holds, if there is  $l_0>0$ such that
$H^l(B N_g,  \C)$ vanishes for all $g\neq e$ and  all $l\ge l_0$.
\end{lemma}

So a non commutative counterpart of the group cohomology should be the periodic cyclic cohomology of the group ring.

\begin{proof}
Denote by $\langle \Gamma\rangle_f$
the set of conjugacy classes of  finite order elements in $\Gamma$,
 and by $\langle\Gamma\rangle_i$ the one with elements of infinite order.
Then $HC^{\bullet}(\C\Gamma)$ admits a decomposition (see~\cite{Burghelea})
\begin{equation}
\label{eq:HC.group.algebra}
HC^{l}(\C\Gamma)=\{\Pi_{g\in\langle\Gamma\rangle_f}  [\oplus_{m+n=l} H^{m}(N_g, \C)\otimes HC^{n}(\C)]\}
\oplus \Pi_{g\in\langle \Gamma \rangle_i}H^{l}(N_g, \C).
\end{equation}

Because $\Gamma$ is torsion free, so $g\in\langle \Gamma\rangle_f$ if and only if $g$ is the  identity $e$.
It follows from the equality
$
H^m(N_e, \C)=H^{m}(\Gamma, \C)
$ that
 (\ref{eq:HC.group.algebra}) is reduced to
\[
HC^{l}(\C\Gamma)=[\oplus_{l-m\ge0, \mathop{even}}H^m(\Gamma, \C)]\oplus \Pi_{g\neq e}
H^l(N_g,  \C).
\]
\end{proof}

\begin{example}
$(1)$ (\ref{eq:HP.H.gp.alg}) holds if $\Gamma=\Z^k$ is  free abelian. This follows from the fact that
\[
H^l(B(Z_{\Gamma}(g)/\langle g\rangle), \C)=H^l(B(\Z^k/\langle g\rangle), \C)=0
\]
 for $l >k$.

\vspace{2mm}

$(2)$
(\ref{eq:HP.H.gp.alg}) is true when $\Gamma$ is a torsion free Gromov hyperbolic group, such as the fundamental group of a compact hyperbolic manifold.
In such  case,
 the centrilizer of an element $g\in\Gamma$ of infinite order contains the cyclic group $\langle g\rangle$ as a subgroup of finite index and then $H^l(N_g, \C)$ vanishes.

 \vspace{2mm}

$(3)$
 In general, (\ref{eq:HP.H.gp.alg}) is not true if $\Gamma$ is not torsion free.
\end{example}

\subsection{Pairing with $K$-theory for group $C^*$-algebras}
Let us quickly  review the Connes-Chern character
\[
\Ch :  \ K_0(\cA)\rightarrow HP_0(\cA)
\]
known as  the notion of Chern character in the context of noncommutative geometry.
This  is determined by pairing  $K$-theory with cyclic cohomology
\[
HC^{2k}(\cA)\times K_0(\cA)\rightarrow\C.
\]

Let $e$ be an idempotent of the matrix algebra  $M_q(\cA)$ which represents
an element in $K_0(\cA)$. Let
$\Tr$ be the matrix trace.
A $2k$-cyclic cocycle $\phi$ lifts to the linear map on matrix algebras
 $$\phi_q: M_q(\cA)^{\otimes(2k+1)}\rightarrow M_q(\C).$$
Then we denote by $\phi\#\Tr$ the $(2k+1)$-linear functional by
\[
(\phi\#\Tr)(a^0,\ldots, a^{2k}):=\Tr[\phi_q(a^0,\ldots, a^{2k})] \qquad a^j\in M_q(\cA).
\]
This induces  a natural pairing $HC^{2k}(\cA)\times K_0(\cA)\rightarrow\C$ given by
\begin{equation}
\label{eq:pairing.HP.K}
\langle[\phi], [e]\rangle= \frac{1}{k!}(\phi\#\Tr)(e, \ldots, e).
\end{equation}

\begin{lemma} \label{lem pairing HP K}
The pairing~(\ref{eq:pairing.HP.K}) is compatible with the $S$-map (\ref{eq:S-map}), i.e., $
\langle [\phi], [e]\rangle=\langle [S\phi], [e]\rangle$ and thus this leads to a well defined pairing
\begin{equation}
\label{eq:pairing}
HP^{0}(\cA)\times K_0(\cA)\rightarrow\C.
\end{equation}
\end{lemma}
(Section~III.$3$ Proposition $2$ {\cite{Connes94}.
 See also \cite{ConnesNDG} page $110$ for  the odd degree case).

\begin{definition}\label{CCmap}
The Connes-Chern character map on $K$-theory is the map induced by the pairing in Lemma \ref{lem pairing HP K}
\begin{equation*}
\label{eq:Ch.char}
\Ch: K_{*}(\cA)\rightarrow \Hom(HP^{*}(\cA), \C) \qquad *=0, 1.
\end{equation*}
Identifying the dual theory $HP_{*}(\cA)$, the cyclic homology of $\cA$, with $\Hom(HP^{*}(\cA), \C)$, under a natural pairing $HP_{*}(\cA)\otimes HP^{*}(\cA)\rightarrow\C$, (\ref{eq:Ch.char}) is equivalent to
\[
\Ch: K_{*}(\cA)\rightarrow HP_{*}(\cA) \qquad *=0, 1.
\]
\end{definition}


Denote  the restriction
\begin{equation*}
\label{eq:Ch_n}
\Ch_n: K_{\bullet}(\cA)\rightarrow \Hom(HC^{n}(\cA), \C)
\end{equation*}
by composing the right-hand-side  with $HC^n(\cA)\rightarrow HP^{\bullet}(\cA)$, given by (\ref{eq:periodic.cyclic.cohomology}),
where $n$ and $\bullet$ have the same parity.

For example, the {\em first Chern class} $c_1(E) = \Ch_2[E]$
 is an element in $\Hom(HC^2(\cA), \C)$ given by the pairing
\[
 \qquad [\tau]\mapsto \langle [E], [\tau]\rangle.
\]

Let $\Gamma$ be a finitely generated group.
We introduce an intermediate algebra which can be regarded as a kind of non commutative version of
 the algebra of smooth functions.

 The following Lemma is well known, and is used in \cite{Connes-Moscovici}.
 \begin{lemma} \label{smooth-alg}

 Consider  an intermediate algebra
 $\C \Gamma \subset A  \subset C^*_r(\Gamma)$. If

 $(1)$ $A$
is  closed under holomorphic functional calculus in $C^*_r(\Gamma)$, and

 $(2)$
$A$ is dense in $C^*_r(\Gamma)$,

then the isomorphism holds:
\begin{equation}
\label{eq:K.smooth.alg}
K_0(C^*_r(\Gamma))\simeq K_0(A).
\end{equation}
\end{lemma}
We will call $A$ as a {\em smooth algebra} of $\Gamma$, if it satisfies both the above two conditions,
and will denote
 $A = C^{\infty}(\Gamma)$.
Notice that a
smooth algebra always exists (say, take $C_r^*(\Gamma)$ itself).

  For a  general  $\Gamma$, a  smooth algebra may not be unique.
  On the other hand a group with the property RD (rapidly decay)  has a canonical candidate for $C^{\infty}(\Gamma)$.
For every smooth subalgebra, there is a well defined Connes-Chern character map
\begin{equation}
\label{eq:CC}
\Ch: K_*(C^*_r(\Gamma))\rightarrow HP_*(C^{\infty}(\Gamma)).
\end{equation}
A canonical smooth subalgebra and a geometric description of the Connes-Chern character is obtained by Lott for hyperbolic groups~\cite{Lott}.

\vspace{3mm}

\subsection{Assembly map}
Let $\Gamma$ be a finitely generated group.
Let us briefly explain
 the Baum-Connes assembly map
 \[
\mu: K_*(B\Gamma)\rightarrow K_*(C^*_r(\Gamma)).\]

To understand this map,
let us consider  the example of a free abelian   $\Gamma  \cong \Z^m$.
Then the assembly map is expressed as
$\mu: K_*(T^m)\rightarrow K^*(\hat{T}^m)$.
An element of the $K$-homology group $K_0(B \Z^m)$ is represented by a triple
$(E,F, D)$ where both $E$ and $F$ are complex vector bundles over $T^m$ and $D$ is an elliptic differential operator
between their sections. To describe $\mu(E,F,D)$, take an element $\rho \in \hat{T}^m$
which represents a homomorphism $\rho: \Z^m \to S^1$.
It gives a flat $U(1)$ bundle on $T^m$, and
 twists the triplet  as $(E_{\rho}, F_{\rho}, D_{\rho})$.
 Then the family of indices
 \[
 \sqcup_{\rho \in \hat{T}^m} \ \ker D_{\rho}  - \coker D_{\rho}
 \]
 represents an element in  $K^0(\hat{T}^m)$.
 This construction works only for a commutative $\Gamma$.

 Let us reformulate this assignment of $\mu$, replacing $C(\hat T^m)$ by $C^*_r(\Z)$.
 Recall the isomorphism $C^*_r (\Z^m) \cong C(\hat{T}^m)$ given by the Gelfand transform.
 Consider a flat $C^*$-bundle
$\gamma : = \R^m \times_{\Z^m} C^*_r (\Z^m)$, and twist the two bundles
as $E_{\gamma}: = E \otimes \gamma$ and  $F_{\gamma}:= F \otimes \gamma$
so that they are both  $C^*_r(\Z^m)$-module bundles.
Since $\gamma$ is flat, $D$ extends to an operator
$D_{\gamma}$ between them. Then both $\ker D_{\gamma}$ and $\coker D_{\gamma}$
are finitely generated projective $C^*_r(\Z^m)$ modules (after stabilization).
 So their difference
 $\ker D_{\gamma} - \coker D_{\gamma}$ gives rise to an element $K_0(C^*_r(\Z^m))$.
 Through the isomorphism $K^0(\hat{T}^m) \cong K_0(C^*_r(\Z^m))$, these two constructions coincide.

 Now the latter construction works for any finitely generated groups $\Gamma$: just replace  $\R^m$ by $E \Gamma$
 and $\Z^m$ by $\Gamma$. This is \emph{the Baum-Connes assembly map}.
 An element in $K_*(B\Gamma)$ is represented by a continuous map
$f: M \to B\Gamma $ with a complex vector bundle $E \to B\Gamma$,
where $M$ is a compact spin$^c$ manifold.
Then we obtain a twisted Dirac operator $D_{f^*E}$ over $S \otimes f^*E$
with the spin$^c$ bundle $S$. The image of the  element in
$K_*(C^*_r(\Gamma))$ is given by the higher index of $D_{f^*E}$
(see Section~\ref{higher-index}).

 The \emph{Baum-Connes conjecture} claims rational isomorphism of the assembly map.
 It is still open in general,
 but various classes of discrete groups have been verified to be true.
 Both the Novikov conjecture and the Gromov-Lawson conjecture follows from
 rational injectivity of the assembly map,
 and the Kaplansky conjecture follows from surjectivity of the map.

\subsection{Admissibility}
The assembly map  $\mu$
 factors through the index map of $\Gamma$-invariant elliptic operators
\[
\mu_0: K_*(B\Gamma)\rightarrow K_*(\C\Gamma\otimes \cR)
\]
i.e., $\mu=j_*\circ\mu_0$ for the inclusion $j: \C\Gamma\otimes\cR\rightarrow C^*_r(\Gamma)\otimes\cK.$
Here $\cR$ is the algebra of smooth operators on $M$ and $\cK$ consists of compact operators, the $C^*$-algebra closure of $\cR.$
See Section~\ref{higher-index} or \cite{Connes-Moscovici}.

Recall $ \tau_{\eta} \in HC^{\bullet}(\C\Gamma)$ given by (\ref{eq:group.cyclic.coho}) for $\eta\in H^*(\Gamma, \C)$.
Let us say that an element of the  group cohomology   $\eta\in H^*(\Gamma; \C)$ is {\em extendable}, if
there exists $\tilde\tau_{\eta}\in\Hom(K_*(C^*_r(\Gamma)), \C)$ such that $j^*\tilde\tau_{\eta}=\tau_{\eta}\#\Tr$. This implies that
\begin{equation}
\label{eq:CM.ind}
\langle \tilde\tau_{\eta}, \mu(D) \rangle=\langle \tau_{\eta}\#\Tr, \mu_0(D) \rangle
\end{equation}
for all $[D]\in K_*(B\Gamma).$
See (\ref{eq:pairing.HP.K})  for
 $\tau_{\eta}\#\Tr$ or  \cite{Connes-Moscovici} on page $377$.
We say that a group $\Gamma$ is extendable, if any element $\eta\in H^*(\Gamma; \C)$ is  extendable.

Let $C^{\infty}(\Gamma)$ be a smooth algebra.
Recall that the inclusion $i:\C\Gamma\rightarrow C^{\infty}(\Gamma)$ induces a map
$i^*: HC^{*}(C^{\infty}(\Gamma))\rightarrow HC^{*}(\C\Gamma).$
Let us introduce a class of discrete groups:

\begin{definition}\label{admissible}
A group $\Gamma$ is \emph{admissible}, if
it admits a smooth algebra $C^{\infty}(\Gamma)$ such that there exists a map $\nu: H^*(\Gamma, \C)\rightarrow HC^*(C^{\infty}(\Gamma))$ such that $i^*\nu(\eta)=\tau_{\eta}$ for all $\eta\in H^*(\Gamma, \C)$.
\end{definition}

\begin{example}\label{hyp-ab}

$(1)$
$\Gamma$ is admissible, if it satisfies both  polynomial cohomology and rapid decay property in the sense of \cite{Connes-Moscovici}.
It is known that a finitely generated hyperbolic group with respect to some word metric satisfies the two properties.

$(2) $
$\Gamma$ is admissible, if it is isomorphic to
a free abelian group or a torsion free hyperbolic group.
In fact, a canonical smooth algebra $C^{\infty}(\Gamma)$  exists
for both classes of groups.
By Lemma~\ref{ab,hyp} $\tau$ is an isomorphism,
 i.e., $H^*(\Gamma; \C)\simeq HP^*(\C\Gamma)$.
 Thus, the claim follows because of the existence of $\nu': HP^*(\C\Gamma)\rightarrow HP^*(C^{\infty}(\Gamma))$ such that $i^*\nu'=1$ which is a result of surjectivity of $i^*$. Note that
 $HP^*(\C\Gamma)\simeq HP^*(C^{\infty}(\Gamma))$ for free abelian group
 and that $i^*: HP^*(C^{\infty}(\Gamma))\rightarrow HP^*(\C\Gamma)$
is  surjective  when $\Gamma$ is hyperbolic~\cite{Gromov}.
  \end{example}

\begin{remark}
\label{rem:pairing}
$(1)$ If $\Gamma$ is admissible, the map $\nu$ gives rise to a well defined pairing
\[
H^*(\Gamma ; \C) \times HP_*(C^{\infty}(\Gamma))\rightarrow\C.
\]

$(2)$ If moreover $HP^*(C^{\infty}(\Gamma))$ is finite dimensional, then there is a well defined pairing
\[
H^*(\Gamma ; \C)  \otimes H^*(\Gamma ; \C)
\times HP^*(C^{\infty}(\Gamma) \otimes C^{\infty}(\Gamma))\rightarrow \C.
\]
See the proof of Lemma \ref{lem:suff.pairing} for K\"unneth formula for the periodic cyclic cohomology.
\end{remark}

Recall that the Connes-Chern character~(\ref{eq:CC}) determines a dual Connes-Chern character
\begin{equation}
\label{eq:dualCh}
\Ch^t: HP^*(C^{\infty}(\Gamma))\rightarrow \Hom(K_*(C^*_r(\Gamma)), \C) \qquad a \mapsto(x\mapsto\langle\Ch(x), a\rangle).
\end{equation}

\begin{lemma}
\label{admext}
If $\Gamma$ is admissible, then $\Gamma$ is extendable.
\end{lemma}

\begin{proof}
Suppose $\Gamma$ is admissible, i.e., $i^*\nu(\eta)=\tau_{\eta}.$
Similar to (\ref{eq:dualCh}), there is a dual Connes-Chern character $\Ch^t: HC^*(\C\Gamma)\rightarrow\Hom(K_*(\C\Gamma\otimes\mathcal{R}), \C)$ (see~\cite{Connes-Moscovici}).
Denoting $\tilde\tau_{\eta}:=\Ch^t(\nu(\eta))\in \Hom(K_*(C^*_r(\Gamma)), \C),$
we have
\[
j^*(\tilde\tau_{\eta})=j^*(\Ch^t(\nu(\eta)))=\Ch^t(i^*(\nu(\eta)))=\Ch^t\tau_{\eta}=\tau_{\eta}\#\Tr.
\]
The second equality follows from the functoriality of $\Ch^t.$
\end{proof}

Suppose the classifying space of $\Gamma$ is
homotopy equivalent to a finite CW complex, and
consider the dualized Baum-Connes assembly map
\[
\mu^t: \Hom(K_*(C^*_r(\Gamma)), \C)\rightarrow K^*(B\Gamma)
\]
 given by
\[
\langle\mu^t(\alpha), \beta\rangle:=\langle\alpha, \mu(\beta)\rangle \qquad
\beta\in K_*(B\Gamma).
\]
This map is directly constructed in~\cite{Miscenko}. See also~\cite{CGM}.
Together with the dual Connes-Chern character $\Ch^t$ and ordinary Chern character $\ch$ we obtain a canonical map
\begin{equation}
\label{eq:can}
\begin{CD}
HP^*(C^{\infty}(\Gamma))@>\Ch^t>>\Hom(K_*(C^*_r(\Gamma)),\C) @>{\mu^t}  >> K^*(B\Gamma) @>{\ch} >> H^*(\Gamma; \C).
\end{CD}
\end{equation}
Here, the last map, the Chern character $\ch$  for $K$-homology, is defined so that the pairing of $K$-theory and $K$-homology is compatible with the pairing of cohomology and homology after taking Chern character (see~\cite{ConnesNDG}).

\begin{proposition}
\label{prop:admsuj}
If $\Gamma$ is admissible, then the canonical map $HP^*(C^{\infty}(\Gamma))\rightarrow H^*(\Gamma, \C)$ given by (\ref{eq:can}) is surjective, mapping $\nu(\eta)$ to $\eta$.
\end{proposition}

\begin{proof}
If $\Gamma$ is admissible, then $\nu(\eta)$ exists for all $\eta\in H^*(\Gamma, \C)$. Then we need only to show that
$\ch\circ\mu^t\circ\Ch^t(\nu(\eta))=\eta.$
Every element in $K_*(B\Gamma)$ is represented by a $f: M \to B\Gamma$ with $E \to B\Gamma$ and a $\mathrm{spin}^c$ Dirac operator $D$ on $M$. We denote this element by $f_*[D_E].$
It is thus sufficient to show
\[
\langle\ch\circ\mu^t\circ\Ch^t(\nu(\eta)), \ch(f_*[D_E])\rangle=\langle\eta, \ch(f_*[D_E])\rangle.
\]
From the proof of Lemma~\ref{admext}, one may choose $\tau_{\eta}=\Ch^t(\nu(\eta))$ so that $j^*\tilde\tau_{\eta}=\tau_{\eta}\#\Tr$.
Thus
\[
\langle\ch\circ\mu^t\circ\Ch^t(\nu(\eta)), \ch(f_*[D_E])\rangle=\langle\ch(\mu^t(\tau_{\eta})), \ch(f_*[D_E])\rangle.
\]
By definition of $\ch$ and $\mu^t$ and then the Connes-Moscovici's index theorem, we have
\[
\langle\ch(\mu^t(\tau_{\eta})), \ch(f_*[D_E])\rangle=\langle\mu^t(\tau_{\eta}), f_*[D_E]\rangle=\langle\tau_{\eta}, \mu(f_*[D_E])\rangle=\int_M\hat{A}(M)\ch(E)\wedge f^*\eta.
\]
On the other hand by definition of $\ch$, $\ch(D)$ is the de Rham current $f_*( \ \int_M\hat{A}(M)\wedge \ \ )$
(see  Theorem \ref{thm:CM}), and then
\[
\langle\eta, \ch(f_*[D_E])\rangle=\langle f^*\eta, \ch(D_E)\rangle=\int_M\hat{A}(M)\ch(E)\wedge f^*\eta.
\]
The proposition is then proved.
\end{proof}


\begin{corollary}
\label{cor:equ.con}
Suppose the classifying space of $\Gamma$ is
homotopy equivalent to a finite CW complex.
Then $\Gamma$ being extendable implies that the assembly map
$\mu:K_*(B\Gamma)\rightarrow K_*(C^*_r(\Gamma))$ is rationally injective.
\end{corollary}

\begin{proof}
From the previous proposition, the composition $\mu^t\circ\ch: \Hom(K_*(C^*_r(\Gamma)), \C)\rightarrow H^*(\Gamma, \C)$ maps $\Ch^t(\nu(\eta))$ to $\eta$, hence surjective. The corollary then follows observing that rational injectivity of $\mu$ is equivalent to rational surjectivity of $\mu^t$
and that the Chern character $K^*(B\Gamma)\rightarrow H^*(\Gamma; \C)$ is
rationally isomorphic.
\end{proof}

\begin{remark}
Either extendability or admissibility of $\Gamma$ implies the Novikov conjecture for the group.
The rational injectivity of the assembly map $\mu$ is called  the {\em strong Novikov conjecture}, which is a sufficient condition to induce both Novikov conjecture (homotopy invariance of higher signatures) and
Gromov-Lawson conjecture, where the latter claims that no positive scalar curvature metric
can be equipped with on a compact $K(\Gamma,1)$ manifold.
\end{remark}

\vspace{3mm}

The following technical Lemma gives a sufficient condition that the canonical map (\ref{eq:can}) (with tensors) is an isomorphism, which is used in Definition~\ref{def:twD}.

\begin{lemma}
\label{lem:suff.pairing}
Suppose that  the classifying space of $\Gamma$ is
homotopy equivalent to a finite CW complex,
 $\Gamma$ is admissible, and $HP^*(C^{\infty}(\Gamma))$ is finite dimensional.

Then there is a surjective map
\[
HP^{l}(C^{\infty}(\Gamma)\otimes C^{\infty}(\Gamma)\otimes C^{\infty}(B))\rightarrow \oplus_{i+j+k=l(2)}H^i(\Gamma)\otimes H^j(\Gamma)\otimes H_k(B).
\]
If  moreover $\Gamma$ satisfies the Baum-Connes conjecture and Connes-Chern character (\ref{eq:CC}) is an isomorphism, then the map is an isomorphism.
\end{lemma}
Here, the tensor product on the left hand side is the projective tensor product of Frech\'et algebras (see \cite{Lott1} for example). When $\Gamma$ is abelian, $C^{\infty}(\Gamma)$ and $C^{\infty}(B)$ are nuclear, the tensor product is uniquely defined.

For the reduced group $C^*$-algebras, we will use the minimal tensor product $C_r^*(\Gamma) \otimes_{min} C_r^*(\Gamma)$
and just denote as $C_r^*(\Gamma) \otimes C_r^*(\Gamma)$.

\vspace{3mm}






\begin{example}
Free abelian groups satisfies the assumptions of Lemma~\ref{lem:suff.pairing}.
\end{example}



To conclude the subsection, we present a proof of Lemma~\ref{lem:suff.pairing}.

\begin{proof}[Proof of Lemma~\ref{lem:suff.pairing}]
First of all, we have an isomorphism
\begin{multline}
\label{eq:ku}
HP^{l}(C^{\infty}(\Gamma)\otimes C^{\infty}(\Gamma)\otimes C^{\infty}(B))\simeq\\
 \oplus_{i+j+k=l(2)}HP^{i}(C^{\infty}(\Gamma))\otimes HP^{j}(C^{\infty}(\Gamma))\otimes HP^{k}(C^{\infty}(B))
\end{multline}
following from the K\"unneth theorem.
In fact, the cup product~(\ref{eq:cup.product}) on the cyclic cohomolology induces the one on periodic cyclic cohomology
\begin{equation*}
\oplus_{i+j=k} HP^i(\cA_1)\otimes HP^j(\cA_2)\rightarrow HP^{k}(\cA_1\otimes\cA_2)
\end{equation*}
It is known to be isomorphic  when one of $\cA_i$ is countably generated and $HP^{\bullet}(\cA_i)$ is finite dimensional.

The algebra $C^{\infty}(\Gamma)$
is always countably generated when $\Gamma$ is a countable discrete group.
$C^{\infty}(B)$  is also countably generated.
 $HP^{\bullet}(C^{\infty}(B))$ is finite dimensional when $B$  is compact
 (see subsection $4.1$).
 Together with the assumption that $HP^*(C^{\infty}(\Gamma))$ is finite dimensional, we obtain the isomorphism (\ref{eq:ku}).

In view of $HP^*(C^{\infty}(B))\cong H_*(B)$ (see Example~\ref{ex:cyclic.coho}), we only need a surjective map from $HP^*(C^{\infty}(\Gamma))$ to $H^*(\Gamma; \C)$, guaranteed by Lemma~\ref{prop:admsuj}: the composition of maps
\begin{equation}
\label{eq:comp}
HP^*(C^{\infty}(\Gamma))\xrightarrow{\Ch^t} \Hom(K_*(C^*_r(\Gamma)), \C)\xrightarrow{\mu^t} K^*(B\Gamma)\xrightarrow{\ch} H^*(\Gamma; \C)
\end{equation}
is a surjection.

Together with (\ref{eq:ku}) there exist a surjective map as required.

If $\Gamma$ satisfies the Baum-Connes conjecture, then $\mu$ is rationally isomorphic
and so is $\mu^t$. The first and third maps in (\ref{eq:comp})
are both  isomorphic over  $\C$. Hence, the above surjective map is actually
an isomorphism.
\end{proof}

 \subsection{Higher indices}\label{higher-index}
Let $\Gamma$ be a discrete group where $B\Gamma$ is its classifying space. We shall review the cohomological formula of the higher index of a $\Gamma$-invariant twisted Dirac operator by Connes and Moscovici~\cite{Connes-Moscovici}.

 Let $M$ be a closed Riemannian manifold and $\phi: M\rightarrow B\Gamma$ a continuous map.
 Let $\widetilde M^{\phi} \to M$ be the $\Gamma$-principal bundle given by the pull-back $\phi^*\pi$ of the universal covering $\pi: E\Gamma\rightarrow B\Gamma$.
Denote by $D$ an elliptic operator over $E \to M$,
and by $\widetilde D^{\phi}$ its lift    over $\widetilde M^{\phi}$  as a $\Gamma$-invariant elliptic operator.

Let $C^*_r(\Gamma)$ be the reduced group $C^*$-algebra which is given by the
operator norm completion of the group ring $\C\Gamma$ by the regular representation on $l^2(\Gamma)$.
The quotient of $\widetilde M^{\phi}\times C^*_r(\Gamma)$ by the diagonal action of $\Gamma$ gives  a flat
$C^*_r(\Gamma)$ bundle $\gamma$ on $M$.
Denote by $D_{\Gamma}^{\phi}$ the twisted  operator obtained by twisting $D$
with  the flat bundle $\gamma$.
Then $D_{\Gamma}^{\phi}$ is a $C^*_r(\Gamma)$-Fredholm operator over $E \otimes \gamma$.

The {\em higher index} of $\widetilde D^{\phi}$ is defined as  the index $\Ind D^{\phi}_{\Gamma}$ of the $C^*_r(\Gamma)$-Fredholm operator $D_{\Gamma}^{\phi}$
\begin{equation*}
\label{eq:AFredholm.index}
 \Ind\widetilde D^{\phi}:=
 [ \text{Ker} D_{\Gamma}^{\phi} ] -[ \text{Coker} D_{\Gamma}^{\phi} ]\in K_{0}(C^*_r(\Gamma))
\end{equation*}
 where Ker $D_{\Gamma}^{\phi}$ and Coker $D_{\Gamma}^{\phi}$ are both
 regarded as  finitely generated projective $C^*_r(\Gamma)$-modules after compact perturbations.
See~\cite{Kasparov81}.

By choosing a suitable parametrix of $\tilde D^{\phi}$, Connes and Moscovici~\cite{Connes-Moscovici}
verified  that the higher index of $\widetilde D^{\phi}$ comes from an element  in the $K$-theory of a smaller algebra
\begin{equation*}
\label{eq:higher.index'}
\Ind\widetilde D^{\phi}\in K_0(\C\Gamma\otimes \cR)
\end{equation*}
 where $\mathcal R$ is the algebra of smoothing operators on $M.$
Recall that $\mathcal R$ is dense in the $C^*$-algebra of compact operators $\cK$.
The  higher index coincides with the image of this element  under composition of the maps
\[
K_0(\C\Gamma\otimes \cR)\rightarrow K_0(C^*_r(\Gamma)\otimes\cK)\simeq K_0(C^*_r(\Gamma))
\]
induced by the inclusion
$j:\C\Gamma\otimes \cR\rightarrow C^*_r(\Gamma)\otimes\cK$.
That is, $j_*(\Ind\tilde D^{\phi})=\Ind\tilde D^{\phi}=\Ind D^{\phi}_{\Gamma}.$
See Lemma~6.1 of~\cite{Connes-Moscovici}.
So we shall also call these elements the higher index.

\begin{remark}
In Section~\ref{sec:twD-inv.nc}, we have to use the $C^*_r(\Gamma)$-Fredhom index of $D_{\Gamma}^{\phi}$ in $K_0(C^*_r(\Gamma))$ given by the higher index of $\widetilde D^{\phi}$, rather than
the element in $K_0(\C\Gamma\otimes \cR)$.
Therefore, as we shall see below, $\Gamma$ being admissible is essential in deriving a local index theorem for $D_{\Gamma}^{\phi}.$
\end{remark}

Let us state the cohomological formula on the higher index of  a  twisted Dirac operator.
Denote the element by $\tau_{\eta} \in HC^n(\C\Gamma)$,
which corresponds to
an element $\eta \in H^n(\Gamma, \C)$,
through the map $\tau$ in (\ref{eq:group.cyclic.coho}).
  By composition with the operator trace it gives a cyclic cocycle
  $$\tau_{\eta}\#\Tr\in HC^n(\C\Gamma\otimes \cR).$$

\begin{theorem}[\cite{Connes-Moscovici}, see also ~\cite{Lott}]
Let $\phi: M\rightarrow B\Gamma$ be a continuous map.
Let $D_E$ be a twisted Dirac operator over $M$ and $\widetilde D_E^{\phi}$ be its lift to the $\Gamma$-principal bundle $\widetilde M^{\phi}\rightarrow M.$

Then the image of the higher index $\Ind\widetilde D_E^{\phi}\in K_0(\C\Gamma\otimes \cR) $
 by the Connes-Chern character admits  the cohomological formula
\begin{equation*}
\label{eq:coho.formula}
\langle \Ch(\Ind\widetilde D_E^{\phi}), \tau_{\eta}\#\Tr\rangle=c\int_M\hat A(M)\wedge\ch(E)\wedge\phi^*\eta
\end{equation*}
which  is an element in $\Hom(HP^{0}(\C\Gamma\otimes \cR), \C)$,
where $c$ is some normalizing constant.
\end{theorem}
In particular the first Chern class $c_1(\Ind\widetilde D_E^{\phi})$
of the higher index
is described  as a homomorphism for $\eta\in H^2(\Gamma, \C)$
\[
HC^2(\C\Gamma\otimes \cR)\rightarrow \C,
\qquad \tau_{\eta}\#\Tr\mapsto c\int_M \hat A(M)\wedge\ch(E)\wedge\phi^*\eta.
\]

From analytic view point,
 it is important to find the cohomological formula of the Connes-Chern character of the higher indices $\Ind D^{\phi}_{\Gamma}$
 of $C^*_r(\Gamma)$-Fredholm operators
 in $K_0(C^*_r(\Gamma)).$
For this,
 the cyclic cocycle $\tau_{\eta}\in HC^{\bullet}(\C\Gamma)$ should be extended to a map $\tilde\tau_{\eta}: K_0(C^*_r(\Gamma))\rightarrow\C$ such that
$j^*\tilde\tau_{\eta}=\tau_{\eta}\#\Tr$ for every group cocycle $\eta\in H^{\bullet}(\Gamma; \C)$.
Connes-Moscovici devised a sufficient condition on $\Gamma$ for the extendability, namely,
polynomial cohomology and rapid decay (Proposition~6.3 of~\cite{Connes-Moscovici}), and in our setting this is implied by $\Gamma$ being admissible (Lemma~\ref{admext}).

\begin{proposition}[\cite{Connes-Moscovici}]
\label{thm:CM}
Suppose that
 $\Gamma$ is admissible so that
   $\tau_{\eta}\in HC^{\bullet}(\C\Gamma)$ gives rise to a map
   $\tilde\tau_{\eta}: K_0(C^*_r(\Gamma))\rightarrow\C$.

   Then   the cohomological formula holds:
\begin{equation}
\label{eq:cohomological.formula}
\langle \Ch(\Ind D_{E, \Gamma}^{\phi}), \widetilde\tau_{\eta}\rangle
=c\int_M\hat A(M)\wedge\ch(E)\wedge\phi^*\eta.
\end{equation}
\end{proposition}


\begin{proposition}
\label{prop:coho.formula}
Assume that $\Gamma$ is admissible.
Let $M$ be a closed manifold and $\phi: M\rightarrow B\Gamma$ a continuous map.
Let $D_E$ be a twisted Dirac operator over $M$,
 and $\widetilde M^{\phi}\rightarrow M$ be the principal $\Gamma$-bundle pulled back from $E\Gamma\rightarrow B\Gamma$ by $\phi$.
 Let $D_{E,\Gamma}^{\phi}$ be the operator $D_{E}$ twisted by the bundle
 $\gamma = \tilde M^{\phi}\times_{\Gamma}C^*_r(\Gamma).$

Then  the Connes-Chern character of the  $C^*_r(\Gamma)$-Fredholm index of $D_{E, \Gamma}^{\phi}$ in $K_0(C^*_r(\Gamma))$
 is completely determined by the cohomological formula~(\ref{eq:cohomological.formula}).
\end{proposition}

\begin{proof}
For every $y\in \Hom(K_*(C^*_r(\Gamma)), \C)$,
let  us put $\eta:=\ch(\mu^t(y))\in H^*(\Gamma; \C)$.
Then the equality
\begin{equation}
\label{eq:m}
\langle \mu(D), \tilde\tau_{\eta}\rangle=\langle\mu(D), y\rangle
 \end{equation}
 holds,
because both sides are equal to $\langle\ch(D), \eta\rangle$ by the proof of Lemma~\ref{lem:com}.
Note that the equality
$\eta = \ch(\mu^t(\tilde{\tau}_{\eta}))$ also holds.

By definition, the equality holds:
\[
j_*\Ind\widetilde D_E^{\phi}=\Ind D_{E, \Gamma}^{\phi}.
\]
Then we have
\begin{align*}
\langle \Ind D^{\phi}_{E, \Gamma}, y\rangle
=&\langle\Ind D^{\phi}_{E, \Gamma}, \tilde\tau_{\eta}\rangle\\
=&\langle \Ind \widetilde D^{\phi}_{E}, \tau_{\eta}\#\Tr\rangle \\
=&c\int_M\hat A(M)\wedge\Ch(E)\wedge\phi^*[\ch(\mu^t(y))].
\end{align*}
The first equality follows from (\ref{eq:m}); the second  from the $\Gamma$ being admissible; the third is Connes-Moscovici's index formula.
Therefore, $\Ch(\Ind D^{\phi}_{E, \Gamma})$ is completely determined by the cohomological formula.
\end{proof}

\section{Twisted Donaldson's Invariant for non commutative case}
\label{sec:twD-inv.nc}
Throughout section $5$, we always assume that
 all four manifolds $Y$ or $X$
 are compact and spin with $b^+ >1$, without mention.

Let
$E\rightarrow X$ be
an $SU(2)$ vector bundle and $\Gamma$ a discrete group.
Fix a homomorphism ${\bf f}: \pi_1(X)=\pi_1(\hat X)\rightarrow\Gamma$ which induces a continuous map
\[
\hat f: \hat X\rightarrow B\Gamma.
\]

In this section, we will introduce the twisted $\mu$-map
\[
\mu^{\tw}_{\bf f}:  \Omega^{Spin}_{*}(X) \otimes \Omega^{Spin}_{*} (B\Gamma)
    \rightarrow
HP_0(C^{\infty}(\Gamma)\otimes C^{\infty}(\Gamma)\otimes C^{\infty}( B))
\]
where
$\Omega^{Spin}_{*}(B\Gamma)$ is the spin bordism group, and
$B$ runs over compact submanifolds in ${\frak B}^*(\hat E)$.
Let $A^{\pi}(X)$ be as in  (\ref{eq A pi}).  Then we construct a twisted Donaldson's invariant
\[
\Psi^{\tw, (r)}_{X, {\bf f}}:
    A^{\pi} (X) \times
 (\Omega^{Spin}(X) \otimes\Omega^{Spin}(B\Gamma) )
   \rightarrow
   HP_*(C^{\infty}(\Gamma)\otimes C^{\infty}(\Gamma))
\]
for $r=1,2$
when $\Gamma$ is not necessarily abelian, in terms of  non commutative geometry.

\begin{remark}
Our formulation of the twisted Donaldson invariant requires
existence of $C^{\infty}(\Gamma)$ only, however
we need to assume that $\Gamma$ is admissible,
when we compute the index formula
on the twisted $\mu$ map, which is necessary to check well definedness
of the twisted $\mu$ map.
   \end{remark}
   \vspace{2mm}

The isomorphism $H^*(\Gamma ; \C) \cong HP^*(C^{\infty}(\Gamma))$ holds,
when $\Gamma$ satisfies the Baum-Connes conjecture with isomorphic Chern character (see lemma~\ref{lem:suff.pairing}),
 in particular, when $\Gamma \cong \Z^m$ is free abelian.
The above  map
$\Psi^{\tw, (r)}_{X, {\bf f}}$
extends the  twisted Donaldson map
  $(\Psi^{\tw, (r)}_{X, {\bf f}})_{ab}$
 in section $3$ as

$$
\begin{CD}
   A(X)\otimes
H_*(X;\Z)\otimes H_*(B\Gamma; \Z) @> (\Psi^{\tw, (r)}_{X, {\bf f}})_{ab}
 >> H_*(\Gamma; \Q) \\
 @VV V  @VV   V \\
     A(X)\otimes
H_*(X;\Z)\otimes\Omega^{Spin}(B\Gamma)@> \Psi^{\tw, (r)}_{X, {\bf f}} >>
H_*(\Gamma; \C)\otimes  H_*(\Gamma; \C)
\end{CD}$$
where the right vertical arrow is given by the diagonal embedding.

\subsection{Parametrized higher indices}
Let us take an element
 $\alpha\in \Omega^{Spin}_{*}(X)$
 which is realized by  a map $i: M \rightarrow X$,
and another element
 $\beta \in\Omega^{Spin}_{*}(B\Gamma)$ by
 $j: N\rightarrow B\Gamma$.

 Assume the product spin manifold $M\times N$ has even dimension, and denote by
$S=S^+\oplus S^-$  the spinor bundle equipped with  the spin connection $\nabla^S$
on $M\times N$.
Then we have the  associated Dirac operator
$
D^{\pm}: \Gamma(S^{\pm})\rightarrow \Gamma(S^{\mp})$
over $M\times N$. Notice that $D$ is the operator on the external tensor product of the spinors
$S _M \boxtimes S_N$,
given by the equality
$D = D_M \boxtimes 1 + 1 \boxtimes D_N$.

Take an $SU(2)$-vector bundle $E$ on $X$ with $c_2(E)$ odd.   Recall that in order to avoid difficulty
 caused by reducible flat connections, we use the $U(2)$-vector bundle $\hat{E}$
 on the blow-up $\hat{X} = X \# \overline{\mathbb{CP}}^2$.    It satisfies
  $c_1(\hat{E}) = e$ and $c_2(\hat{E}) = c_2(E)$,
 where $e$ is the generator of  $H^2(\overline{\mathbb{CP}}^2;\Z)  \subset H^2(\hat{X};\Z)$.

Let $B$ be a compact submanifold in ${\frak B}^*(\hat E)$,
and take an element $[A]\in  B$. The restriction $A|_M$
 gives  the Dirac operator $D_A$ twisted with $D_M$  over $M$, and hence it extends to
 the Dirac operator
$ {\mathbb D}_A = D_A \boxtimes 1 + 1 \boxtimes D_N$ over $M \times N$.

Let us also restrict  the universal bundle
$\hat{\mathbb E} \rightarrow \hat{X} \times B$.
Then
  $ {\mathbb D}_A $ is regarded  as  the  Dirac operator on
  $(\hat{\mathbb E}|_{M\times\{A\}}\otimes S_M)\boxtimes S_N.$

  Denote the $C^*_r(\Gamma)$-bundle over $B\Gamma$ by
\[
\gamma:= E\Gamma\times_{\Gamma} C^*_r(\Gamma)
\]
 equipped with a flat connection $\nabla.$
Let $f: \hat{X}\rightarrow B\Gamma$ be the continuous map induced by ${\bf f}$.
Then we consider the pull-back bundle
\begin{equation}
\label{eq:pull.back.alg.bundle}
[\hat{\mathbb E}|_{M\times\{A\}}\otimes (f\circ i)^* \gamma] \boxtimes j^*\gamma \rightarrow M \times N.
\end{equation}
The set of $L^2$-sections of this bundle forms a $C^*_r(\Gamma)  \otimes C^*_r(\Gamma)$ module.

The restriction $A|_M$ together with the flat connection gives the  pull-back connection
\begin{equation}
\label{eq:connection.pullback}
(A|_{M}\otimes 1 + 1 \otimes  (f\circ i)^*\nabla )  \boxtimes  1 + 1 \boxtimes j^*\nabla.
\end{equation}

Twist the Dirac operator $D$  with the connection~(\ref{eq:connection.pullback})
over  $S\otimes([\mathbb E|_{M\times\{A\}}\otimes (f\circ i) \gamma] \boxtimes j^*\gamma)$,
and denote it by $\widetilde D_A$.
It is a $(C^*_r(\Gamma)\otimes C^*_r(\Gamma))$-Fredholm operator on $M\times N$ with the higher index
(see Section $4.3$)
\[
\Ind \widetilde D_A=[\ker\widetilde D_A^+]-[\ker\widetilde D_A^-]\in K_0(C^*_r(\Gamma)\otimes C^*_r(\Gamma)).
\]
\begin{remark}
In the notation of Section $4$, $\widetilde D_A$ should have been denoted  $D_{A, \Gamma\times\Gamma}^{\phi}$ where $\phi$ is the continuous map $(f\circ i, j): M\times N\rightarrow B(\Gamma\times\Gamma).$
\end{remark}

Now consider a family of Dirac operators
\[
\mathbb{D}:=\{\widetilde D_{A}\}_{A\in B}
\]
over $M\times N$ parametrised by $B$, acting on sections  of the $(C^*_r(\Gamma)\otimes C^*_r(\Gamma))$-bundle
\[
S\otimes(\hat{\mathbb E}|_{M\times B}\boxtimes (f\circ i)^* \gamma \boxtimes j^*\gamma)
 \rightarrow M \times {B} \times N.
\]
Recall Lemma $4.1$ and
 that the classical Atiyah-Singer  index for the family.
\[
\ind \{D_{A}\}_{A\in  B}=\{[\ker D^+_{A}]-[\ker D^-_{A}]\}_{A\in \frak B}\in K^0(B)
=K_0(C(B))
\]
takes its value in $K$ cohomology of $B$,
 determined by the Fredholm property of each $D_{A}$.

Replacing $[\ker D^+_{A}]-[\ker D^-_{A}]$ by the higher index
 of $\widetilde D_{A}$, we obtain the family of higher indices
\[
\Ind_B \mathbb{D}=\{\Ind \widetilde D_{A}\}_{A\in B}\in K_0(C( B, C^*_r(\Gamma)\otimes C^*_r(\Gamma))).
\]
The relevent $C^*$ algebra in $K$-theory is
$$C( B, C^*_r(\Gamma)\otimes C^*_r(\Gamma))\cong C( B) \otimes C^*_r(\Gamma)\otimes C^*_r(\Gamma)$$
 which consists of the set of continuous maps with coefficient in $C^*_r(\Gamma)\otimes C^*_r(\Gamma)$.

$\Ind_B\mathbb{D}$ is compatible with restrictions in $K$-theory, i.e.,
\[
\Ind_{B'}\mathbb{D}|_B=\{\Ind\widetilde D_A\}_{A\in B}=\Ind_{B}\mathbb{D}
\]
where $B\subset B'.$

\begin{remark}
There is  a local formula of the higher index as below
\begin{lemma}
The higher index of $\widetilde D_A$ is  computed  by
\[
\Ind \widetilde D_A:=\begin{bmatrix}S_0^2 & S_0 (1+S_0)Q \\ S_1\widetilde D_A & 1-S_1^2 \end{bmatrix}-\begin{bmatrix}0 & 0 \\ 0 & 1\end{bmatrix}\in K_0(C^*_r(\Gamma)\otimes C^*_r(\Gamma)),
\]
where $Q$ is a parametrix of $\widetilde D_A$ where
\[
S_0:=1-Q\widetilde D_A , \qquad S_1:=1-\widetilde D_A Q
\]
are compact operators in the $C^*_r(\Gamma)\otimes C^*_r(\Gamma)$-module.
\end{lemma}
When $D_{A}$ depends continuously on $A$, then the operators
$\widetilde D_A$, $Q$ and $S_i$ also
depend
continuouly on $A$. So the family of higher indices gives an element in $K_0(C(B, C^*_r(\Gamma)\otimes C^*_r(\Gamma)))$.
\end{remark}

\subsection{Twisted  $\mu$-map}
The Connes-Chern character in Definition \ref{CCmap}  with Lemma \ref{smooth-alg}
give rise to the map
\begin{align*}
\Ch: K_0(C(B, C^*_r(\Gamma)\otimes C^*_r(\Gamma)))
& \cong K_0(C(B, C^{\infty}(\Gamma) \otimes C^{\infty}(\Gamma)))  \\
& \rightarrow HP_0(C^{\infty}(\Gamma) \otimes C^{\infty}(\Gamma)\otimes C^{\infty}( B))
 \end{align*}
where $HP_0(C^{\infty}(\Gamma) \otimes C^{\infty}(\Gamma)\otimes C^{\infty}(B))$ is identified with
$\Hom(HP^0(C^{\infty}(\Gamma) \otimes C^{\infty}(\Gamma)\otimes C^{\infty}( B)), \C)$.
Note that holomorphic calculus being closed under project tensor product is not verifed. However,
 assuming Kunneth formula for $K$-theory, for instance, the conditions for Lemma~\ref{lem:suff.pairing}, we always have the above isomorphism.

\begin{definition} Let us fix a smooth algebra $C^{\infty}(\Gamma)$.
The   twisted $\mu$-map
\begin{equation}
\label{eq:twisted.D-mu}
\mu^{\tw}_{\bf f}:  \Omega^{Spin}_{*}(X) \otimes
\Omega^{Spin}_{*}(B\Gamma)
\rightarrow \lim_{ \substack{\longleftarrow \\ B} }  HP_0(C^{\infty}(\Gamma)\otimes C^{\infty}(\Gamma)\otimes C^{\infty}( B))
\end{equation}
is defined for every compact $B$ in ${\frak B}^*(\hat{E})$
by the Connes-Chern character of the higher index of the family of twisted Dirac operators
 $\mathbb D=\{\widetilde D_A\}_{A\in B}$
\[
\mu^{\tw}_{\bf f}(\alpha, \beta):=\{\Ch(\Ind_B\mathbb D)\}_{B}.
\]
\end{definition}
In particular $\mu^{\tw}_{\bf f}(\alpha, \beta)$ is independent of choices of
$i: M \rightarrow X$ of the class $\alpha$ of $\Omega^{Spin}_{*}(X)$,
choice of $j: N\rightarrow B\Gamma$ of the class $\beta$ of $\Omega^{Spin}_{*}(B\Gamma)$ and the choice of spin structure on $M\times N.$

Assume that $\Gamma$ is admissible, and $HP^*(C^{\infty}(\Gamma))$ is finite dimensional.
In view of Remark~\ref{rem:pairing}, we have the pairing
\begin{multline}
\label{eq:pairing.lem}
\bigoplus_{i+j+k=l(2)} \
H^i(\Gamma)\otimes   H^j(\Gamma)\otimes H_k( B)  \\
  \times \ HP_l(C^{\infty}(\Gamma)\otimes C^{\infty}(\Gamma)\otimes C^{\infty}(B)) \ \rightarrow \ \C
 \end{multline}
given by
\[
([\eta], [\xi], [C], [E])\mapsto\langle\Tr\#(\tau_{\eta}\cup\tau_{\xi}\cup\tau_{C}), [E]\rangle.
\]
Later  we shall use this pairing to  evaluate the twisted $\mu$-map.

Recall that an element $\beta \in \Omega^{Spin}_{*}(B\Gamma)$ is given by a map
$j: N \to B \Gamma$.
Realize an element $\alpha   \in  \Omega^{Spin}_{*}(X)$ by
 $i: M \rightarrow X$.
Let
$\hat{\mathbb E} \rightarrow \hat{X} \times B$ be  the universal bundle.

The twisted $\mu$ map
admits the cohomological formula.

\begin{theorem}
\label{thm.coho.formula}
Assume that $\Gamma$ is admissible, and $HP^*(C^{\infty}(\Gamma))$ is finite dimensional.
Let us take elements
 $[\eta]\in H^i(\Gamma; \C)$, $[\xi
]\in H^j(\Gamma; \C)$ and $[C]\in H_k( B)$.

Under the pairing (\ref{eq:pairing.lem})
the twisted  $\mu$-map~(\ref{eq:twisted.D-mu}) is determined by the following cohomological formula
\begin{align*}
& \langle \
[\eta] \otimes [\xi] \otimes [C],
 \mu^{\tw}_{\bf f}(\alpha, \beta)
 \ \rangle \\
 &  = \ \langle
 \ (i \times \text{id })^*\ch(\hat{\mathbb E})\wedge f^*\eta  \ , \
 [M \times C]  \ \rangle
  \ \langle  \
 \hat A( N)\wedge j^*\xi \ ,  \ [N] \ \rangle \
\in \C
\end{align*}
where
$ f: \hat{X} \rightarrow B\Gamma$ is the continuous map depending on ${\bf f}$.
\end{theorem}

\begin{proof}
Observe that
 the wedge product is   compatible with  the cup product
  under $\tau$
\begin{align*}
H^{i+j}(B\Gamma\times B\Gamma)&\rightarrow HP^{i+j}(\C\Gamma\otimes \C\Gamma)\\
[p_1^*\eta\wedge p_2^*\xi]&\mapsto [\tau_{\eta}\cup\tau_{\xi}].
\end{align*}
Because $\Gamma$ is admissible, there exist $[\tilde\tau_{\eta}], [\tilde\tau_{\xi}]\in \Hom(K_*(C^{*}_r\Gamma),\C)$
 such that
 \[
 \langle\mu(D), \tilde\tau_{\eta}\cup\tilde\tau_{\xi} \rangle=\langle\mu_0(D), (\tau_{\eta}\cup\tau_{\xi})\#\Tr\rangle.
 \]
 In our context, $\mu(D)=\Ind\widetilde D_A$ and the right hand side is calculted by Connes-Moscovici's theorem. Thus,
we have
 the formula
\[
\langle \Ch(\Ind \widetilde D_A), \tilde\tau_{\eta}\cup\tilde\tau_{\xi}\rangle
=c\int_{M\times N}\hat A(M\times N)\wedge i^* \ch(\hat{\mathbb E}|_{M\times\{A\}})\wedge
f^*\eta\wedge j^*\xi\in\C
\]
for $A\in B$. Moreover, by proposition~\ref{prop:coho.formula} $\Ch(\Ind \widetilde D_A)$ is completely determined by this formula.
In our family case,
 the proof in~\cite{Connes-Moscovici} goes in  a parallel way with  coefficient, and hence we obtain
 the family version
\begin{multline}
\label{eq:1}
\langle \Ch(\Ind \mathbb D), \tilde\tau_{\eta}\cup\tilde\tau_{\xi}\rangle=\\
c\int_{M\times N}\hat A(M\times N)\wedge i^* \ch(\hat{\mathbb E}|_{M\times B})\wedge
f^*\eta\wedge j^*\xi\in H^*(B).
\end{multline}
More precisely
  $[\tilde\tau_{\eta}\cup\tilde\tau_{\xi}]\in \Hom(K_*(C^*_r(\Gamma)\otimes C^*_r(\Gamma)), \C)$ is
  paired with
\[
\Ind \mathbb D
 \in K_0(C( B))\otimes K_0(C^*_r(\Gamma)\otimes C^*_r(\Gamma))
\]
under the  K\"unneth formula on periodic cyclic homology~\cite{Emmanouil}.
Notice that this is the Atiyah-Singer's  index theorem for family,
 when the group cocycle is trivial.

Now the equalities hold since
$\hat{A}(M)=1$ by $\dim M \leq 3$
\begin{align*}
& \langle \
[\eta] \otimes [\xi] \otimes [C],
 \mu(\alpha, \beta)
 \ \rangle \\
 &  = \ \langle
 \ (i \times \text{id })^*\ch(\hat{\mathbb E})\wedge\hat A( N)\wedge f^*\eta\wedge j^*\xi \ , \
 [M\times N\times C] \ \rangle \\
 &  = \ \langle
 \ (i \times \text{id })^*\ch(\hat{\mathbb E})\wedge f^*\eta  \ , \
 [M \times C]  \ \rangle
  \ \langle  \
 \hat A( N)\wedge j^*\xi \ ,  \ [N] \ \rangle \
\in \C
\end{align*}
  under  the isomorphism
\[
H_k( B)\simeq HP^k(C^{\infty}( B)) \qquad [C]\mapsto [\tau_C].
\]
The formula~(\ref{eq:1}) is obtained.

Now similarly to the proof of proposition~\ref{prop:coho.formula} $\Ch(\Ind \mathbb D)$ is determined by this formula because $\Gamma$ is admissible.
\end{proof}

Let  us  choose  a closed submanifold
$B \subset \mathfrak M(\hat E)$ (see proposition \ref{prop V compact})
  \begin{equation*}
   B \equiv      V(\hat{E}, \hat{g}; \Sigma_0, M_1, \dots, M_{d})
        =\mathfrak M(\hat E)\cap V_{\Sigma_0}\cap V_{M_1}\cap\cdots\cap V_{M_d}.
\end{equation*}
Here, $\dim\frak M(E)=\sum_{k=1}^d(4-\dim M_k)+r$ with  $\dim B=r =1, 2.$
Note that the homology class of $B$ depends on $A^{\pi}(X)$, i.e., $[B]$ depends only on $\hat{E}$, the homotopy class $[M_j : S^1 \rightarrow X] \in \pi_1(X)$ with $\dim M_j = 1$   and the homology class $[M_j] \in H_*(X;\Z)$ with $\dim M_j \geq 2$.

\begin{definition}
\label{def:twD}
Let us fix  a smooth algebra $ C^{\infty}(\Gamma)$.
The twisted Donaldson invariant is given by a map
\begin{align*}
\Psi^{\tw, (r)}_{X, {\bf f}}:    &   A^{\pi}(X)
\times  ( \Omega^{Spin}(X)\otimes  \Omega^{Spin}(B\Gamma) )    \rightarrow
HP_*(C^{\infty}(\Gamma) \otimes C^{\infty}(\Gamma)), \\
& \qquad
( \ [M_1], \ldots, [M_d], \alpha, \beta \ )
\ \mapsto \ \frac{1}{2 \cdot 4^{d_0}} \langle \ \mu^{\tw}(\alpha, \beta), [B] \ \rangle.
\end{align*}
Here, $\dim B=r =1, 2$ and $d_0$ is the number of homology classes $[M_j]$ of degree $0$.
\end{definition}
In particular
this gives rise to the map
\[
\Psi^{\tw, (r)}_{X, {\bf f}}:
A^{\pi}(X)  \times  ( \Omega^{Spin}(X) \otimes   \Omega^{Spin}(B\Gamma))
 \rightarrow  H_*(\Gamma; \C) \otimes H_*(\Gamma; \C)
\]
when
$\Gamma$ satisfies the conditions of lemma~\ref{lem:suff.pairing}, for example, when $\Gamma$ is a free abelian group.

\begin{remark}
Recall that when $\Gamma$ admits a cohomological formula in the sense of theorem \ref{thm.coho.formula}, $\pi_1(X)$ can be reduced to $H_1(X;\Z)$ and $A^{\pi}(X)$ can be simplified to the form $A(X)$ in (\ref{eq A(X)}).
\end{remark}

\subsection{Relation to the commutative case}
Suppose
$\Gamma$ is free abelian.
So $\Gamma$ is  admissible.
 The space of irreducible representations $\hat\Gamma=\Pic(B\Gamma)$ is isomorphic to a torus,
  and Fourier transform gives rise to the isomorphism
\[
C^*_r(\Gamma) \ \cong \ C_0(\hat\Gamma) \ = \ \cup_{a \in\hat\Gamma} \ \C_{a}
\]
where $\C_{a}$ stands for the  one dimensional complex
 representation space of the character $a\in\hat\Gamma.$

\begin{lemma}
The
  higher index $\Ind \widetilde D_A\in K_0(C^*_r(\Gamma)\otimes C^*_r(\Gamma))$ coincides with the
  index for  family with the value  in $K^0(\hat\Gamma\times\hat\Gamma)$, passing through the above isomorphism.
 \end{lemma}

 \begin{proof}

$C^*_r(\Gamma)$ bundle $\gamma$ over $B\Gamma$ is a family of line bundles over
 $ B \Gamma \times \hat\Gamma$
\begin{equation}
\label{eq:decom.gamma}
\gamma:=E\Gamma\times_{\Gamma}C^*_r(\Gamma)=E\Gamma\times_{\Gamma}(\cup_{a\in\hat\Gamma}\C_{a})=
\bigsqcup_{a \in \hat\Gamma}  \  E\Gamma\times_{\Gamma}\C_{a}
\end{equation}
which is nothing but the universal line bundle ${\mathbb L}$.

Now we have the line bundle over $B\Gamma\times\hat \Gamma\times B\Gamma\times\hat \Gamma$
$$\gamma \boxtimes \gamma = \bigsqcup_{a, b \in \hat\Gamma} \
( \   E\Gamma\times_{\Gamma}\C_{a} \ ) \boxtimes
( \  E\Gamma\times_{\Gamma}\C_{b} \ ).
$$
Denote by $L_{a}$ the summand $E\Gamma\times_{\Gamma}\C_{a}$,  and
consider a bundle with connection
\[
\mathbb E|_{M\times\{A\}}\boxtimes_M
[(f\circ i)^*L_{a}\boxtimes j^*L_{b}]\rightarrow M\times N
\]
where the connection
is given by $A|_M$ on $E|_M$
and the induced flat connection on $(f\circ i)^*L_{\alpha}\boxtimes j^* L_{\beta}.$
Let $D_{A,a, b}$ be the Dirac operator $D$ on $M\times N$ twisted by this connection.
Then $\{D_{A, a, b}\}_{a,b \in \hat\Gamma}$ is a continuous family of Dirac operators whose family index
takes the value
\[
(\ind D_{A,a, b})_{a,b\in\hat\Gamma}\in K^0(\hat\Gamma\times\hat\Gamma)
\]
which corresponds the higher index:
\[
\Ind \widetilde D_A\in K_0(C^*_r(\Gamma)\otimes C^*_r(\Gamma)).
\]
\end{proof}

If we vary $A\in B$, then  the higher index
\[
\mathbb{D}=\{\widetilde D_A\}_{A\in  B}
\]
is also parametrised by $B\subset\frak M(\hat E)$.
Under the isomorphism
\[
K_0(C( B, C^*_r(\Gamma)\otimes C^*_r(\Gamma)))\cong K^0( B\times\hat\Gamma\times\hat\Gamma),
\]
the higher index for family   in  Section $5.1$
\[
\Ind \mathbb{D}=\{\Ind \widetilde D_{A}\}_{A\in B}\in K_0(C(B, C^*_r(\Gamma)\otimes C^*_r(\Gamma)))
\]
corresponds  to the  index for family
\[
(\ind D_{A, a,b})_{A\in B,a,b\in\hat\Gamma}\in K^0(B\times\hat\Gamma\times\hat\Gamma).
\]

The image of  the Chern character of the  index for family takes the value
\begin{align*}
\mu^{\tw}_{\bf f}(\alpha, \beta)=\Ch(\Ind\mathbb{D}) & \in H^{even}( B\times\hat\Gamma\times\hat\Gamma; \C) \\
& \cong
HP_0(C^{\infty}(B)\otimes C^{\infty}(\Gamma)\otimes C^{\infty}(\Gamma)).
\end{align*}
by lemma~\ref{lem:suff.pairing}.

\vspace{3mm}

So far we have constructed two versions of
the twisted $\mu$ maps
\begin{align*}
& \mu_{{\bf f}, ab}^{\tw}(\alpha, \beta) \in HP_0(C^{\infty}(B)\otimes C^{\infty}(\Gamma)), \\
& \mu_{\bf f}^{\tw}(\alpha, \beta) \in HP_0(C^{\infty}(B)\otimes C^{\infty}(\Gamma)\otimes C^{\infty}(\Gamma))
\end{align*}
and  the twisted Donaldson's maps
\begin{align*}
\Psi_{X, {\bf f}, ab}^{\tw, (r)}:    &   A(X)
\otimes  H_*(X;\Z) \otimes  \Omega^{Spin}(B\Gamma)\rightarrow
HP_*(C^{\infty}(\Gamma) ), \\
\Psi_{X, {\bf f}}^{\tw,(r)}:    &   A(X)
\otimes  H_*(X;\Z) \otimes  \Omega^{Spin}(B\Gamma)\rightarrow
HP_*(C^{\infty}(\Gamma) \otimes C^{\infty}(\Gamma)).
\end{align*}
where in both cases the formers are given in section $3$ and the latters are in section $5$
($ab$ stands for `abelian').

Let us describe their relations below.
Let
$$\Delta: \hat{\Gamma} \hookrightarrow \hat{\Gamma} \times \hat{\Gamma}$$
be the diagonal map.

\begin{proposition}
Suppose  $\pi_1(X) = \Gamma$ is  free abelian.
then they satisfy the following relations
\begin{align*}
\Delta^*(\Psi) = \Psi_{ab} , \quad
 \Delta^*(\mu(\alpha, \beta) ) =
\mu_{ab}(\alpha, \beta).
\end{align*}
\end{proposition}

\begin{proof}
Recall the commutative case in section $3$
\[
\mathbb E|_{M\times B}\boxtimes_X (f\circ i)^* \mathbb L \boxtimes_{\Pic}  j^*\mathbb L
\rightarrow M  \times N \times {B}
\]
where $\mathbb L\rightarrow B\Gamma\times\hat\Gamma$ is the universal
family of  flat line bundles
 defined in section~\ref{sec:family.det.line}.

Notice that the  fibered product
\[
\mathbb L\boxtimes_{\Pic} \mathbb L=\{(E\Gamma\times_{\Gamma}E\Gamma)
\times_{\Gamma}\C_{a}\}_{a\in\hat\Gamma}\rightarrow B\Gamma\times B\Gamma\times\hat\Gamma
\]
is obtained by restricting $\gamma\boxtimes\gamma$ over the diagonal in
$\hat\Gamma\times\hat\Gamma$
$$\gamma \boxtimes_{\Pic} \gamma=
\bigsqcup_{a\in\hat\Gamma} \  ( E\Gamma\times_{\Gamma}\C_{a}) \boxtimes
(E\Gamma\times_{\Gamma}\C_{a})\rightarrow B\Gamma\times B\Gamma\times \hat\Gamma.$$
Then it follows from naturality  that the equalities hold:
\begin{align*}
\Delta^*\Ch(\ind\{D_{A, a,b}\}_{A\in B,a,b\in\hat\Gamma})
& =\Ch( \Delta^*\ind\{D_{A, a,b}\}_{A\in B,a,b\in\hat\Gamma}) \\
& = \Ch(\ind\{D_{A, a}\}_{A\in B,a\in\hat\Gamma}).
\end{align*}
\end{proof}



\begin{remark}
Notice that the diagonal map induces a $*$-homomorphism
$$C_0(\hat{\Gamma}) \otimes  C_0(\hat{\Gamma}) \to C_0(\hat{\Gamma}) $$
but in general,
 there are no such homomorpshism
$C_r^*(\Gamma) \otimes C_r^*(\Gamma)  \to C_r^*(\Gamma) $
if $\Gamma$ is non commutative.
\end{remark}


\begin{thebibliography}{99}

\bibitem{AMR}
S. Akbulut, T. Mrowka and Y. Ruan, {\it Torsion classes and a universal constraint on Donaldson invariants for odd manifolds}, Trans. Amer. Math. Soc. 347 (1995), 63--76.

\bibitem{Atiyah-SingerIV}
M. F. Atiyah, I. M. Singer, {\it The index of elliptic operators. IV.} Ann. of Math. (2) 93 1971 119--138.

\bibitem{BF}
S. Bauer and M. Furuta,
{A stable cohomotopy refinement of Seiberg-Witten invariants. I},
Invent. Math. 155 (2004),  1--19.

\bibitem{Blackadar}
B. Blackadar, {\it $K$-theory for operator algebras.}
MSRI Publication Series 5, Springer- Verlag, New York Heidelberg Berlin Tokyo, (1986).

\bibitem{Burghelea}
D. Burghelea, {\it The cyclic homology of the group rings.}   Comment. Math. Helv. 60 (1985), 354--365.

\bibitem{Connes94}
A. Connes, {\it Noncommutative geometry.} Academic Press, Inc., San Diego, CA, 1994. xiv+661 pp.

\bibitem{ConnesNDG}
A. Connes, {\it Noncommutative differential geometry.} Inst. Hautes Etudes Sci. Publ. Math. 62 (1985), 257 -- 360.

\bibitem{ConnesYM}
A. Connes, {\it The action functional in Non-Commutative Geometry.}
Commun. Math.  Phys. 117 (1988), 673 -- 683.


\bibitem{Connes-Moscovici}
A. Connes, H. Moscovici, {\it Cyclic cohomology, the Novikov conjecture and hyperbolic groups.} Topology 29 (1990), no. 3, 345--388.

\bibitem{CGM}
A. Connes, M. Gromov and H. Moscovici, {\it Group cohomology with Lipschitz control and higher signatures.} Geom. Funct. Anal. 3 (1993), no. 1, 1--78.

\bibitem{Donaldson orientation}
S. K. Donaldson,
{\it The orientation of Yang-Mills moduli spaces and 4-manifold topology},
J. Differential Geom. 26 (1987), no. 3, 397--428.

\bibitem{Donaldson Polynomial}
S.K. Donaldson,
{\it Polynomial invariants for smooth four-manifolds},
Topology 29 (1990), no. 3, 257--315.

\bibitem{Donaldson Floer}
S. K. Donaldson,
Floer homology groups in Yang-Mills theory. With the assistance of M. Furuta and D. Kotschick. Cambridge Tracts in Mathematics, 147. Cambridge University Press, Cambridge, 2002.




\bibitem{DK}
S. K. Donaldson and P. B. Kronheimer, {\it The geometry of four-manifolds}, Oxford Mathematical Monographs. Oxford Science Publications. The Clarendon Press, Oxford University Press, New York, 1990. x+440 pp.

\bibitem{Emmanouil} I. Emmanoul, {\it The K\"unneth formula in periodic cyclic homology}, K-theory 10:197--214 (1996).


\bibitem{FM}
R. Friedman, and J. W. Morgan,
 Smooth four-manifolds and complex surfaces.
 Ergebnisse der Mathematik und ihrer Grenzgebiete 27. Springer-Verlag, Berlin, 1994.


\bibitem{Gromov}
M. Gromov, {\it Hyperbolic groups. In Essays in group theory}, S.M. Gersten, (Editor); MSRI Publications; 8. 75--264, Springer, New York (1987).

\bibitem{IL}
M. Ishida and  C. LeBrun,
{\it Curvature, connected sums, and Seiberg-Witten theory},
 Comm. Anal. Geom. {\bf 11}, 809--836 (2003)


\bibitem{IS}
M. Ishida and H. Sasahira,
{\it Stable cohomotopy Seiberg-Witten invariants of connected sums of four-manifolds with positive first Betti number II: Applications},
Communications in Analysis and Geometry  Volume 25, Number 2, 373--393, 2017.


\bibitem{Kasparov81} G.G. Kasparov, {\it The operator $K$-functor and extensions of $C^*$-algebras}, (English Transl.), Math. USSR Izv. 16 (1981), 513--672.

\bibitem{kato} T. Kato, {\it Covering monopole map and higher degree in non commutative geometry},
arXiv:1606.02402v1 (2016).

\bibitem{Kawauchi}
A. Kawauchi,
{\it Splitting a 4-manifold with infinite cyclic fundamental group},
Osaka J. Math. 31 (1994), no. 3, 489--495.


\bibitem{KM}  P. B. Kronheimer and T. S. Mrowka, {\it Embedded surfaces and the structure of Donaldson's polynomial invariants},  J. Differential Geom.  41,  no. 3 (1995), 573--734.

\bibitem{Lisca}
P. Lisca,
{\it On the Donaldson polynomials of elliptic surfaces},
 Math. Ann. 299 (1994), no. 4, 629 -- 639.


\bibitem{Loday}
J.-L. Loday, {\it Cyclic homology. }
Grundlehren der Mathematischen Wissenschaften [Fundamental Principles of Mathematical Sciences], 301. Springer-Verlag, Berlin, 1998. xx+513 pp.

\bibitem{Lott}
J. Lott, {\it Superconnections and higher index theory.}
Geom. Funct. Anal. 2 (1992), no. 4, 421--454.

\bibitem{Lott1}
J. Lott, Diffeomorphisms and noncommutative analytic torsion, Mem. Amer. Math. Soc., 141, no. 673, 1999, viii+56 pp.

\bibitem{Lusztig}
G. Lusztig, {\it Novikov's higher signature and families of elliptic operators},
J. Differential Geometry 7, 229-256 (1971).

\bibitem{Miscenko} A.S. Mischchenko, {\it Infinite dimensional representations of discrete groups and higher signatures}, Izv. Akad. S.S.S. R. Ser. Mat. 38 (1974) 81--106.

\bibitem{MM}
J. W. Morgan and T. S. Mrowka,
{\it A note on Donaldson's polynomial invariants},
Internat. Math. Res. Notices 1992, no. 10, 223--230.

\bibitem{MM2}
J. W. Morgan and T. S.  Mrowka,
{\it On the diffeomorphism classification of regular elliptic surfaces,}
Internat. Math. Res. Notices 1993, no. 6, 183 -- 184.


\bibitem{MO}
J. W. Morgan and  K. G. O'Grady,
Differential topology of complex surfaces. Elliptic surfaces with $pg=1$: smooth classification.
With the collaboration of Millie Niss.
Lecture Notes in Mathematics, 1545. Springer-Verlag, Berlin, 1993

\bibitem{Sasa}
H. Sasahira,
{\it Spin structures on Seiberg-Witten moduli spaces},
J. Math. Sci. Univ. Tokyo. {\bf 13},  347--363 (2006).

\bibitem{Spanier}
E. H. Spanier, Algebraic topology. Corrected reprint. Springer-Verlag, New York-Berlin, 1981.



\end{thebibliography}
\end{document}